\documentclass[12pt,letterpaper]{amsart}  
\usepackage{fullpage}
\usepackage{hyperref}
\usepackage{amssymb}
\usepackage{enumitem}
\usepackage{amsmath}
\usepackage{amsthm}
\usepackage{amssymb}
\usepackage{mathtools}
\usepackage{float}
\usepackage{tikz}
\usepackage{tikz-cd}
\usepackage[utf8]{inputenc}
\usepackage[english]{babel}
\usepackage{csquotes}
\usepackage{amsmath,calligra,mathrsfs}
\usepackage[capitalize,noabbrev]{cleveref}
\usepackage[backend=bibtex,maxbibnames=99,sorting=anyt]{biblatex}
\renewbibmacro{in:}{}
\bibliography{References}

\usetikzlibrary{matrix}

\tikzstyle{vertex}=[circle, draw, fill=blue, minimum size=3pt]

\newtheorem{theorem}{Theorem}[section]
\newtheorem*{theorem*}{Theorem}
\newtheorem{corollary}[theorem]{Corollary}
\newtheorem*{corollary*}{Corollary}

\newtheorem{proposition}[theorem]{Proposition}
\newtheorem*{proposition*}{Proposition}
\newtheorem{lemma}[theorem]{Lemma}

\theoremstyle{definition}
\newtheorem{definition}[theorem]{Definition} 
\newtheorem*{definition*}{Definition}
\newtheorem{example}[theorem]{Example}
\newtheorem*{example*}{Example}
\theoremstyle{remark}
\newtheorem{remark}[theorem]{Remark}

\numberwithin{equation}{section}

\newcommand{\Pow}[1]{2^{#1}}

\newcommand{\qdcl}{\cat{Cl_{A}}}
\newcommand{\sucl}{\cat{Cl_s}}
\newcommand{\suqdcl}{\cat{Cl_{sA}}}

\DeclareMathOperator{\colim}{colim}
\DeclareMathOperator{\St}{St}
\DeclareMathOperator{\mdist}{mdist}

\DeclareFieldFormat{postnote}{#1}
\DeclareFieldFormat{multipostnote}{#1}

\newcommand{\topsimp}[1]{|\Delta^{#1}|^{I_{\tau}}}
\newcommand{\indsimp}[1]{|\Delta^{#1}|^{J_{\top}}}
\newcommand{\disimp}[1]{|\Delta^{#1}|^{J_+}}
\newcommand{\gsimp}[1]{|\Delta^{#1}|^{J}}

\newcommand{\indcube}[1]{|\Box^{#1}|^{J_{\top}}}
\newcommand{\iindcube}[1]{|\Box^{#1}|^{(J_{\top},\boxplus)}}

\newcommand{\dicube}[1]{|\Box^{#1}|^{J_+}}
\newcommand{\idicube}[1]{|\Box^{#1}|^{(J_+,\boxplus)}}
\newcommand{\gcube}[1]{|\Box^{#1}|^{(J,\otimes)}}

\newcommand{\ttimesgcube}[1]{|\Box^{#1}|^{(J,\times)}}
\newcommand{\timesgcube}[1]{|\Box^{#1}|^{J}}

\newcommand{\gsimpchain}[2]{C^{J}_{#1}(#2)}

\newcommand{\indcubechain}[2]{C^{(J_{\top},\times)}_{#1}(#2)}

\newcommand{\dicubechain}[2]{C^{(J_+,\times)}_{#1}(#2)}

\newcommand{\gcubechain}[2]{C^{(J,\otimes)}_{#1}(#2)}
\newcommand{\gcubechainwhole}[2]{Q^{(J,\otimes)}_{#1}(#2)}

\newcommand{\gsimphom}[2]{H^{J}_{#1}(#2)}

\newcommand{\topcubehom}[2]{H^{(I_{\tau},\times)}_{#1}(#2)}

\newcommand{\itopcubehom}[2]{H^{(I_{\tau},\boxplus)}_{#1}(#2)}

\newcommand{\indcubehom}[2]{H^{(J_{\top},\times)}_{#1}(#2)}
\newcommand{\iindcubehom}[2]{H^{(J_{\top},\boxplus)}_{#1}(#2)}
\newcommand{\dicubehom}[2]{H^{(J_+,\times)}_{#1}(#2)}
\newcommand{\idicubehom}[2]{H^{(J_+,\boxplus)}_{#1}(#2)}

\newcommand{\gcubehom}[2]{H^{(J,\otimes)}_{#1}(#2)}
\newcommand{\rgcubehom}[2]{\tilde{H}^{(J,\otimes)}_{#1}(#2)}

\newcommand\new[1]{#1}

\newcommand{\cat}{\mathbf}
\newcommand{\abs}[1]{\lvert#1\vert}

\newcommand{\Id}{\mathbf{1}}
\newcommand{\isom}{\cong}
\newcommand{\incl}{\hookrightarrow}
\newcommand{\xto}{\xrightarrow}
\newcommand{\To}{\Rightarrow}
\newcommand{\R}{\mathbb{R}}
\newcommand{\Z}{\mathbb{Z}}
\newcommand{\N}{\mathbb{N}}
\newcommand{\eps}{\varepsilon}
\newcommand{\isomto}{\xrightarrow{\cong}}
\newcommand{\Cech}{\check{C}}
\newcommand{\transpose}{T}

\makeatletter
\@namedef{subjclassname@2020}{\textup{2020} Mathematics Subject Classification}
\makeatother

\DeclareMathOperator{\im}{im}
\DeclareMathOperator{\VR}{VR}
\DeclareMathOperator{\qd}{A}
\DeclareMathOperator{\dc}{dc}
\DeclareMathOperator{\tr}{tr}
\DeclareMathOperator{\cosk}{cosk}
\DeclareMathOperator{\dist}{dist}
\DeclareMathOperator{\Sub}{Sub}

\title{Homotopy, homology, and persistent homology using closure spaces}

\subjclass[2020]{Primary 55N31; Secondary 54A05, 05C20}  

\author{Peter Bubenik \and Nikola Mili\'cevi\'c}

\tikzset{%
    symbol/.style={%
        draw=none,
        every to/.append style={%
            edge node={node [sloped, allow upside down, auto=false]{$#1$}}}
    }
}

\begin{document}

\begin{abstract}
  We develop persistent homology in the setting of \new{filtrations of} (\v{C}ech) closure spaces. Examples of \new{filtrations of} closure spaces include metric spaces, weighted graphs, weighted directed graphs, and \new{filtrations of} topological spaces. We use various products and intervals for closure spaces to obtain six homotopy theories, six cubical singular homology theories, and three simplicial singular homology theories. Applied to \new{filtrations of} closure spaces, these homology theories produce persistence modules. We extend the definition of Gromov-Hausdorff distance \new{from metric spaces to filtrations of} closure spaces and use it to prove that \new{any persistence module obtained from a homotopy-invariant functor on closure spaces is stable}. 
\end{abstract}

\maketitle

\setcounter{tocdepth}{1}
\tableofcontents

\section*{Introduction} \label{sec:intro} 

A primary tool in applied algebraic topology is persistent homology~\cite{elz:tPaS,MR2121296,oudot:book}, which is a key component of topological data analysis~\cite{MR3839171,Rabadan:2019}. For data encoded as a metric space, consider the cover by balls of a fixed radius $r$ and its nerve, called the \new{(intrinsic)} \v{C}ech complex, or the clique complex of the $1$-skeleton of this nerve, called the Vietoris-Rips complex. Allowing $r$ to vary, we obtain a filtered simplicial complex. Applying homology with coefficients in a field produces a persistence module. Under mild hypotheses, this persistence module has a complete invariant called a barcode or persistence diagram~\cite{Crawley-Boevey:2015,cdsgo:book,ccbds,cdso:geometric,Schmahl:2020}.

This approach uses simplicial homology, or, if one uses filtered topological spaces instead of filtered simplicial complexes, singular homology. An alternative approach to applied algebraic topology has  developed homology theories and homotopy theories for metric spaces~\cite{barcelo2014discrete}, graphs~\cite{barcelo2001foundations,dochtermann2009hom}, and directed graphs~\cite{grigor2014homotopy,dochtermann2023homomorphism}.
Digital images, the object of study of digital topology, are a special case of graphs and  their homology and homotopy theories have been defined as well  \cite{Boxer:1999,Lupton:2022}.
However, there is a more general axiomatization of topological spaces, namely Eduard \v{C}ech's closure spaces~\cite{vcech1966topological} (\cref{def:closure_spaces}), that includes topological spaces, graphs, and directed graphs as full subcategories, which suggests that these approaches may be combined~\cite{rieser2021vcech}.
Closure spaces are also called pretopological spaces.
Closure spaces have been used in
shape recognition~\cite{emptoz1983modele,frelicot1998pretopological},
image analysis~\cite{lamure1987espaces,bonnevay2009pretopological},
supervised learning~\cite{frank1996pretopological,frelicot1998pretopology}  
and complex systems modeling~\cite{ahat2009pollution,largeron2002pretopological}.

\subsection*{Our contributions}

We use closure spaces to port methods of algebraic topology to settings used in applied topology.
The category of closure spaces and continuous maps has, as full subcategories, the categories of topological spaces and continuous maps, graphs and graph homomorphisms, and directed graphs and directed graph homomorphisms.
  We use interval objects, an idea from abstract homotopy theory, to develop homotopy theories and homology theories for closure spaces.
  For example, we give a systematic development of homotopy theories and homology theories for graphs and directed graphs.
  We introduce a new setting for applied algebraic topology, the category of \new{filtrations of} closure spaces and natural transformations.
  This category has, as full subcategories, the categories of metric spaces, \new{filtrations of} topological spaces, weighted graphs, and weighted directed graphs.
  Each of our homology theories for closure spaces produces a persistent homology of \new{filtrations of} closure spaces.
  We give functorial definitions of Vietoris-Rips and (intrinsic) \v{C}ech complexes for closure spaces, which together with simplicial homology also produce persistent homologies of \new{filtrations of} closure spaces.
  We prove that all of these persistent homologies are stable; these results include and extend many existing stability theorems.
  For example, we obtain stability of many variants of persistent homology for metric spaces and for weighted directed graphs.

A closure space $(X,c)$, consists of a set $X$ and a closure operation $c$, which sends subsets of $X$ to subsets of $X$ and satisfies certain axioms (\cref{def:closure_spaces}).
Let $\cat{Cl}$ denote the category of closure spaces and continuous maps (\cref{def:continuous_functions}).

\subsubsection*{Homotopy}
 
We define homotopies between maps of closure spaces using cylinders, which are given by taking a
product operation (\cref{def:product_operator})
with an interval (\cref{def:interval_object}).
For examples of intervals, we have (\cref{def:interval_objects}): $I_{\tau}$, the unit interval with the topological closure; $J_{\top}$, the set $\{0,1\}$ with the indiscrete closure, 
in which $c(\{0\}) = \{0,1\}$ and $c(\{1\}) = \{0,1\}$;
and the Sierpinski space $J_{+}$,
in which $c(\{0\}) = \{0,1\}$ and $c(\{1\}) = \{1\}$.
For examples of product operations, we have $\times$, the (categorical) product (\cref{def:product_closure}), and $\boxplus$, the inductive product (\cref{def:inductive_product_closure}).

\begin{definition*}[\cref{def:homotopy}]
  For each interval $J$ and product operation $\otimes$ and each pair of closure spaces $X$ and $Y$ there is a equivalence relation $\sim_{(J,\otimes)}$, which we call \emph{$(J,\otimes)$ homotopy}, on the set of continuous maps from $X$ to $Y$.
\end{definition*}

For closure spaces $X$ and $Y$ we have a partial order on pairs $(J,\otimes)$, where $J$ is an interval and $\otimes$ is a product operation, given by $(J,\otimes) \leq (J',\otimes')$ if for all $f,g:X \to Y$, $f \sim_{(J,\otimes)} g$
implies that
$f \sim_{(J',\otimes')} g$.
We consider numerous intervals and reduce to the following poset, which is independent of the choices of $X$ and $Y$.

\begin{theorem*}[\cref{cor:homotopy-poset,thm:homotopy-poset}]
  There is a poset of
  intervals and product operations with
  distinct non-trivial homotopy relations given by
  the following Hasse diagram.
  
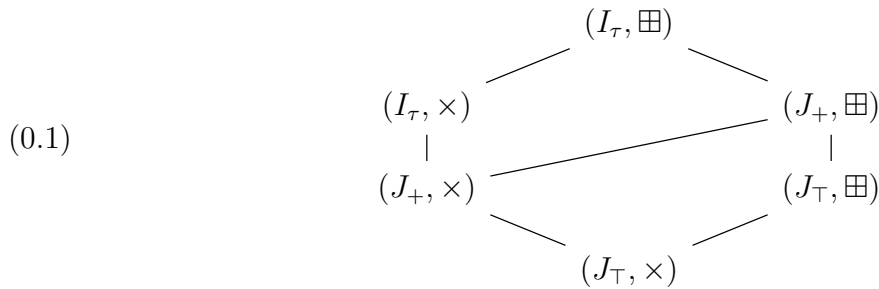
\begin{equation}  \label{cd:hasse}
    \begin{tikzcd}[every arrow/.append style={dash},row sep=2ex]
      & (I_{\tau},\boxplus) \ar[dr] \ar[dl] \\
      (I_{\tau},\times) \ar[d] & & (J_+,\boxplus) \ar[dll] \ar[d]\\
      (J_+,\times) \ar[dr] & & (J_{\top},\boxplus) \ar[dl]\\
      & (J_\top,\times)
    \end{tikzcd}
  \end{equation}
\end{theorem*}

\subsubsection*{Homology}

We study cubical and simplicial singular homology theories for closure spaces.

\begin{definition*}[\cref{def:cube,def:singular_cubes}]
  For each pair consisting of an interval $J$ and a closure operation $\otimes$ there is a corresponding cubical singular homology theory.
\end{definition*}

\begin{theorem*}[\cref{theorem:homotopy_invariance_of_homology}]
  Each of the these cubical homology theories is invariant with respect to the corresponding homotopy relation.
\end{theorem*}

\begin{definition*}[\cref{def:simplices,def:singular_simplices}]
  For each interval $J$ and the (categorical) product, $\times$, there is a corresponding simplicial singular homology theory.
\end{definition*}

\subsubsection*{Persistent homology}

We develop a new framework for applied algebraic topology using
\new{\emph{filtrations of closure spaces}, given by diagrams of monomorphisms of closure spaces (\cref{def:filtered-closure-space}).
An example of such a diagram is a \emph{filtered closure space}, which consists of a closure space $(X,c)$ and
a family of closure spaces $\{(X_a,c_a)\}_{a \in \R}$ such that 
for $a \leq b$, $X_a \subset X_b \subset X$ and for all $A \subset X$, $c_a(A) \subset c_b(A) \subset c(A)$.
Let $\cat{FCl}$ denote the category of \new{filtrations of} closure spaces and natural transformations,
and let $\cat{F^{\subset} Cl}$ denote the full subcategory of filtered closure spaces.}

Let $\cat{Met}$ denote the category of metric spaces and $1$-Lipschitz (i.e. non-expansive) maps.
A weighted graph is a graph in which each vertex and each edge has a real number called its weight, such that the weight of an edge is not less than the weight of its boundary vertices.
Let $\cat{wGph}$ denote the category of weighted simple graphs and non-weight-increasing graph homomorphisms.
Similarly, we have the category $\cat{wDiGph}$ of weighted simple directed graphs.

\begin{proposition*}[\cref{prop:embeddings-R}]
  We have the following full embeddings of categories
  \begin{equation*}
    \cat{Met} \incl \cat{wGph} \incl \cat{wDiGph} \incl \cat{FCl},
  \end{equation*}
  \new{where an embedding is a faithful functor that is injective on objects.}
\end{proposition*}





If we apply any of our cubical or simplicial singular homology theories elementwise to a \new{filtration of} closure spaces we obtain a persistence module (\cref{sec:pm}).
%
  \new{So,} for each of our cubical and simplicial singular homology theories and for each of the categories $\cat{Met}$,
  \new{$\cat{wGph}$},
  $\cat{w DiGph}$, 
  $\cat{FCl}$, we have a functor to 
  the category of persistence modules.

\subsubsection*{Stability}

For \new{filtrations of} closure spaces, we define a distance $d_{GH}$, which we call the Gromov-Hausdorff distance
(\cref{def:GH-FRCl}).

\begin{theorem*}[\cref{thm:gh-metric}]
  The Gromov-Hausdorff distance for \new{filtrations of} closure spaces is reflexive, symmetric, and satisfies the triangle inequality.
\end{theorem*}

This distance generalizes the usual Gromov-Hausdorff distance for metric spaces.

\begin{theorem*}[\cref{thm:gh}]
     There is an isometric embedding $(\cat{Met},d_{GH}) \incl (\cat{F Cl},d_{GH})$.
\end{theorem*}

\new{We prove the following stability theorem, which generalizes many of the stability theorems in the persistence literature.}

\new{
\begin{theorem*}[\cref{thm:stability}]
    Let $X$ and $Y$ be filtrations of closure spaces. 
    Let $H$ be a functor that is homotopy invariant for one of the homotopy theories in \eqref{cd:hasse}.
    Then \[d_I(HX,HY) \leq 2d_{GH}(X,Y),\]
    where $d_I$ denotes the interleaving distance (\cref{def:interleaving_distance}).
\end{theorem*}
}

\new{The results of this section are based on our definition of $\eps$-correspondence for filtered closure spaces (\cref{def:eps-multivalued-map}).
The stability theorem above is a consequence of the following sharper result.}

\new{
\begin{theorem*}[\cref{thm:stability-new}]
    Let $X$ and $Y$ be filtered closure spaces. 
    Let $H$ be a functor that is homotopy invariant for one of the homotopy theories in \eqref{cd:hasse}.
    If there is a $\eps$-correspondence between $X$ and $Y$ then
    $HX$ and $HY$ are $\eps$-interleaved.
\end{theorem*}}

\subsubsection*{Vietoris-Rips and \v{C}ech complexes}

Let $\cat{Simp}$ denote the category of (abstract) simplicial complexes and simplicial maps.
  We generalize the Vietoris-Rips complex and \new{(intrinsic)} \v{C}ech complex constructions to define functors $\VR, \Cech: \cat{Cl} \to \cat{Simp}$ (\cref{def:vr,def:cech}).
  We use \new{the closed}
  stars of simplices 
  to define a functor $\St: \cat{Simp} \to \cat{Cl}$ (\cref{def:star,def:G}).
  Let $\cat{Gph}$ denote the category of simple graphs and graph homomorphisms.
  
\begin{theorem*}[\cref{prop:adjunction-vr-g,theorem:cech_doesnt_have_adjoint,thm:adjunctions}]
  The Vietoris-Rips functor $\VR$ has left adjoint the star functor $\St$.
This adjunction factors through $\cat{Gph}$.
  \begin{equation*}
    \begin{tikzcd}
    \VR: \cat{Cl} \ar[r,shift right=1ex,"\bot"] &
      \cat{Gph} \ar[l,shift right=1ex] \ar[r,shift right=1ex,"\bot"] & \cat{Simp} : \St \ar[l,shift right=1ex]
    \end{tikzcd}
  \end{equation*}
  In contrast, the \v{C}ech functor, $\Cech$ does not have a left or right adjoint.
\end{theorem*}

If we apply the Vietoris-Rips or \v{C}ech functors elementwise to a filtered closure space, we obtain a filtered simplicial complex.
Applying simplicial homology we obtain a persistence module. The stability of this persistence module and its persistence diagram
\new{(\cref{thm:stability-rips} follow from the stability theorem above and the following result}.

\begin{theorem*}[\cref{proposition:homotopy_and_contiguity}]
The functors $\VR$ and $\Cech$ send one-step $(J_{\top},\times)$-homotopic maps to contiguous simplicial maps.
  Conversely, the functor $\St$ sends contiguous simplicial maps to one-step $(J_{\top},\times)$-homotopic maps.
\end{theorem*}

\subsubsection*{Application to graphs and directed graphs}

Recall that simple graphs and simple directed graphs are full subcategories of closure spaces. 
Simple (directed) graphs have the (di)graph product (\cref{def:regular_digraph_product}) and the cartesian product (\cref{def:digraph_product}).

\begin{proposition*}[\cref{prop:regular_digraph_product_is_product,prop:digraph_product_is_inductive_product}]
  Restricted to simple (directed) graphs, the product and the inductive product of closure spaces are the (di)graph product and the \new{(di)graph} cartesian product, respectively.
\end{proposition*}

Our homotopy theories and cubical singular homology theories are of interest in the special cases of graphs and directed graphs and some of these have been previously studied (\cref{lemma:restriction_of_homotopies,lemma:restricting_homology}).
As observed above, these homology theories produce persistent homology theories for weighted graphs and weighted directed graphs.

\subsection*{Related work}

Antonio Rieser \cite{rieser2021vcech}
and Demaria and Bogin \cite{demaria1984homotopy}
used the unit interval and the categorical product to define a homotopy theory for closure spaces. 
The latter also used this interval and product to define a simplicial singular homology theory for closure spaces.
Our work is particularly indebted to \cite{rieser2021vcech}.
 
Babson, Barcelo, de Longueville, Kramer, Laubenbacher, and Weaver used the simple graph with one edge as an interval together with the inductive/cartesian product on graphs to define the discrete homotopy theory of simple graphs, which they first called $A$-theory~
\cite{barcelo2001foundations,barcelo2005perspectives,babson2006homotopy}.
Barcelo, Capraro and White \cite{barcelo2014discrete}
used the same interval and product to define a cubical singular homology theory.
They also observed that as a special case, these homotopy and homology theories may be applied to metric spaces at a fixed scale.
Dochtermann~\cite{dochtermann2009hom} used the same interval and the categorical product to define a homotopy theory for simple graphs.
A special case of these homotopy theories, in which the underlying set is a finite subset of the lattice $\Z^n$, is studied in digital topology~\cite{Boxer:1999,Lupton:2022}.
Grigor'yan, Lin, Muranov, and Yau~\cite{grigor2014homotopy} used the simple directed graph with a single directed edge as an interval together with the inductive/cartesian product on directed graphs to define a homotopy theory of simple directed graphs. 
Dochtermann and Singh~\cite{dochtermann2023homomorphism} used the same interval with the categorical product to define a homotopy theory for simple directed graphs.

  In a companion paper~\cite{bubenik2021eilenberg},
  we use acyclic models to show that for an interval $J$ and the categorical product $\times$, corresponding cubical and simplicial singular homology theories agree.
  In fact, we show that their underlying chain complexes are chain homotopic.
    We also show that they satisfy a Mayer-Vietoris property and the excision property.
    Rieser~\cite{rieser2020vietoris} and Palacios~\cite{vela2019homology} have previously defined Vietoris-Rips homology and \v{C}ech homology, respectively for closure spaces.

Our results on stability with respect to Gromov-Hausdorff distance are indebted to
Chazal, de Silva, and Oudot~\cite{cdso:geometric}.
  Correspondences have also been defined and used by Segarra for finite edge-weighted digraphs~\cite{segarra2016metric} and by Chowdhury and M\'{e}moli for finite weighted digraphs~\cite{chowdhury2018functorial}.
  Turner has studied filtered simplicial complexes and persistence modules obtained from weighted digraphs~\cite{turner2019rips}.


\section{Closure spaces}
\label{section:background} 

In this section we provide background on Eduard \v{C}ech's closure spaces \cite{vcech1966topological}.
To start, we give three equivalent definitions of closure spaces using closures, interiors, and neighborhoods.
In each case, with an additional axiom we obtain topological spaces.
For any set $X$, let $\Pow{X}$ denote the collection of subsets of $X$.

\subsection{Closures}

We begin by defining a closure operator, which assigns subsets to subsets.

\begin{definition}
\label{def:closure_spaces}
Let $X$ be a set.
A function $c:\Pow{X}\to \Pow{X}$ is called a
  \emph{closure} or \emph{closure operator}
on $X$ if the following axioms are satisfied for all $A,B \subset X$:
\begin{enumerate}
\item (grounded) $c(\varnothing)=\varnothing$,
\item (extensive) $A\subset c(A)$,
\item \new{(additive)  $c(A \cup B) = c(A) \cup c(B)$.}
\end{enumerate}
\new{The ordered pair $(X,c)$ is called a (\emph{\v{C}ech}) \emph{closure space}.}
\new{Note that by additivity, for $A\subset B \subset X$, we have $c(A\cup B)=c(B)=c(A)\cup c(B)$, from which we get monotonicity, $c(A)\subset c(B)$.}

Subsets $A \subset X$ for which $c(A)=A$ are called \emph{closed}.
If in addition we have the following axiom for all $A \subset X$,
  \begin{enumerate}
      \setcounter{enumi}{3}
  \item (idempotent) $c(c(A)) = c(A)$,
  \end{enumerate}
  then we call $c$ a \emph{Kuratowski} closure \new{operator}.
\end{definition}

\begin{example}
\label{example:discrete_and_indiscrete_closures}
Let $X$ be a set.
The identity map $\mathbf{1}_{\Pow{X}}:\Pow{X}\to \Pow{X}$ is a closure operator on $X$, called the \emph{discrete} closure on $X$.
The closure operator defined by the map $A\mapsto X$ for $A\neq \varnothing$ and $\varnothing\mapsto\varnothing$ is called the \emph{indiscrete} closure on $X$.
These are both Kuratowski \new{closure operators}.
\end{example}

\begin{example} \label{example:closure_space_not_a_topological_space}
Consider $\R^n$  with the euclidean metric $d$. Let $r \ge 0$ and let $c_r$ be the closure on $\R^n$ defined by $c_r(A) = \{x\in \R^n\, |\, d(x, A)\le r\}$ for $A\subset \R^n$,  where $d(x,A) = \inf_{y\in A}d(x,y)$.
Then $c_0$ is \new{a} Kuratowski \new{closure operator} and for $r > 0$, $c_r$ is not \new{a} Kuratowski \new{closure operator}.
\end{example}

\new{
\begin{example}
  Let $(X,c)$ be a closure space and let $A \subset X$.
  For $B \subset A$, define $c_A(B) = c(B) \cap A$.
  Then $(A,c_A)$ is a closure space called a \emph{subspace} of $(X,c)$.
\end{example}
}

\new{Next, we consider morphisms of closure spaces.}

\new{
\begin{definition}{\cite[16.A.1]{vcech1966topological}}
\label{def:continuous_functions}
Let $(X,c_X)$ and $(Y,c_Y)$ be closure spaces.
A \emph{continuous map} $f:(X,c_X) \to (Y,c_Y)$ is a function $f:X \to Y$ such that for every $A\subset X$, $f(c_X(A))\subset c_Y(f(A))$.
A continuous map $f$ is called a \emph{homeomorphism} if $f$ is a bijection and has a continuous inverse.
\end{definition}}

\new{
For closure spaces $(X,c_X)$ and $(Y,c_Y)$, any function $f:X \to Y$ is continuous if $c_X$ is the discrete closure or if $c_Y$ is the indiscrete closure.}

\new{
\begin{lemma}{\cite[16.A.3]{vcech1966topological}}
\label{prop:composition_of_cont_maps}
The composition of continuous maps is continuous.
\end{lemma}}

\new{
\begin{definition} \label{def:closure_space_category}
Let $\cat{Cl}$ denote the category of closure spaces and continuous maps.
\end{definition}}

\new{
The initial object in $\cat{Cl}$ is the empty set with its unique closure and the terminal object in $\cat{Cl}$ is the one point set $*$ with its unique closure. Both of these are Kuratowski closure spaces.
The collection of closures on a set has a partial order, which is induced by the partial order on the collection of subsets given by inclusion.}

\new{
\begin{definition}
\label{def:coarser_finer_closure_operators}
Let $X$ be a set. Suppose $c_1$ and $c_2$ are two closures on $X$.
We write $c_1\le c_2$ and say that $c_1$ is \emph{finer} than $c_2$
and $c_2$ is \emph{coarser} than $c_1$ if $c_1(A)\subset c_2(A)$ for all $A\subset X$. 
\end{definition}}

\new{
Observe that $c_1$ is finer than $c_2$ if and only if the identity map $\mathbf{1}_X:(X,c_1)\to (X,c_2)$ is continuous.
Under the relation $\leq$, the collection of closure operators on a set $X$ forms a poset with initial element the discrete closure and terminal element the indiscrete closure.}

\subsection{Interiors}

Dual to a closure we have the following.

\begin{definition}
\label{def:interior}
Let $X$ be a set.
A function $i:\Pow{X}\to \Pow{X}$ is called an \emph{interior (operator)}
on $X$ if the following axioms are satisfied for all $A,B \subset X$:
\begin{enumerate}
\item (grounded) $i(\varnothing)=\varnothing$,
\item (intensive) $i(A) \subset A$,
\item \new{(additive) $i(A\cap B)=i(A) \cap i(B)$.}
\end{enumerate}
\new{Note that by additivity, for all $A\subset B \subset X$, $i(A \cap B)=i(A)=i(A)\cap i(B)$, from which we get monotonicity, $i(A)\subset i(B)$.}
Subsets $A \subset X$ for which $i(A)=A$ are called \emph{open}.
If in addition, we have the following axiom for all $A \subset X$,
  \begin{enumerate}
      \setcounter{enumi}{3}
  \item (idempotent) $i(i(A)) = i(A)$,
  \end{enumerate}
  then we call $i$ a \emph{Kuratowski} interior \new{operator}.
\end{definition}

\begin{proposition}{\cite[Proposition 2.6]{rieser2021vcech}}
  For a set $X$, let $n: \Pow{X} \to \Pow{X}$ be the \emph{complement} operator given by $n(A) = X \setminus A$, for $A \subset X$. Then a closure $c$ has a corresponding interior $i$, and vice versa, under the following relationships,
  \begin{equation*}
    i = ncn, \quad \text{and} \quad c = nin.
  \end{equation*}
\end{proposition}

\new{From the previous proposition it follows that a subset in a closure space is closed iff its complement is open and vice versa.
  Furthermore, $c$ is a Kuratowski closure \new{operator} iff $i$ is a Kuratowski interior \new{operator}.}

%
%

\subsection{Neighborhoods}

\new{Hausdorff defined topological spaces using neighborhoods~\cite[p.\ 213]{Hausdorff:1914}. It is now standard to omit his last axiom (the Hausdorff axiom). If we also omit his second last axiom, then we obtain what we call neighborhood spaces, which we will show in the next section are equivalent to \v{C}ech's closure spaces.}
Neighborhood functions assign each element a collection of subsets.

\begin{definition} \label{def:neighborhood}
  Let $X$ be a set.
  Let $\mathcal{N}: X \to \Pow{\Pow{X}}$ be a function such that for all $x \in X$,
  \begin{enumerate}
  \item (nonempty) $\mathcal{N}(x) \neq \varnothing$,
  \item (contains $x$) for all $A \in \mathcal{N}(x)$, $x \in A$,
  \item (upward closed) if $A \in \mathcal{N}(x)$ and $A \subset B$ then $B \in \mathcal{N}(x)$, and
  \item (closed under binary intersections) if $A,B \in \mathcal{N}(x)$ then $A \cap B \in \mathcal{N}(x)$.
  \end{enumerate}
  The last two conditions say that $\mathcal{N}(x)$ is a \emph{filter}.
  Thus we have that $\mathcal{N}$ specifies for each $x \in X$ a nonempty filter, each of whose sets contains $x$.
  For $x \in X$, call $A \in \mathcal{N}(x)$ a \emph{neighborhood} of $x$ and call $\mathcal{N}(x)$ a \emph{neighborhood filter}.
  Call $\mathcal{N}$ a \emph{collection of neighborhood filters} and call $(X,\mathcal{N})$ a \emph{neighborhood space}.
  If in addition, the following axiom is satisfied,
  \begin{enumerate}
      \setcounter{enumi}{4}
  \item if $A \in \mathcal{N}(x)$ then there exists a $B \in \mathcal{N}(x)$ with $B \subset A$ such that for all $y \in B$, there exists $C \in \mathcal{N}(y)$ such that $C \subset B$.
  \end{enumerate}
  then call $(X,\mathcal{N})$ a \emph{topological space}~\cite[p.\ 213]{Hausdorff:1914}.
\end{definition}

\begin{lemma} \label{cor:nghd-base}
  Given a set $X$, we may define a unique neighborhood space by specifying for each $x \in X$, a \emph{base} for a neighborhood filter. That is, a nonempty collection $\mathcal{B}(x)$ of subsets of $X$ such that each $U \in \mathcal{B}(x)$ contains $x$ and if $U,V \in \mathcal{B}(x)$ then there exists a $W \in \mathcal{B}(x)$ such that $W \subset U \cap V$.
\end{lemma}

\begin{definition}
    Let $(X,\mathcal{N})$ and $(Y,\mathcal{M})$ be neighborhood spaces.
    A \emph{continuous} map $f:(X,\mathcal{N}) \to (Y,\mathcal{M})$ is a function $f:X \to Y$ such that for all $x \in X$ and for all $A \in \mathcal{M}(f(x))$, $f^{-1}(A) \in \mathcal{N}(x)$.
Equivalently, $f$ is continuous if and only if for each $x \in X$ and for each $A \in \mathcal{M}(f(x))$ there is a $B \in \mathcal{N}(x)$ such that $f(B) \subset A$.
\end{definition}

If we have a base $\mathcal{B}$ for $\mathcal{N}$ and a base $\mathcal{C}$ for $\mathcal{M}$ then $f$ is continuous iff for each $x \in X$ and for each $A \in \mathcal{C}(f(x))$ there is a $B \in \mathcal{B}(x)$ such that $f(B) \subset A$.

\subsection{Correspondence of closures and neighborhoods}

For topological spaces, closures and interiors may be defined using neighborhoods and vice versa; the same is true for closure spaces.

\begin{definition} \label{def:cl-nghd}
  Let $X$ be a set.
  Given a closure $c$ and corresponding interior $i$, define $\mathcal{N}:X \to \Pow{\Pow{X}}$ by
  \begin{align*}
    \mathcal{N}(x) &= \{A \subset X \ | \ x \in i(A)\} \\
                   &= \{A \subset X \ | \ x \not\in c(X \setminus A)\}.    
  \end{align*}
  Given a collection of neighborhood filters $\mathcal{N}$, define $i,c: \Pow{X} \to \Pow{X}$ by
  \begin{gather*}
      i(A) = \{x \in X \ | \ \exists U \in \mathcal{N}(x) \text{ such that } U \subset A\}, \\
    c(A) = \{x \in X \ | \ \forall U \in \mathcal{N}(x), U \cap A \neq \varnothing\}.
  \end{gather*}

\end{definition}

\begin{proposition}{\cite[16.A.4, 16.A.5]{vcech1966topological}}
  The constructions in \cref{def:cl-nghd} together with the identity map on functions define an isomorphism of categories between closure spaces and continuous maps and neighborhood spaces and continuous maps.
  Furthermore, this isomorphism restricts to an isomorphism of the full subcategories of Kuratowski closure spaces and topological spaces.
\end{proposition}

We will use this isomorphism of categories implicitly.
For example, we will consider topological spaces to be synonymous with Kuratowski closure spaces.
%
Under this correspondence, a neighborhood $U$ of $x \in X$ is open iff $i(U)=U$ iff for all $y \in U$ there is a neighborhood $V$ of $y$ such that $V \subset U$.
Therefore, if in \cref{cor:nghd-base}, for each $x \in X$, $\mathcal{B}(x)$ consists of open neighborhoods, then we obtain a topology on $X$.

\subsection{Topological spaces}

Topological spaces are special cases of closure spaces and for each closure space there is a canonical topological space.

\begin{definition} \label{def:top}
  Let $\cat{Top}$ denote the full subcategory of $\mathbf{Cl}$ whose objects are Kuratowski closure spaces. That is, $\cat{Top}$ is the category of topological spaces and continuous maps.
\end{definition}

\begin{lemma}{\cite[16.A.10]{vcech1966topological}}
\label{theorem:continuous_maps_of_topological_spaces}
Let $(X,\new{c_X}) \in \cat{Cl}$ and $(Y,\new{c_Y}) \in \cat{Top}$.
A function $f:X \to Y$ is a continuous map if and only if the inverse image of every open set is open.
\end{lemma}

\begin{lemma}{\cite[16.B.1-16.B.3]{vcech1966topological}}
\label{prop:topological_modification}
Let $(X,c)$ be a closure space.
For $A \subset X$, define $\tau(c)(A)$ to be the intersection of all closed sets containing $A$.
Then $\tau (c)$ is
is the finest Kuratowski closure \new{operator} coarser than $c$, and is called the \emph{topological modification of $c$}.
\end{lemma}

\begin{proposition}{\cite[16.B.4]{vcech1966topological}}
\label{prop:adjoints_between_closures_and_topologies}
Let $(X,c)$ be a closure space and let $(Y,\tau)$ be a Kuratowski closure space.
Consider a function $f:X\to Y$. Then $f:(X,c)\to (Y,\tau)$ is continuous if and only if $f:(X,\tau(c))\to (Y,\tau)$ is continuous. That is, there exists a natural bijection between the sets of morphisms
\[\mathbf{Cl}((X,c),(Y,\tau))\cong \mathbf{Top}((X,\tau(c)),(Y,\tau)).\]
Equivalently, $\tau$ is the left adjoint to the inclusion functor $\iota:\mathbf{Top}\to \mathbf{Cl}$.
\end{proposition}

\subsection{Covers and pasting}

Similarly to topological spaces, we have a pasting lemma.

A \emph{cover} of a closure space $(X,c)$ is a family of subsets of $X$ 
whose union is $X$.
  A cover
  is called \emph{locally finite} if each $x\in X$ has a neighborhood intersecting only finitely many
  elements of the cover.

\begin{theorem}{\cite[17.A.16]{vcech1966topological}}
\label{theorem:pasting_lemma_for_closure_spaces}
Let $(X,\new{c_X})$ and $(Y,\new{c_Y})$ be closure spaces and let $\{U_{\alpha}\,|\,\alpha\in A\}$ be a locally finite cover of $(X,\new{c_X})$. Let $f:X \to Y$ be a map of sets. If $f|_{\new{c_X}(U_{\alpha})}:(\new{c_X}(U_{\alpha}),c_{\new{c_X}(U_{\alpha})}) \to (Y,\new{c_Y})$ is continuous for each $\alpha\in A$, then $f:(X,\new{c_X})\to (Y,\new{c_Y})$ is continuous.
\end{theorem}

\subsection{Elementary examples}

We will use some of these examples in \cref{section:homotopy,section:homology}.
Let $m \geq 0$.

\begin{definition}
\label{def:interval_objects}
\begin{enumerate}[wide,
labelindent=0pt]
\item Let $I_{\tau}$  denote the unit interval $[0,1]$ with the Kuratowski closure \new{operator} given the closure $c_0$ from \cref{example:closure_space_not_a_topological_space}. This closure corresponds to the standard topology on the unit interval.
\item Let $J_{m,\bot}$ denote the set $\{0,\dots, m\}$ with the discrete closure. As a special case, set $J_{\bot} = J_{1,\bot}$.
\item Let $J_{m,\top}$ denote the set $\{0,\ldots,m\}$ with the indiscrete closure. As a special case, set $J_{\top} = J_{1,\top}$.
\item Let $J_m$ denote the set $\{0,\dots ,m\}$ with the closure operator $c(i)=\{j\in \{0,\dots, m\}\,|\, |i-j|\le 1\}$. Note that $J_1 = J_{\top}$.
\item Let $J_+$ denote the set $\{0,1\}$ with the closure operator $c_+(0)=\{0,1\}$, $c_+(1)=\{1\}$. Let $J_{-}$ denote the set $\{0,1\}$ with the closure operator $c_-(0)=\{0\}$, $c_-(1)=\{0,1\}$.
These are homeomorphic Kuratowski closure spaces.
In fact, $J_{-}$ is the \emph{Sierpinski space}, which has open sets $\varnothing$, $\{1\}$, and $\{0,1\}$.
\item
For each $0 \leq k \leq 2^m-1$ we define a closure operator $c_k$ on the set $\{0,\ldots,m\}$
as follows.
Consider the binary representation of $k$.
For $1 \leq i \leq m$, $i-1$ is contained in $c_k(i)$ iff the $i$th lowest order bit is $0$ and
$i$ is contained in $c_k(i-1)$ iff the $i$th lowest order bit is $1$.
  We denote this closure space by $J_{m,k}$.
  Note that $J_{1,1} = J_{+}$ and $J_{1,0} = J_{-}$.
\item Let $J_{m,\leq}$ denote the set $\{0,1,\dots ,m\}$ with the Kuratowski closure operator $c(i)=\{j\,\mid\, i\le j\}$.  The open sets are the down-sets. Note that $J_{1,\leq} = J_+$.
\end{enumerate}
\end{definition}

 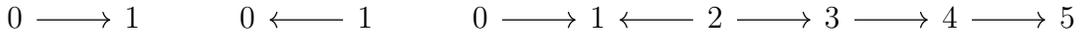
\begin{figure}[H]
  \begin{tikzcd}
    0 \ar[r] & 1\\
  \end{tikzcd} \quad \quad
  \begin{tikzcd}
    0 & 1 \ar[l]\\
  \end{tikzcd} \quad \quad
  \begin{tikzcd}
    0 \ar[r] & 1 & 2 \ar[l] \ar[r] & 3 \ar[r] & 4 \ar[r] & 5\\
  \end{tikzcd}
  \caption{Representations of the closure spaces $J_+$ (left), $J_-$ (middle), and $J_{5,29}$ (right).
    The head of the arrow is contained in the closure of the tail of the arrow. Note that 29 in binary is 11101}
 \label{fig:closure_space_examples}
 \end{figure}

\subsection{Products and inductive products}

Closure spaces have two canonical products, the product and the inductive product, which we describe below.

\begin{definition}{\cite[17.C.1 and 17.C.3]{vcech1966topological}}
\label{def:product_closure}
Let $\{(X_{\alpha},c_{\alpha})\}_{\alpha\in A}$ be a collection of closure spaces.
For $\alpha \in A$, 
let $\mathcal{N}_{\alpha}$ denote the corresponding collection of neighborhood filters, or more generally, a collection of neighborhood filter bases.
Let $X$ denote the set $\prod_{\alpha\in A}X_{\alpha}$, and for $\alpha\in A$, let $\pi_{\alpha}:X\to X_{\alpha}$ denote the projection map.
For $x\in X$,
  define $\mathcal{N}(x)$ to be given by finite intersections of sets of the form
$\pi_{\alpha}^{-1}(V_{\alpha})$, where $V_{\alpha} \in \mathcal{N}_{\alpha}(x_{\alpha})$, and $x_{\alpha} \in X_{\alpha}$.
Then $\mathcal{N}$ is a collection of neighborhood filter bases for $X$.
We call the corresponding closure $c$ the \emph{product closure} for $X$ and call the pair $(X,c)$ the \emph{product closure space}.

When we have two closure spaces $(X,\new{c_X})$ and $(Y,\new{c_Y})$ we will denote the product closure space by $(X\times Y,\new{c_X\times c_Y})$.
In this case, if $X$ and $Y$ have collections of neighborhood filters or neighborhood filter bases $\mathcal{N}$ and $\mathcal{M}$ then 
for $(x,y) \in X \times Y$, $\mathcal{N}(x,y) = \mathcal{N}(x) \times \mathcal{M}(y)$.
It follows that for $A \subset X \times Y$, $(\new{c_X \times c_Y})(A) = \{(x,y) \in X \times Y \ | \ \forall U \in \mathcal{N}(x), V \in \mathcal{M}(y), (U \times V) \cap A \neq \varnothing\}$.
\end{definition}

\begin{theorem}{\cite[17.C.6]{vcech1966topological}}
\label{prop:product_closure}
  The product closure is the coarsest closure for which the
projection maps are continuous.
\end{theorem}

\begin{definition}{\cite[17.D.1]{vcech1966topological}}
\label{def:inductive_product_closure}
Let $(X,\new{c_X})$ and $(Y,\new{c_Y})$ be closure spaces with corresponding collections of neighborhood filters or neighborhood filter bases, $\mathcal{N}$ and $\mathcal{M}$, respectively.
For $(x,y)\in X\times Y$, let $\mathcal{P}(x,y)$ be the collection of all sets of the form
\begin{equation} \label{eq:inductive} 
(\{x\}\times V)\cup (U\times \{y\})  
\end{equation}
where $V \in \mathcal{M}(y)$ and $U \in \mathcal{N}(x)$.
Then $\mathcal{P}$ is a collection of neighborhood filter bases for $X \times Y$.
We call the corresponding closure $\new{c_X} \boxplus \new{c_Y}$
the \emph{inductive product closure},
and call $(X \times Y, \new{c_X \boxplus c_Y})$ the inductive product closure space.
The sets in \eqref{eq:inductive} may be written as $\{(x',y') \in U \times V \,|\, x'=x \text{ or } y'=y\}$.
\end{definition}

The terminology `inductive product' is due to the following result, which may be used to define the inductive product as an `inductively generated closure' on the product of the underlying sets~\cite[33.D.1]{vcech1966topological}, whereas the product may be called the `projective product' and may be defined as a `projectively generated closure' on the same set~\cite[32.A.3(f)]{vcech1966topological}.

\begin{theorem}{\cite[17.D.3]{vcech1966topological}}
\label{theorem:characterization_of_ind_prod_closure}
Let $(X,\new{c_X})$ and $(Y,\new{c_Y})$ be closure spaces.
The inductive product closure is the finest closure for $X\times Y$ for which all of the maps
\begin{itemize}
\item $X \to X \times Y$ given by $x\mapsto (x,y_0)$, where $y_0 \in Y$, and
\item $Y \to X \times Y$ given by $y \mapsto (x_0,y)$, where $x_0 \in X$,
\end{itemize}
are continuous.
\end{theorem}

\begin{proposition}{\cite[17.D.2]{vcech1966topological}}
\label{prop:comparison_of_products}
The inductive product closure is finer than
the product closure.
\end{proposition}

\begin{lemma}
\label{lemma:ind_product_and_product_of_discrete_spaces}
Let $X$ be a closure space and let $Y$ be a discrete space. Then $X\times Y=X\boxplus Y$.
\end{lemma}

\begin{proof}
  Let $\mathcal{N}$ be the corresponding collection of neighborhood filters for $X$.
  Since $Y$ is discrete it has a corresponding collection of neighborhood filter bases given by $\mathcal{M}(y) = \{y\}$, for $y \in Y$.
Let $(x,y) \in X \times Y$.
Then
$\{U \times \{y\} \ | \ U \in \mathcal{N}(x)\}$ is a
neighborhood filter base
of $(x,y)$ for both closures.
\end{proof}

\begin{proposition}{\cite[17.C.11]{vcech1966topological}}
\label{prop:product_of_maps}
  Suppose we are given for each $a \in A$ closure spaces  $(X_a,c_{X_a})$ and
  $(Y_a,c_{Y_a})$ and a map of sets $f_a:X_a \to Y_a$.
  If for all $a \in A$, $f_a$ is continuous, then the mapping $f:(\prod_{a\in A}X_a,\prod_{a\in A}c_{X_a})\to (\prod_{a\in A}Y_a,\prod_{a\in A}c_{Y_a})$ defined by $\{x_a\}_{a\in A}\mapsto \{f_x(x_a)\}_{a\in A}$ is continuous. Conversely, if $f$ is continuous and $\prod_{a\in A}X_a\neq \varnothing$, then for all $a \in A$, $f_a$ is continuous.
\end{proposition}

\begin{proposition}{\cite[17.D.5]{vcech1966topological}}
\label{prop:ind_prod_cl_good_for_homotopy_theory}
Let $(X,\new{c_X}),(Y,\new{c_Y})$ and $(Z,\new{c_Z})$ be closure spaces. A function $f:X \times Y \to Z$ is a continuous map $f:(X,\new{c_X})\boxplus (Y,\new{c_Y})\to (Z,\new{c_Z})$ if and only if all of the maps 
\begin{itemize}
\item $X \to Z$ given by $x\mapsto f(x,y_0)$, where $y_0\in Y$, and
\item $Y \to Z$ given by $y \mapsto f(x_0,y)$, where $x_0\in X$,
\end{itemize}
are continuous.
\end{proposition}

\begin{corollary}
\label{prop:inductive_product_of_maps}
Let $f:(X,\new{c_X})\to (Y,\new{c_Y})$ and $g:(X',\new{c_{X'}})\to (Y',\new{c_{Y'}})$ be continuous maps of closure spaces. Then the map $f\times g:(X,\new{c_X})\boxplus (X',\new{c_{X'}})\to (Y,\new{c_{Y}})\boxplus (Y',\new{c_{Y'}})$ defined by $(f\times g)(x,x') =(f(x),g(x'))$ is continuous.
\end{corollary}

\subsection{Coproducts and pushouts}

Closure spaces have (small) colimits and, in particular, have pushouts.

\begin{definition}{\cite[17.B.1]{vcech1966topological}}
\label{def:coproducts}
Let $\{(X_i,c_i)\}_{i\in I}$ be a collection of closure spaces. The \emph{coproduct} 
of $\{(X_i,c_i)\}_{i\in I}$ is the  disjoint union of sets $X=\coprod_i X_i$ with the closure operator $c$ defined by $c(\coprod_i A_i)=\coprod_i c_i(A_i)$ for any subset $\coprod_i A_i$ of $X$.
\end{definition}

\begin{definition} \label{def:coequalizer}
  Given continuous maps
  $
  \begin{tikzcd}(X,\new{c_X}) \ar[r,shift left,"f"] \ar[r,shift right,"g"'] & (Y,\new{c_Y})
  \end{tikzcd}
  $
 the \emph{coequalizer of $f$ and $g$} consists of the closure space $(Q,\new{c_Q})$ and the map $p:Y \to Q$ defined as follows.
  Let $Q$ be the quotient set $Y/\! \sim$ for the equivalence relation given by $f(x) \sim g(x)$ for all $x \in X$. Let $p:Y \to Q$ be the quotient map.
  For $A \subset Q$, define $\new{c_Q}(A) = p(\new{c_Y}(p^{-1}(A)))$.
\end{definition}

\begin{theorem}{\cite[33.A.4 and 33.A.5]{vcech1966topological}} \label{cocomplete}
 The coproduct and coequalizer defined above are the categorical coproduct and coequalizer in the category $\cat{Cl}$ and hence $\cat{Cl}$ is cocomplete.
\end{theorem}

As an application, we have pushouts of closure spaces.

\begin{definition} \label{def:pushout}
  The \emph{pushout} in $\cat{Cl}$ is defined as follows. Given the solid arrow diagram
    \begin{equation*}
      \begin{tikzcd}
        (A,\new{c_A}) \ar[r,"f"] \ar[d,"g"] & (X,\new{c_X}) \ar[d,dashed,"i"]\\
        (Y,\new{c_X}) \ar[r,dashed,"j"'] & (P,\new{c_P})
      \end{tikzcd}
    \end{equation*}
   define $P = (X \amalg Y)/\! \sim$ where $f(a) \sim g(a)$ for all $a \in A$, let $i,j$ be the induced maps, and for $B \subset P$, define $\new{c_P}(B) = i(\new{c_X}(i^{-1}(B))) \cup j(\new{c_Y}(j^{-1}(B)))$.
\end{definition}

\begin{lemma}
\label{lemma:generalized_intervals}
$J_{m,\bot}$, $J_m$ and $J_{m,k}$ are obtained by $m$-fold pushouts of $J_{\bot}$, $J_{\top}$, and $J_{-},J_{+}$, respectively, under $*$. 
\end{lemma}

\begin{proof}
  For $m=2$ consider the following pushouts.
\begin{equation*}
\begin{tikzcd}
*\arrow[r,"1"] \arrow[d,"0"]& J_{\bot}\arrow[d]\\
J_{\bot}\arrow[r]& J_{2,\bot}
\end{tikzcd}
\begin{tikzcd}
*\arrow[r,"1"] \arrow[d,"0"]& J_{\top}\arrow[d]\\
J_{\top}\arrow[r]& J_2
\end{tikzcd}
\begin{tikzcd}
*\arrow[r,"1"] \arrow[d,"0"]& J_-\arrow[d]\\
J_-\arrow[r]& J_{2,0}
\end{tikzcd}
\begin{tikzcd}
*\arrow[r,"1"] \arrow[d,"0"]& J_{+}\arrow[d]\\
J_{-}\arrow[r]& J_{2,1}
\end{tikzcd}
\begin{tikzcd}
*\arrow[r,"1"] \arrow[d,"0"]& J_{-}\arrow[d]\\
J_{+}\arrow[r]& J_{2,2}
\end{tikzcd}
\begin{tikzcd}
*\arrow[r,"1"] \arrow[d,"0"]& J_+\arrow[d]\\
J_+\arrow[r]& J_{2,3}
\end{tikzcd}
\end{equation*}
For the general case proceed by induction. 
\end{proof}

\begin{lemma}
\label{lemma:comparison_of_intervals}
The identity maps $J_{m,k}\xrightarrow{\mathbf{1}} J_m$ are continuous for all $m\ge 0$ and $0\le k\le 2^m-1$.
The `round up' map $f_+:I_{\tau}\to J_{+}$ defined by $f(x)=0$ if $x<\frac{1}{2}$ and $f(x)=1$ if $x\ge \frac{1}{2}$
and the `round down' map $f_{-}:I_{\tau}\to J_{-}$ defined by $f(x)=0$ if $x\leq\frac{1}{2}$ and $f(x)=1$ if $x> \frac{1}{2}$
are continuous.
For any $m\ge 1$ and $0\le k\le 2^m-1$,
these may be combined to obtain continuous maps $f:[0,m] \to J_{m,k}$,
where $[0,m]$ has the Kuratowski closure \new{operator} corresponding to the standard topology.
Precomposing with the map $t \mapsto mt$, we obtain a continuous surjective map $f:I_{\tau}\to J_{m,k}$.
\end{lemma}

\begin{proof}
For each $m\ge 0$ and $0\le k\le 2^m-1$ the closure operators for $J_{m,k}$ are finer than the one for $J_m$.
For the `round up' map, note that $J_+$ is a topological space with non-empty open sets $\{0\}$ and $\{0,1\}$. By \cref{theorem:continuous_maps_of_topological_spaces}, it is sufficient to show that $f^{-1}_+(0)$ and $f_+^{-1}(\{0,1\})$ are open in $I_{\tau}$, which they are by the definition of $f$. 
The `round down' map is similar.
For the third statement consider the cover of $[0,m]$ consisting of the closed sets $[0,1],[1,2],\ldots,[m-1,m]$.
  By \cref{lemma:generalized_intervals}, 
  for each $i \in \{1,\ldots,m\}$, we have either a `round up' or a `round down' map from $[i-1,i]$ to $J_{m,k}$ whose image is $\{i-1,i\}$ and which sends $i-1$ to $i-1$ and $i$ to $i$.
The third statement follows from 
\cref{theorem:pasting_lemma_for_closure_spaces}.
The last statement follows from \cref{prop:composition_of_cont_maps}.
\end{proof}

\subsection{Pullbacks}

Closure spaces have (small) limits and, in particular, have pullbacks.

\begin{definition} \label{def:equalizer}
  Given continuous maps
  $
  \begin{tikzcd}(X,\new{c_X}) \ar[r,shift left,"f"] \ar[r,shift right,"g"'] & (Y,\new{c_Y})
  \end{tikzcd}
  $
  the \emph{equalizer of $f$ and $g$} consists of the closure space $(E,\new{c_E})$ and the map $i:E \to X$ defined as follows.
  Let $E$ be the subset $\{x \in X \ | \ f(x)=g(x)\}$ with $i$ the inclusion map.
  For $A \subset E$, define $\new{c_E}(A) = \new{c_X}(A) \cap E$.
\end{definition}

\begin{theorem}{\cite[32.A.4 and 32.A.10]{vcech1966topological}} \label{complete}
  The product (\cref{def:product_closure}) and equalizer defined above are the categorical product and equalizer in the category $\cat{Cl}$ and hence $\cat{Cl}$ is complete.
\end{theorem}

\begin{proposition}
\label{prop:limits_and_colimits_of_closure_spaces}
 Every limit of topological spaces in $\mathbf{Cl}$ is a topological space. On the other hand, colimits of topological spaces in $\mathbf{Cl}$ are not necessarily topological spaces.
\end{proposition}
 
\begin{proof}
  The inclusion functor $\iota:\mathbf{Top}\to \mathbf{Cl}$ is a right-adjoint by \cref{prop:adjoints_between_closures_and_topologies} and thus preserves limits.
 For the second statement, consider
  the second pushout diagram in the proof of \cref{lemma:generalized_intervals}.
  Note that the one point space and $J_{\top}$ are both topological spaces (with the indiscrete topology), however the pushout $J_2$, is not a topological space.
  See also \cite[Example 2.52]{rieser2021vcech} which is a more modern take on \cite[Introduction to Section 33.B]{vcech1966topological}.
\end{proof}

\subsection{Monomorphisms and epimorphisms}
\label{sec:mono-epi}

The monomorphisms and epimorphisms in $\cat{Cl}$ are easy to describe.

\begin{lemma}
    Consider $f:(X,\new{c_X}) \to (Y,\new{c_Y}) \in \cat{Cl}$. Then
    \begin{enumerate}
    \item \label{it:mono} $f$ is a monomorphism if an only if the underlying map of sets is injective, and 
    \item \label{it:epi} $f$ is an epimorphism if and only if the underlying map of sets is surjective.
    \end{enumerate}

\end{lemma}

\begin{proof}
  Let $U: \cat{Cl} \to \cat{Set}$ denote the underlying set functor.

  \eqref{it:mono} $(\Leftarrow)$
This implication holds
 in any concrete category: $fh=fg$ implies that $U(f)U(h) = U(fh) = U(fg) = U(f)U(g)$, which implies that $Uh=Ug$, which implies that $h=g$.

  $(\Rightarrow)$ Let $x,x' \in X$. Consider the maps $x,x':* \to X$ with image $x$ and $x'$. Then $f(x)=f(x')$ implies that $x=x'$.

  \eqref{it:epi} $(\Leftarrow)$
This implication holds
 in any concrete category: if $gf=hf$ then $U(g)U(h) = U(gh) = U(hf) = U(h)U(f)$ which implies that $U(g)=U(h)$ and thus $g=h$.

  $(\Rightarrow)$ Consider the maps $g,h:Y \to (\{0,1\},c_{\top})$ given by $g(y) = 1$ if $y$ is in the image of $f$ and $0$ otherwise, and $h(y) = 1$ for all $y \in Y$. Then $hf=gf$ implies that $h=g$ and thus $Uf$ is surjective.
\end{proof}

\subsection{Symmetric closures and Alexandroff closures}
\label{section:subcategories_of_closure_spaces}

In \cref{section:graphs_and_closures}, we will relate simple directed graphs and simple graphs to Alexandroff closures spaces and symmetric Alexandroff spaces, respectively.

\begin{definition}
\label{def:symmetric}
A closure space $(X,c)$ is
\emph{symmetric}
if $y \in c(x)$ implies $x\in c(y)$ for all $x,y\in X$.
\end{definition}

A closure space is symmetric if and only if it is \emph{semi-uniformizable}
\cite[Definition 23.A.3 and Theorem 23.B.3]{vcech1966topological}.

\begin{definition}{
    \cite[Example 14.A.5(f)]{vcech1966topological}
    \cite[Section 3.6]{dikranjan2013categorical}
  }
\label{def:quasi_discrete_closure}
Let $(X,c)$ be a closure space.
Say that the closure $c$ is \emph{Alexandroff} if
  for every collection $\{A_i\}_{i \in I}$ of subsets of $X$, $c(\bigcup_{i \in I} A_i) = \bigcup_{i \in I} c(A_i)$.
  Equivalently, $c$ is Alexandroff if for every $A \subset X$,
  $c(A) = \bigcup_{a \in A} c(a)$.
\end{definition}

Note that every finite closure space is Alexandroff.
For a closure space $(X,c)$, if the closure \new{operator} is Alexandroff and Kuratowski then the corresponding topological space is an \emph{Alexandroff space}~\cite{Alexandroff:1937}.

\begin{definition}
\label{def:closure_subcategory}
We denote by $\sucl$, $\qdcl$ and $\suqdcl$ the full subcategories of $\mathbf{Cl}$ consisting of symmetric, Alexandroff, and symmetric Alexandroff closure spaces, respectively.
\end{definition}

\begin{proposition}
  \label{prop:continuity_on_atomic_closure spaces}
  Suppose $(X,\new{c_X})$ is an Alexandroff closure space. Let $(Y,\new{c_Y})$ be a closure space. Then $f:(X,\new{c_X})\to (Y,\new{c_Y})$ is continuous if and only if $\forall x\in X$, $f(\new{c_X}(x))\subset \new{c_Y}(f(x))$.
\end{proposition}

\begin{proof}
  The forward direction follows from the definition of continuity. For the reverse direction, let $A\subset X$. Then $f(\new{c_X}(A))=f(\bigcup_{x\in A}\new{c_X}(x))=\bigcup_{x\in A}(f(\new{c_X}(x)))\subset \bigcup_{x\in A}\new{c_Y}(f(x))\subset \new{c_Y}(f(A))$.
  Thus $f$ is continuous.
\end{proof}

\begin{definition} \label{def:induced-qd}
  Let $X$ be a set and let $\rho:X \to \Pow{X}$ such that $x \in \rho(x)$.
  Define $c:\Pow{X} \to \Pow{X}$ by $c(A) = \cup_{x \in A}\rho(x)$ for $A \subset X$. It is easy to verify that $c$ is an Alexandroff closure operator, which we call the \emph{induced Alexandroff closure operator}.
\end{definition}

\begin{definition} \label{def:qd}
    Let $(X,c)$ be a closure space.
    Let $\qd(c)$ denote the induced Alexandroff closure operator on the restriction of $c$ to one-point sets. 
    Then $\qd(c)$ is finer than $c$ and is called the \emph{Alexandroff modification} of $c$ (\cite[Definition 26.A.1]{vcech1966topological}).
    The mappings $(X,c)\mapsto (X,\qd(c))$ and
    $(f:(X,\new{c_X})\to (Y,\new{c_Y}))\mapsto (f:(X,\qd(\new{c_X}))\to (Y,\qd(\new{c_Y})))$
    define a functor $\qd:\mathbf{Cl}\to \qdcl$.
\end{definition}

\begin{proposition}
  \label{prop:adjunction-cl-clqd} 
Let $(X,\new{c_X}) \in \qdcl$ and $(Y,\new{c_Y}) \in \cat{Cl}$.
  Given a set map $f:X \to Y$, $f:(X,\new{c_X}) \to (Y,\qd(\new{c_Y}))$ is continuous iff $f:(X,\new{c_X}) \to (Y,\new{c_Y})$ is continuous.
  Thus, 
  we have a natural bijection 
\[\mathbf{Cl}((X,\new{c_X}),(Y,\new{c_Y}))\cong \qdcl((X,\new{c_X}),(Y,\qd(\new{c_Y}))).\]
That is, $\qd$
is right adjoint to the inclusion functor $\qdcl \incl \mathbf{Cl}$.
\end{proposition}

\begin{proof}
  $(\Rightarrow)$ Let $A \subset X$. Then $f(\new{c_X}(A)) \subset \qd(\new{c_Y})(f(A)) \subset \new{c_Y}(f(A))$.

 $(\Leftarrow)$ Let $x \in X$. Then $f(\new{c_X}(x)) \subset \new{c_Y}(f(x)) = \qd(\new{c_Y})(f(x))$.
The result follows from \cref{prop:continuity_on_atomic_closure spaces}.
\end{proof}

\begin{definition}
Let $(X,c) \in \qdcl$. Let the \emph{reverse closure}, $c^{\transpose}$, be the induced Alexandroff closure operator of $\rho(x) = \{y \in X \ | \ x \in c(y)\}$. That is, for $A \subset X$, $c^{\transpose}(A) = \{y \in X \ | \ c(y) \cap A \neq \varnothing\}$.
\end{definition}

\begin{lemma}
 We have a \emph{reverse} functor $(-)^{\transpose}: \qdcl \to \qdcl$ mapping $(X,c)$ to $(X,c^{\transpose})$ and sending functions to themselves.
\end{lemma}

\begin{proof}
 Let $(X,\new{c_X}),(Y,\new{c_Y}) \in \qdcl$ and $f:X \to Y$ such that for all $x \in X$, $f(\new{c_X}(x)) \subset \new{c_Y}(f(x))$. Let $x \in X$.
  It remains to check that $f(\new{c_X}^{\transpose}(x)) \subset \new{c_Y}^{\transpose}(f(x))$.
  Let $a \in \new{c_X}^{\transpose}(x)$. Then $x \in \new{c_X}(a)$ and thus $f(x) \in f(\new{c_X}(a)) \subset \new{c_Y}(f(a))$.
  Therefore $f(a) \in \new{c_Y}^{\transpose}(f(x))$.
\end{proof}

\begin{definition}
  Let $(X,\new{c_X}) \in \qdcl$. Let $s(\new{c_X})$ be the Alexandroff closure induced by $\rho(x) = \{y \in \new{c_X}(x) \ | \ x \in \new{c_X}(y)\}$.
  If $y \in \rho(x)$ then $y \in \new{c_X}(x)$ and $x \in \new{c_X}(y)$. So $x \in \rho(y)$.
  That is, $s(\new{c_X})$ is symmetric.
  Thus $s(\new{c_X})$ is a symmetric Alexandroff closure finer than $\new{c_X}$, called the \emph{symmetrization} of $\new{c_X}$.
  Let $f:(X,\new{c_X}) \to (Y,\new{c_Y})$.
  For $x \in X$, $f(s(\new{c_X})(x)) = f\{y \in \new{c_X}(x) \ | \ x \in \new{c_X}(y)\}$ and
  $s(\new{c_Y})(f(x)) = \{z \in \new{c_Y}(f(x)) \ | \ f(x) \in \new{c_Y}(z)\}$.
  Now $y \in \new{c_X}(x)$ implies that $f(y) \subset \new{c_Y}(f(x))$ and $x \in \new{c_X}(y)$ implies that $f(x) \in \new{c_Y}(f(y))$.
  Therefore $f(s(\new{c_X})(x)) \subset s(\new{c_Y})(f(x))$.
  That is, $f:(X,s(\new{c_X})) \to (Y,s(\new{c_Y}))$.
  Let $s: \qdcl \to \suqdcl$ denote the functor defined by these mappings.
\end{definition}

\begin{proposition} \label{prop:adjunction-clqd-clsqd}
Let $(X,\new{c_X}) \in \suqdcl$ and $(Y,\new{c_Y}) \in \qdcl$.
  Given a set map $f:X \to Y$, $f:(X,\new{c_X}) \to (Y,s(\new{c_Y}))$ is continuous iff $f:(X,\new{c_X}) \to (Y,\new{c_Y})$ is continuous.
  Thus, 
  we have a natural isomorphism
\[\qdcl((X,\new{c_X}),(Y,\new{c_Y}))\cong \suqdcl((X,\new{c_X}),(Y,s(\new{c_Y}))).\]
That is, $s$
is right adjoint to the inclusion functor $\suqdcl \incl \qdcl$.
\end{proposition}

\begin{proof}
    $(\Rightarrow)$ Let $x \in X$. Then $f(\new{c_X}(x)) \subset s(\new{c_Y})(f(x)) \subset \new{c_Y}(f(x))$.

    $(\Leftarrow)$ Let $x \in X$. Then $f(\new{c_X}(x)) \subset \new{c_Y}(f(x))$. Let $y \in \new{c_X}(x)$. Then $f(y) \in \new{c_Y}(f(x))$. Also, $x \in \new{c_X}(y)$. Thus $f(x) \in \new{c_Y}(f(y))$.
    Hence $f(y) \in s(\new{c_Y})(f(x))$ and therefore $f(\new{c_X}(x)) \subset s(\new{c_Y})(f(x))$.
  \end{proof}

By definition, the following functors are equal.

\begin{proposition} \label{prop:reverse}
$s \circ (-)^{\transpose} = s$.
\end{proposition}

\subsection{Simple  graphs and simple directed graphs as closure spaces}
\label{section:graphs_and_closures}

The categories of simple directed graphs and simple graphs
are isomorphic to the categories of Alexandroff closure spaces and symmetric Alexandroff spaces, respectively
\cite[Chapter 26]{vcech1966topological}
\cite[Section 3.6]{dikranjan2013categorical}.
We allow the sets of vertices and edges of graphs to have arbitrary cardinality.
Given a set $X$, let $\Delta = \{(x,x) \ | \ x \in X\}$.

\begin{definition} \label{def:simple-graph}
A \emph{simple graph} is a pair $(X,E)$ where $X$ is a set and $E$ is a collection of pairs of elements of $X$.
  That is, $E$ is a symmetric relation on $X$ such that $E \cap \Delta = \varnothing$.
  More generally, a \emph{simple directed graph} or \emph{simple digraph} is a pair $(X,E)$ where $X$ is a set and $E$ is a relation on $X$ such that $E \cap \Delta = \varnothing$.
  Let a \emph{spatial digraph} be a pair $(X,E)$ where $X$ is a set and $E$ is a reflexive relation on $X$.
  For each simple digraph there is a corresponding spatial digraph $(X,\overline{E})$, where $\overline{E} = E \cup \Delta$. Note that $E = \overline{E} \setminus \Delta$.
  A map of sets $f:X \to Y$ is \emph{(di)graph homomorphism} between $(X,E)$ and $(Y,F)$ if whenever $x\overline{E}x'$, we have $f(x)\overline{F}f(x')$.
\end{definition}

Let $\cat{DiGph}$ denote the category of simple digraphs and digraph homomorphisms, and let $\cat{Gph}$ denote the full subcategory of simple graphs and graph homomorphisms. 
Recall the categories, $\qdcl$ and $\suqdcl$, of Alexandroff closure spaces and symmetric Alexandroff closure spaces, respectively  (\cref{def:closure_subcategory}).

\begin{definition}
\label{def:functor_from_digraphs_to_closure_spaces}
Let $\Psi: \mathbf{DiGph}\to \qdcl$ be the functor that assigns to each simple digraph $(X,E)$ the closure space $(X,c_E)$,
where $c_E$ is the induced Alexandroff closure (\cref{{def:induced-qd}}) determined by the map $\rho_E:X \to \Pow{X}$ given by
\begin{equation} \label{eq:rho}
  \rho_E(x) 
  = \{y \in X \,|\, x\overline{E}y\}.
\end{equation}

Given a digraph homomorphism $f:(X,E)\to (Y,F)$, the map $\Psi(f):(X,c_E)\to (Y,c_F)$ is given by the map of sets $f$.
By \cref{prop:continuity_on_atomic_closure spaces}, the continuity of $\Psi(f)$ is equivalent to the map $f$ being a digraph homomorphism.
\end{definition}

\begin{definition}
\label{def:functor_from_closure_spaces_to_digraphs}
Let $\Phi:\qdcl\to \mathbf{DiGph}$ be the functor that assigns to each Alexandroff closure space $(X,c)$ the simple digraph $(X,E_c)$ defined by
\[
  x\overline{E}_c y\Longleftrightarrow y\in c(x). 
\]
Given a continuous map $f:(X,\new{c_X})\to (Y,\new{c_Y})$, let $\new{\Phi(f)}:(X,E_{\new{c_X}})\to (Y,E_{\new{c_Y}})$ be the map of sets $f$.
By \cref{prop:continuity_on_atomic_closure spaces}, the continuity of $\new{\Phi(f)}$ is equivalent to the $f$ being a digraph homomorphism.
\end{definition}

We therefore have the following.

\begin{proposition}
\label{lemma:digraphs_are_quasi_discrete_closure_spaces}
The functors $\Psi$ and $\Phi$ are inverses and thus define an isomorphism of categories $\qdcl \isom \cat{DiGph}$.
Furthermore, they restrict to an isomorphism $\suqdcl \isom \cat{Gph}$.
\hfill \qedsymbol
\end{proposition}

\begin{definition} \label{def:reverse-digraph}
  Let $(X,E)$ be a digraph.
  The \emph{reverse digraph} $(X,E^{\transpose})$, is given by $yE^{\transpose}x$ iff $xEy$ for $x,y \in X$. That is, it is the digraph obtained by reversing the directed edges.
  We have a \emph{reverse functor}, $(-)^{\transpose}: \cat{DiGph} \to \cat{DiGph}$ sending a digraph $(X,E)$ to its reverse digraph $(X,E^{\transpose})$ and sending a digraph homomorphism
  $(X,E) \to (Y,F)$ given by $f:X \to Y$ to the digraph homomorphism $(X,E^{\transpose}) \to (Y,F^{\transpose})$ given by $f$.
\end{definition}

Observe that the reverse functors on digraphs and Alexandroff closures are compatible.
That is, $\Phi \circ (-)^{\transpose} = (-)^{\transpose} \circ \Phi$ and thus $(-)^{\transpose} \circ \Psi = \Psi \circ (-)^{\transpose}$.

\section{Graph products from closure space products}
\label{sec:graph-products}

We observe that via the identification of graphs and directed graphs as closure spaces in \cref{section:graphs_and_closures}, one may obtain the canonical product operations of graphs from the canonical product operations of closure spaces.

\subsection{A neighborhood filter base for a digraph}
\label{sec:base-digraph}

We identify a base for the neighborhood filter corresponding to the Alexandroff closure for a digraph, which we will be used in \cref{sec:graph-products-proofs}.

\begin{definition} \label{def:complement-digraph}
  Let $(X,E)$ be a digraph.
  The \emph{complement digraph} $(X,E^c)$ is given by $xE^cy$ iff $x \neq y$ and not $xEy$. That is, a directed edge is in $E^c$ iff it is not in $E$.
\end{definition}

\begin{lemma}
\label{lemma:neighborhoor_in_a_digraph}
Given a digraph $(X,E)$, we have
the corresponding Alexandroff closure $c_E$ (\cref{def:functor_from_digraphs_to_closure_spaces}),
the reverse digraph $E^{\transpose}$ (\cref{def:reverse-digraph}),
and the function $\rho_{E^{\transpose}}:X \to \Pow{X}$ given by \eqref{eq:rho}.
  Let $x\in X$. Then, the singleton
  $\{\rho_{E^{\transpose}}(x)\}$
  is a
  base for \new{the} neighborhood filter
  at $x$ (\cref{def:neighborhood}) in the closure space $(X,c_E)$.
\end{lemma}

\begin{proof}
  Let $x \in X$.
  Let $i_E$ denote the interior corresponding to the closure $c_E$.
  By definition and because the closure operator $c_E$ is Alexandroff, we have 
\begin{multline*}
  i_E(\rho_{E^{\transpose}}(x))=X\setminus c_E(X\setminus \rho_{E^{\transpose}}(x))
  = X \setminus  c_E(X \setminus  (\{y\in X \,|\, y\overline{E}x\}) \\
  = X \setminus  c_E(\{y \in X\,|\, y\overline{E}^cx\})
  =X\setminus \bigcup\limits_{y\in X,y\overline{E}^c x}\rho_E(y).
\end{multline*}
For all $y \in X$,
if $y\overline{E}^cx$ then $x \not\in \rho_E(y)$.
Therefore $x \in i_E(\rho_{E^{\transpose}}(x))$.
Thus, $\rho_{E^{\transpose}}(x)$ is a neighborhood of $\{x\}$. 

\new{Now let $U$ be a neighborhood of $x$.} Then $x\in i_E(U)=X\setminus  c_E(X\setminus  U)$. Suppose $y \in X \setminus U$. Then $c_E(y)\subset c_E(X\setminus U)$. Thus $x\not\in c_E(y)$
and hence $y \not\in c_{E^{\transpose}}(x) = \rho_{E^{\transpose}}(x)$.
That is, $y \in X \setminus \rho_{E^{\transpose}}(x)$.
Therefore $\rho_{E^{\transpose}}(x)\subset U$.
Hence $X \setminus U \subset X \setminus \rho_{E^{\transpose}}(x)$ and thus
$\rho_{E^{\transpose}}(x)$ is a base for a neighborhood filter at $x$.
\end{proof}

\subsection{Graph products as special cases of closure space products}
\label{sec:graph-products-proofs}

We will now prove that the \new{digraph product} and the \new{digraph} cartesian product are special cases of the \new{categorical} product of closure spaces and the inductive product of closure spaces, respectively.

\begin{remark}
\new{
The digraph product and the digraph cartesian product are two distinct product operations on digraphs. Given digraphs $(X,E)$ and $(Y,E)$, both products are digraph structures on the cartesian product of sets $X\times Y$. The digraph product includes ``diagonal" edges, whereas the digraph cartesian product does not.}
\end{remark}

\begin{definition}
\label{def:regular_digraph_product}
Let $(X,E)$ and $(Y,F)$ be two digraphs. Define the \emph{digraph product} $X\times Y$ to be the digraph $(X\times Y,E \times F)$,
where $(x,y)\overline{E\times F}(x',y')$ iff $x\overline{E}_Xx'$ and $y\overline{E}_Yy'$.
\end{definition}

\begin{proposition}
\label{prop:regular_digraph_product_is_product}
Let $(X,E)$ and $(Y,F)$ be digraphs. Then $(X\times Y,c_{E\times F})=(X\times Y,c_{E}\times c_{F})$ (\cref{def:product_closure}).
\end{proposition}

\begin{proof}
 We will show that both closures share a base for a neighborhood filter (\cref{def:neighborhood}) at each point.
  It follows by \cref{cor:nghd-base} that they are equal.

  Let $x \in X$ and $y \in Y$.
  By \cref{lemma:neighborhoor_in_a_digraph}, $c_{E}$ has a base for a neighborhood filter $\{\rho_{E^{\transpose}}(x)\}$ at $x$,
  $c_{F}$ has a base for a neighborhood filter $\{\rho_{F^{\transpose}}(y)\}$ at $y$, and
  $c_{E \times F}$ has a base for a neighborhood filter $\{\rho_{(E \times F)^{\transpose}}(x,y)\}$ at $(x,y)$.
  By \cref{def:product_closure}, $c_{E}\times c_{F}$ has a base for a neighborhood filter $\{\rho_{E^{\transpose}}(x)\times\rho_{F^{\transpose}}(y)\}$ at $(x,y)$.
  By \cref{eq:rho} and \cref{def:regular_digraph_product},
  $\rho_{E^{\transpose}}(x) = \{x' \in X \,|\, x'\overline{E}_Xx\}$,
  $\rho_{F^{\transpose}}(y) = \{y' \in Y \,|\, y'\overline{E}_Yy\}$, and
  $\rho_{(E \times F)^{\transpose}}(x,y) = \{(x',y') \in X \times Y \,|\, x'\overline{E}_Xx \text{ and } y'\overline{E}_Yy\}$. 
  Therefore $\rho_{(E \times F)^{\transpose}} = \rho_{E^{\transpose}} \times \rho_{F^{\transpose}}$.
\end{proof}

Combining \cref{lemma:digraphs_are_quasi_discrete_closure_spaces} and \cref{prop:regular_digraph_product_is_product} we have the following.

\begin{corollary}
\label{corollary:prod_of_quasi_dis_is_quasi_dis}
Let $(X,\new{c_X})$ and $(Y,\new{c_Y})$ be two Alexandroff closure spaces. Then $(X\times Y,\new{c_X}\times \new{c_Y})$ is also Alexandroff.
\end{corollary}

\begin{definition}{\cite[Definition 2.3]{grigor2014homotopy}}
\label{def:digraph_product}
Let $(X,E)$ and $(Y,F)$ be two digraphs. Define the \emph{\new{digraph} cartesian product} $X\boxplus Y$
to be the digraph $(X\times Y,E \boxplus F)$,
where
for $x,x'\in X$ and $y,y'\in Y$, $(x,y)(E \boxplus F)(x',y')$ if and only if either $x=x'$ and $yFy'$, or $xEx'$ and $y=y'$.
Equivalently, $(x,y)\overline{E \boxplus F}(x',y')$ iff $x\overline{E}_Xx'$ and $y\overline{E}_Yy'$ and either $x=x'$ or $y=y'$.
\end{definition}

\begin{proposition}
\label{prop:digraph_product_is_inductive_product}
Let $(X,E)$ and $(Y,F)$ be digraphs. Then $(X\times Y,c_{E\boxplus F})=(X\times Y,
c_{E}\boxplus c_{F}
)$ (\cref{def:inductive_product_closure}).
\end{proposition}

\begin{proof}
  We will show that both closures share a base for a neighborhood filter (\cref{def:neighborhood}) at each point.
  It follows by \cref{cor:nghd-base} that they are equal.

  Let $x \in X$ and $y \in Y$.
  By \cref{lemma:neighborhoor_in_a_digraph}, $c_{E}$ has a base for a neighborhood filter $\{\rho_{E^{\transpose}}(x)\}$ at $x$,
  $c_{F}$ has a base for a neighborhood filter $\{\rho_{F^{\transpose}}(y)\}$ at $y$, and
  $c_{E \boxplus F}$ has a base for a neighborhood filter $\{\rho_{(E \boxplus F)^{\transpose}}(x,y)\}$ at $(x,y)$.
  By \cref{def:inductive_product_closure}, $c_{E}\boxplus c_{F}$ has a base for a neighborhood filter $\{(x',y') \in \rho_{E^{\transpose}}(x)\times\rho_{F^{\transpose}}(y) \,|\, x'=x \text{ or } y'=y\}$ at $(x,y)$.
  By \cref{eq:rho} and \cref{def:digraph_product},
  $\rho_{E^{\transpose}}(x) = \{x'\in X \,|\, x'\overline{E}_Xx\}$,
  $\rho_{F^{\transpose}}(y) = \{y'\in Y  \,|\, y'\overline{E}_Yy\}$, and\new{
  \begin{align*}
  \rho_{(E \boxplus F)^{\transpose}}(x,y) &= \{(x',y') \in X \times Y \,|\, x'\overline{E}_Xx \text{ and } y'\overline{E}_Yy \text{ and either } x'=x \text{ or } y'=y\}\\
  &=\{(x',y')\in \rho_{E^T}(x)\times \rho_{F^T}(y)\,|\, x'=x \text{ or } y'=y  \}.
  \end{align*} 
  Therefore the neighborhood filters are the same at each point, and the claim follows.}
\end{proof}

Combining \cref{lemma:digraphs_are_quasi_discrete_closure_spaces} and \cref{prop:digraph_product_is_inductive_product} we have the following.

\begin{corollary}
\label{corollary:ind_prod_of_quasi_dis_is_quasi_dis}
Let $(X,\new{c_X})$ and $(Y,\new{c_Y})$ be two Alexandroff closure spaces. Then $(X\times Y,\new{c_X}\boxplus \new{c_Y})$ is also Alexandroff.
\end{corollary}

\begin{corollary}
 Given two symmetric Alexandroff closure spaces, their product and their \new{inductive} product are also symmetric and Alexandroff.
\end{corollary}

\section{Closures induced by metrics} 
\label{section:metric_spaces_and_closures}

In this section, we consider closure operators induced by a metric.
Our closure operators will be indexed by the set $[0,\infty) \times \{-1,0,1\}$, which we order by the lexicographic order. That is, $(\eps,a) \leq (\eps',a')$ if $\eps < \eps'$ or $\eps = \eps'$ and $a \leq a'$.
For $\eps \geq 0$ we denote $(\eps,-1)$, $(\eps,0)$, and $(\eps,1)$ by $\eps^-$, $\eps$, and $\eps^+$, respectively.

\subsection{Closures from thickenings}
\label{sec:clos-from-thick}

We introduce various `thickening' closures on a metric space $(X,d)$ and examine some properties of these closure operators. In particular, we investigate how they interact with Lipschitz maps between metric spaces, and we determine a
 base of the neighborhood filter at each point.

\begin{definition}
\label{def:closure_induced_by_a_metric}
Let $(X,d)$ be a metric space and let $A \subset X$.
For 
$\eps\ge 0$
define
\begin{align*}
  c_{\eps^-,d}(A)
  =\bigcup\limits_{x\in A}B_{\eps}(x)\,
  &&
  c_{\eps,d}(A)
     =\bigcup\limits_{x\in A}\overline{B}_{\eps}(x),
  &&
  c_{\eps^+,d}(A)
     =\{x\in X\,|\,d(x,A)\le \eps\},
\end{align*}
where
$B_{\eps}(x) = \{y\in X\,|\, d(x,y)<\eps\}$ for $\eps > 0$,
  $B_0(x) = \{x\}$,
$\overline{B}_{\eps}(x) = \{y\in X\,|\, d(x,y)\le \eps\}$,
and
$d(x,A)=\inf_{y\in A}d(x,y)$.
For all $\eps \geq 0$ and $a \in \{-1,0,1\}$, $c_{(\eps,a),d}$ is a closure operator on $X$.
If the metric $d$ is clear from the context then we will denote the closure spaces
$(X,c_{(\eps,a),d})$ by $(X,c_{(\eps,a)})$.
We will sometimes refer to these as \emph{metric closures}.
\end{definition}

Let $(X,d)$ be a metric space. Then $(X,c_{0^+})$ is the
Kuratowski closure space whose corresponding topology is the one induced by
the metric $d$.
Also $(X,c_{0^-})$ = $(X,c_0)$ and this closure is the discrete closure.
If the metric $d$ takes only integer values,
then for all $n$, $c_{(n+1)^-,d} = c_{n,d} = c_{n^+,d}$.

\begin{example}
\label{def:metric_interval_objects}
Consider $([0,1],d)$, the unit interval with $d(x,y) = \abs{x-y}$.
For $(\eps,a)\in [0,1] \times \{-1,0,1\}$,
we will denote the closure space $([0,1],c_{(\eps,a)})$ by $I_{(\eps,a)}$.
Note that $I_{0^-} = I_{0}$ is the unit interval with the discrete closure, $I_{\bot}$, and $I_{0^+}$ is the unit interval with the closure corresponding to the standard topology, $I_{\tau}$. At the other extreme, note that $I_{1} = I_{1^+}$ is the unit interval with the indiscrete closure, $I_{\top}$.
\end{example}

Let $(X,d)$ be a metric space. Observe that for $\eps \ge 0$
the closure space $(X,c_{\eps^+})$ is symmetric and the closure spaces $(X,c_{\eps^-})$ and $(X,c_{\eps})$ are Alexandroff and symmetric (\cref{def:quasi_discrete_closure}).
Furthermore, recall the functor $\qd:\mathbf{Cl}\to \qdcl$ from \cref{prop:adjunction-cl-clqd}. It follows from the definitions of $c_{\eps^+}$ and $c_{\eps}$ that given a metric space $(X,d)$, $\qd(X,c_{\eps^+,d})=(X,c_{\eps,d})$.
From \cref{def:closure_induced_by_a_metric}, it is easy to see the following.

\begin{lemma}
\label{lemma:order_between_closures_induced by a metric}
Let $(X,d)$ be a metric space and let
$(\eps,a) \leq (\eps',a') \in [0,\infty) \times \{-1,0,1\}$.
Then $c_{(\eps,a)} \leq c_{(\eps',a')}$.
\end{lemma}

\begin{lemma}
  Let $(X,d)$ be a metric space and let $A \subset X$.
  If $r \geq 0$
  then $c_{r^+}(A)
  = \bigcap_{s > r} c_{s^-}(A)
  = \bigcap_{s > r} c_{s}(A)
  = \bigcap_{s > r} c_{s^+}(A)$.
  If $r > 0$
  then $c_{r^-}(A)
  = \bigcup_{s < r} c_{s^-}(A)
  = \bigcup_{s < r} c_{s}(A)
  = \bigcup_{s < r} c_{s^+}(A)$.
\end{lemma}

\begin{proof}
  By \cref{lemma:order_between_closures_induced by a metric}, we have that
  for all $r \geq 0$,
  $c_{r^+}(A)
  \subset \bigcap_{s > r} c_{s^-}(A) 
  \subset \bigcap_{s > r} c_{s}(A) 
  \subset \bigcap_{s > r} c_{s^+}(A)$, and
  for all $r > 0$,
  $\bigcup_{s > r} c_{s^-}(A) \subset
  \bigcup_{s > r} c_{s}(A) \subset
  \bigcup_{s > r} c_{s^+}(A) \subset
  c_{r^-}(A)$.
  It remains to prove that for $r \geq 0$, $\bigcap_{s > r} c_{s^+}(A) \subset c_{r^+}(A)$ and for $r > 0$,
  $c_{r^-}(A) \subset \bigcup_{s < r} c_{s^-}(A)$.

  If $x \in \bigcap_{s > r} c_{s^+}(A)$ then $d(x,A) \leq s$ for all $s > r$.
  Thus $d(x,A) \leq r$ and hence $x \in c_{r^+}(A)$.
  Finally, $\bigcup_{s<r}c_{s^-}(A)
  = \bigcup_{s < r} \bigcup_{x\in A} B_s(x) 
  = \bigcup_{x \in A} \bigcup_{s < r} B_s(x) 
  = \bigcup_{x \in A} B_r(x)
  = c_{r^-}(A)$.
\end{proof}

Using the triangle inequality one obtains the following.

\begin{lemma}
\label{lemma:composition_of_closures_induced_by_a_metric}
Let $(X,d)$ be a metric space and let $\eps,\delta\ge 0$. Then $c_{\eps^-}(c_{\delta^-})\le c_{(\eps+\delta)^-},c_{\eps}(c_{\delta})\le c_{\eps+\delta}$ and $c_{\eps^+}(c_{\delta^+})\le c_{(\eps+\delta)^+}$.
\end{lemma}

Let $(X,d)$ and $(Y,e)$ be metric spaces
and let $r \geq 0$. A map $f:X \to Y$ is said to be $r$-Lipschitz if
for $x,x' \in X$, $e(fx,fx') \leq r d(x,x')$.
Maps that are $1$-Lipschitz are also
said to be nonexpansive or short.
Let $\cat{Met}$ denote the category of metric spaces and $1$-Lipschitz maps. It is easy to check the following.

\begin{lemma} \label{lem:metric-map-to-closure-map}
Let $\eps \in [0,\infty) \times \{-1,0,1\}$ and $f:(X,d) \to (Y,e)$ be a $1$-Lipschitz map. Then $f:(X,c_{\eps,d}) \to (Y,c_{\eps,e})$ is a continuous map. Thus, for each each $\eps \in [0,\infty) \times \{-1,0,1\}$ we have a functor $\cat{Met} \to \cat{Cl_s}$.
\end{lemma}

\subsection{Neighborhood filter bases for metric closures}

It will be useful to have a neighborhood filter base for our metric closures.

\begin{lemma}
\label{lemma:interior_of_a_closed_ball}
Let $(X,d)$ be a metric space, $x \in X,$ and $\eps\ge 0$.
\begin{enumerate}
\item
  The singleton $\{B_{\eps}(x)\}$ is a base for a neighborhood filter  at $x$ in $(X,c_{\eps^-})$.
\new{\textnormal{(}}Recall that $B_0(x) = \{x\}$.\new{\textnormal{)}}
\item
  The singleton $\{\overline{B}_{\eps}(x)\}$ is a base for a neighborhood filter  at $x$ in $(X,c_{\eps})$.
\item 
  The collection $\{\overline{B}_{\eps+\delta}(x)\}_{\delta>0}$ is a base for a neighborhood filter  at $x$ in $(X,c_{\eps^+})$.
\end{enumerate}
\end{lemma}

\begin{proof}[Proof of \cref{lemma:interior_of_a_closed_ball}]
  We will prove the third case. The other cases are similar.
  Let $i_{\eps^+}$ denote the interior corresponding to the closure $c_{\eps^+}$.
  Let $\delta > 0$. First we verify that $\overline{B}_{\eps+\delta}(x)$ is a neighborhood of $x$ in $(X,c_{\eps^+})$.
  We have
\begin{multline*}
i_{\eps^+}(\overline{B}_{\eps+\delta}(x))=X\setminus c_{\eps^+}(X\setminus \overline{B}_{\eps+\delta}(x))
  =X\setminus \{y\in X\,|\,d(y,X\setminus\overline{B}_{\eps+\delta}(x))\le \eps\}\\
=\{y\in X\,|\, d(y,X\setminus \overline{B}_{\eps+\delta}(x))>\eps\}.
\end{multline*} 
Now observe that 
$
d(x,X\setminus \overline{B}_{\eps+\delta}(x))=\inf\limits_{y\in X-\overline{B}_{\eps+\delta}(x)} d(x,y). 
$
However, note that for all $y\in X\setminus \overline{B}_{\eps+\delta}(x)$, $d(x,y)>\eps+\delta$. Since $\delta>0$, it follows that 
$
\inf\limits_{y\in X\setminus \overline{B}_{\eps+\delta}(x)} d(x,y)>\eps.
$
Therefore $x\in i_{\eps^+}(\overline{B}_{\eps+\delta}(x))$.

Now suppose $A\subset X$ is a neighborhood of $x$ in $(X,c_{\eps^+})$. By definition we have that $x\in i_{\eps^+}(A)=X\setminus c_{\eps^+}(X\setminus A)=\{y\in X\,|\, d(y,X\setminus A)>\eps\}$. Thus 
\begin{equation*}
d(x,X\setminus A)=\inf\limits_{y\in X\setminus A}d(x,y)>\eps.
\end{equation*} 
Therefore there exists an $\delta>0$ such that $\forall y\in X\setminus A$, $d(x,y)\ge \eps+2\delta$. Hence if $y\in X$ is such that $d(x,y)\le \eps+\delta$, the previous inequality implies that $y\in A$. Thus $\overline{B}_{\eps+\delta}(x)\subset A$.
\end{proof} 

\subsection{Lipschitz maps and metric closures}

For Lipschitz maps we obtain continuous maps between appropriate metric closure spaces. In some cases we also have the converse.

First we have a slight generalization of \cref{lem:metric-map-to-closure-map}\new{, which was also observed in \cite[Proposition 3.9]{rieser2021vcech} for the case $a=1$}.

\begin{lemma}
\label{lemma:characterizing_Lipschitz_maps_of_closure_spaces}
Let $(X,d), (Y,e) \in \cat{Met}$ and let $r \geq 0$.
Let $\eps  \in [0,\infty)$ and $a \in \{-1,0,1\}$.
If $f:X \to Y$ is $r$-Lipschitz
then we have a continuous map
$f: (X,c_{(\eps,a),d}) \to (Y,c_{(r\eps,a),e})$.
\end{lemma}

A continuous map $f:(X,c_{1,d})\to (Y,c_{r,e})$ need not be $r$-Lipschitz as the following example shows.

\begin{example}
\label{example:continuous_generalizes_Lipschitz}
Let $X=\{x_1,x_2\}$ be a two point metric space with distance $d(x_1,x_2)=2$. Let $Y=\{y_1,y_2\}$ be a two point metric space with distance $e(y_1,y_2)=5$. Let $f:X\to Y$ be a map of sets defined by $f(x_i)=y_i$ for $i=1,2$. Then observe that $f:(X,c_{1,d})\to (Y,c_{2^+,e})$ is continuous. However, $f$ is not a $2$-Lipschitz map. Indeed, $5=e(y_1,y_2)\ge 2d(x_1,x_2)=4$.
\end{example}

For a converse to \cref{lemma:characterizing_Lipschitz_maps_of_closure_spaces}, we need a stronger hypothesis.

\begin{lemma}
  \label{lem:lipschitz-converse}
Let $(X,d), (Y,e) \in \cat{Met}$.
Assume that $d$ takes only integer values and that
for all $n$, $c_{1,d}^n=c_{n,d}$.
Then $f:(X,c_{1,d})\to (Y,c_{r,e})$ is continuous if and only if $f:X\to Y$ is $r$-Lipschitz.
\end{lemma}

\begin{proof}
Assume that $\im(d) \subset \Z_{\geq 0}$ and that for all $n$, $c_{1,d}^n = c_{n,d}$.
Consider a continuous map $f:(X,c_{1,d}) \to (Y,c_{r,e})$.
Let $x,y \in X$ and let $m = d(x,y)$.
Then $y \in c_{m,d}(x)$ and thus $f(y) \in f(c_{m,d}(x)) = f(c_{1,d}^m(x)) = f(c_{1,d}(c_{1,d}^{m-1}(x)) \subset c_{r,e}(f(c_{1,d}^{m-1}(x))$.
By induction, $f(y) \in c_{r,e}^m(f(x))$,
which by the triangle inequality is contained in $c_{mr,e}(f(x))$.
Hence $e(f(x),f(y)) \leq mr$ and  
therefore $f:X \to Y$ is $r$-Lipschitz.
\end{proof}

The hypotheses of \cref{lem:lipschitz-converse} hold for the following examples.

\begin{example}
\label{example:generalizing_Lipschitz_maps}
Let $n,m\in \N$. 
\begin{itemize}
\item 
 The set of $n$ elements with pairwise distance $1$.
\item 
$\{0,1,\dots, n\}$, $\Z_{\geq 0}$, or $\Z$ with the absolute value metric.
\item $\{0,1,\ldots,n\}^m$, $\Z_{\geq 0}^m$ or $\Z^m$ with the metric induced by the $1$-norm or the $\infty$-norm.
\end{itemize}
\end{example}

\subsection{Products of metric closures}

The products and inductive products of metric closures
are related to the closures associated to two canonical product metrics.

\begin{definition}
 Let $(X,d)$ and $(Y,e)$ be metric spaces.
  Define the metrics $d+e$ and $d\vee e$ on $X \times Y$ by
  \begin{align*}
    (d+e)((x,y),(x',y')) &= d(x,x') + e(y,y') \\
    (d \vee e)((x,y),(x',y')) &= \max( d(x,x'), e(y,y'))
  \end{align*}
\end{definition}

\begin{theorem}
\label{theorem:prod_closure_as_closure_induced_by_sup_metric}
Let $(X,d)$ and $(Y,e)$ be metric spaces. Let $\eps
\ge 0$, $a\in \{-1,0,1\}$ 
Then $(X\times Y,c_{{(\eps,a)},d}\times c_{{(\eps,a)},e})=(X\times Y,c_{{(\eps,a)},d \vee e})$. 
\end{theorem}

\begin{proof}
We prove the case for $a=1$. The other cases are similar.
First we show that $c_{\eps^+,d}\times c_{\eps^+,e}$ is coarser than $c_{\eps^+,d \vee e}$.  We proceed by showing the projection $\pi_X:(X\times Y, c_{\eps^+,d \vee e})\to (X,c_{\eps^+,d})$ is continuous. Let $A\subset X\times Y$. Let $x\in \pi_X(c_{\eps^+,d \vee e}(A))$. Thus, there is a $y\in Y$ such that $(x,y)\in c_{\eps^+,d \vee e}(A)$. Thus, by definition we have:
\begin{equation*}
d_{d \vee e}((x,y),A)=\inf\limits_{(x',y')\in A}d_{d \vee e}((x,y),(x',y'))\le \eps
\end{equation*}
It follows that $\inf_{x'\in \pi_X(A)}d(x,x')\le \eps$. Thus $x\in c_{\eps^+,d}(\pi_X(A))$. Therefore
$\pi_X$ is continuous. Similarly, $\pi_Y:(X\times Y, c_{\eps^+,d \vee e})\to (Y,c_{\eps^+,e})$ is continuous. By \cref{prop:product_closure}, the product closure is the coarsest closure so that each projection map is continuous. Thus, 
$c_{\eps^+,d}\times c_{\eps^+,e}$ is coarser than $c_{\eps^+,d \vee e}$.

Now we show that $c_{\eps^+,d \vee e}$ is coarser than $c_{\eps^+,d}\times c_{\eps^+,e}$.
Let $A\subset X\times Y$.
Now let $(x,y)\in c_{\eps^+,d}\times c_{\eps^+,e}(A)$.
Let $\delta>0$.
Then by \cref{def:product_closure}, \cref{cor:nghd-base}, and
   \cref{lemma:interior_of_a_closed_ball}
 $\pi_X^{-1}(\overline{B}_{\eps+\delta,d}(x))\cap \pi_Y^{-1}(\overline{B}_{\eps+\delta,e}(y))\cap A\neq \varnothing$. Observe that $\pi_X^{-1}(\overline{B}_{\eps+\delta,d}(x))\cap \pi_Y^{-1}(\overline{B}_{\eps+\delta,e}(y))=\overline{B}_{\eps+\delta,d \vee e}(x,y)$. Thus, $d_{d \vee e}((x,y),A)\le \eps$.
 Therefore, $(x,y)\in c_{\eps^+,d \vee e}(A)$ and thus $c_{\eps^+,d}\times c_{\eps^+,e}(A)\subset c_{\eps^+,d \vee e}(A)$.
\end{proof}

\begin{theorem}
\label{theorem:ind_prod_closure_as_closure_induced_by_+_metric}
Let $(X,d)$ and $(Y,e)$ be metric spaces and let $\eps\ge 0$, $a=\{-1,0,1\}$.
Then for $X \times Y$, $c_{{(\eps,a)},d+e}$ is coarser than $c_{{(\eps,a)},d}\boxplus c_{{(\eps,a)},e}$ \new{\textnormal{(}}\cref{def:inductive_product_closure}\new{\textnormal{)}}.  If 
$e$ takes only integer values
then $(X\times Y,c_{1,d+e})=(X\times Y,c_{1,d}\boxplus c_{1,e})$.
\end{theorem}

\begin{proof}
We prove the first statement for the case $a=1$. The other cases are similar.
To show that $c_{\eps^+,d+e}$ is coarser than $c_{\eps^+,d}\boxplus c_{\eps^+,e}$ by \cref{theorem:characterization_of_ind_prod_closure} it is sufficient to show that for each $x\in X$ and each $y\in Y$ the functions $f_y:(X,c_{\eps^+,d})\to (X\times Y,c_{\eps^+,d+e})$ and $f_x:(Y,c_{\eps^+,e})\to (X\times Y,c_{\eps^+,d+e})$ defined by $f_y(x')=(x',y)$ and $f_x(y')=(x,y')$ are continuous. Let $y\in Y$ and consider the mapping $f_y$. Let $A\subset X$. Let $(x,y)\in f_y(c_{\eps^+,d}(A))$. Then $x\in c_{\eps^+,d}(A)$, therefore $d_{d}(x,A)\le \eps$. Thus, by definition we have that $(d+e)((x,y),A\times \{y\})\le \eps$. Therefore, $(x,y)\in c_{\eps^+,d+e}(f_y(A))$. Thus $f_y(c_{\eps^+,d}(A))\subset c_{\eps^+,d+e}(f_y(A))$ and hence $f_y$ is continuous. Similarly, given $x\in X$, $f_x$ is continuous. Therefore $c_{\eps^+,d+e}$ is coarser than $c_{\eps^+,d}\boxplus c_{\eps^+,e}$.

Now assume that $\im e\subset \Z_{\geq 0}$.
It remains
to show that $c_{1,d}\boxplus c_{1,e}$ is coarser than $c_{1,d+e}$. Let $A\subset X\times Y$ and let $(x,y)\in c_{1,d+e}(A)$. Thus, there exists a $(x',y')\in A$ such that $d(x,y) + e(y,y') \le 1$.
By \cref{def:inductive_product_closure,lemma:interior_of_a_closed_ball},
$(X\times Y,c_{1,d}\boxplus c_{1,e})$
has a base for a neighborhood filter at $(x,y)$ consisting of the element
$W = \{x\} \times \overline{B}_{1,e}(y) \cup \overline{B}_{1,d}(x) \times \{y\}$.
If $y\neq y'$ then, since $\text{im}\, e\subset \Z_{\geq 0}$ it follows that $e(y,y')=1$ and thus $x=x'$.
If $y=y'$ then $e(y,y')=0$ and thus $d(x,x')\le 1$.
In either case we have that $(x',y')\in
W$.
Therefore $W\cap A\neq \varnothing$
and thus
$(x,y)\in (c_{1,d}\boxplus c_{1,e})(A)$. Hence, it follows that $c_{1,d}\boxplus c_{1,e}$ is coarser than $c_{1,d+e}$.
\qedhere
\end{proof}

The next example shows that
we cannot remove the integer-value hypothesis.

\begin{example}
\label{example:inductive_prod_closure_isnt_always_ell_1_c1}
Let $(X,d)=(Y,e)=(\R,d)$ where $d$ is the absolute value metric on $\R$. Let $(0,0)\in \R^2$ and consider $c_{1,d+d}((0,0))$. Note that $(\frac{1}{2},\frac{1}{2})\in c_{1,d+d}((0,0))$. However, $(\frac{1}{2},\frac{1}{2})\not\in (c_{1,d}\boxplus c_{1,d})((0,0))$.
\end{example}

The final example of this section shows that $I_{\tau}\boxplus I_{\tau}$ is not the same as $I_{\tau}\times I_{\tau}$.

\begin{example}
\label{example:inductive_prod_closure_thickening_0}
Consider the unit interval $[0,1]$ with 
the absolute value metric, $d$.
The metric $d$ generates a topology on $[0,1]$ whose closure is $c_{0^+,d}$.
Now consider $[0,1] \times [0,1]$.
By \cref{theorem:prod_closure_as_closure_induced_by_sup_metric} we have $c_{0^+,d}\times c_{0^+,d}=c_{0^+,d \vee d}$.
Since the metrics $d+d$ and $d \vee d$ are equivalent they induce the same topology.
Hence, $c_{0^+,d \vee d}=c_{0^+,d+d}$.
Let $A=(0,1)\times (0,1)\subset [0,1]\times [0,1]$.
Thus $c_{0^+,d+d}(A)=[0,1]\times [0,1]$.
On the other hand, 
no neighborhood of the form $\overline{B}_{\eps}(0) \times \{0\}\cup \{0\}\times \overline{B}_{\delta}(0)$ for any $\eps,\delta>0$ has a non-empty intersection with $A$.
Thus, by
\cref{lemma:interior_of_a_closed_ball,def:inductive_product_closure} and \cref{cor:nghd-base},
$(0,0)\not \in c_{0^+}\boxplus c_{0^+}(A)$.
Therefore $c_{0^+,d+d}\neq c_{0^+,d}\boxplus c_{0^+,d}$.
\end{example}
  
\section{Homotopy in closure spaces}
\label{section:homotopy}

In this section we define various homotopy theories for closure spaces using different intervals together with either the product or the inductive product.
We will study the relationships between
these theories.
We will observe that
these homotopy theories restrict to the full subcategories of closure spaces, $\cat{Top}$, $\qdcl$, $\suqdcl$ where some of them have been previously studied under different names.

\subsection{Product operations}

We formalize the properties of the product and the inductive product that we will need.
Recall that the category of sets, $\cat{Set}$, together with the cartesian product and the one point set $*$ forms a symmetric monoidal category. In addition, there is a functor $U: \cat{Cl} \to \cat{Set}$ that forgets the closure. Also, the terminal object in $\cat{Cl}$ is the one point set $*$ with its unique closure.

\begin{definition}
  \label{def:product_operator}
A \emph{product operation} 
is a functor $\otimes:\mathbf{Cl}\times \mathbf{Cl}\to \mathbf{Cl}$
for which $\cat{Cl}$ has a symmetric monoidal category structure that commutes with the forgetful functor $U$ and the cartesian symmetric monoidal category structure on $\cat{Set}$.
That is, for all $(X,\new{c_X})$, $(Y,\new{c_Y})$, $(X,\new{c_X}) \otimes (Y,\new{c_Y}) = (X \times Y, \new{c_X} \otimes \new{c_Y})$ for some closure operator $\new{c_X} \otimes \new{c_Y}$ on $X \times Y$,
the unit object is the one point space $*$ with its unique closure, and
the associator, unitors, and braiding are given by those on $\cat{Set}$.
\end{definition}

\begin{example}
\label{example:products}
The product closure and inductive product closures are examples of product operations.
Indeed, by their definition, $\boxplus$ and $\times$ commute with the forgetful functor. 
The unit object for both $\boxplus$ and $\times$ is the one point space $*$.

The braiding isomorphisms $\gamma_{XY}:X\times Y \to Y\times X$ and $\gamma_{XY}: X\boxplus Y \to Y\boxplus X$ are both given by $\gamma_{XY}(x,y)=(y,x)$.
We will show that $\gamma_{XY}: X \times Y \to Y \times X$ is continuous. The proof that $\gamma_{XY}: X \boxplus Y \to Y \boxplus X$ is continuous is similar.
Let $A\subset X\times Y$ and let $(x,y)\in (\new{c_X} \times \new{c_Y})(A)$.
Let $U_x$ and $V_y$ be neighborhoods of $x$ and $y$, respectively.
  Then by definition, $(U_x \times V_y) \cap A \neq \varnothing$.
  It follows that $(V_y \times U_x) \cap \gamma_{XY}(A) \neq \varnothing$.
  Thus, by definition, $(y,x) \in (\new{c_Y} \times \new{c_X})(\gamma_{XY}(A))$.
  Hence $\gamma_{XY}((\new{c_X}\times \new{c_Y})(A)) \subset (\new{c_Y} \times \new{c_X})(\gamma_{XY}(A))$.
  Therefore $\gamma_{XY}$ is continuous.

  The continuity of the associator $a_{XYZ}:(X\times Y)\times Z\to X\times (Y\times Z)$ follows from observing that $(X\times Y)\times Z$ and $X\times (Y\times Z)$ share a base for a neighborhood filter at each $(x,y,z)$ consisting of sets of the form $(U_x\times U_y\times U_z)$ where $U_x,U_y,U_z$ are
  neighborhoods of
  $X$, $Y$ and $Z$ respectively.
  The continuity of the associator $a_{XYZ}:(X\boxplus Y)\boxplus Z \to X\boxplus (Y\boxplus Z)$ follows from observing that $(X\boxplus Y)\boxplus Z$ and $X\boxplus(Y\boxplus Z)$ share a
  base for a neighborhood filter  at each $(x,y,z)$ consisting of sets of the form $U_x\times\{y,z\}\cup \{x\}\times U_y\times\{z\}\cup \{x,y\}\times U_z$ where $U_x,U_y,U_z$ are
  neighborhoods of
  $x$, $y$ and $z$ respectively.

Let $A \subset X$. Since $(c_* \times c)(* \times A) = * \times c(A)$, the left unitor $\lambda_X: * \times X \to X$ is continuous. Similarly, the left unitor $\lambda_X: * \boxplus X \to X$ is continuous and both right unitors are continuous.
  The triangle and pentagon identities are satisfied for both $\times$ and $\boxplus$ since they are satisfied for the underlying sets and all the maps in question are continuous. 
\end{example}

\begin{definition} \label{def:partial-order-products}
  Define a partial order on product operations by setting $\otimes_1 \leq \otimes_2$ if there exists natural transformation $\alpha$ from $\otimes_1$ to $\otimes_2$ such that for all closure spaces $X$ and $Y$, $\alpha_{(X,Y)}:X \otimes_1 Y \to X \otimes_2 Y$ is the identity map.
\end{definition}

\begin{lemma}
\label{lemma:projection_maps_on_products}
Let $\otimes $ be a product operation
and let $(X,\new{c_X})$ and $(Y,\new{c_Y})$ be closure spaces. Then the projection maps $\pi_X:(X \times Y, \new{c_X} \otimes \new{c_Y}) \to (X,\new{c_X})$ and $\pi_Y: (X \times Y, \new{c_X} \otimes \new{c_Y}) \to (Y,\new{c_Y})$ are continuous.
\end{lemma}

\begin{proof}
  We verify the first case; the second is similar.
  Let $\cat{1}_X:X\to X$ be the identity and $a:Y\to *$ be the constant map. Then, by functoriality, we have that $\cat{1}_X\otimes a:X\otimes Y\to X\otimes *$ is continuous.
The result follows from composing with
the right unitor isomorphism, $X\otimes * \isomto X$.
\end{proof}

\begin{lemma}
\label{lemma:embedding_maps_into_product}
Let $\otimes$ be a product operation,
let $(X,\new{c_X}), (Y,\new{c_Y})$ be closure spaces and let $x_0 \in X$ and $y_0 \in Y$.
Then the maps $Y\to X\otimes Y$ given by $y \mapsto (x_0,y)$ and $X\to X\otimes Y$ given by $x \mapsto (x,y_0)$ are continuous. 
\end{lemma}

\begin{proof}
  Consider the continuous maps $x_0:*\to X$ given by $x_0(*)=x_0$
  and $\cat{1}_Y:Y\to Y$, the identity map.
  By functoriality, $x_0\otimes \cat{1}_Y:*\otimes Y\to X\otimes Y$ is continuous.
  By precomposing with the right unitor isomorphism $Y \isomto *\otimes Y$, we get that the map $Y \to X \otimes Y$ given by $y \mapsto (x_0,y)$ is continuous.
  The other case is similar.
\end{proof}

From \cref{theorem:characterization_of_ind_prod_closure,lemma:projection_maps_on_products,lemma:embedding_maps_into_product,prop:product_closure} we get the following.

\begin{proposition}
\label{corollary:regular_product_is_coarsest}
Let $\otimes$ be a product operation. Then $\boxplus\le\otimes \le \times$.
\end{proposition}

\begin{lemma} \label{lem:prod-cl-disc-cl} 
  Let $(X,\new{c_X}),(Y,\new{c_Y})$ be closure spaces and assume that $\new{c_Y}$ is the discrete closure.
  Then for any product operation $\otimes$, $(X \times Y, \new{c_X} \otimes \new{c_Y}) = (X \times Y, \new{c_X} \times \new{c_Y})$.
\end{lemma}

\begin{proof}
 Let $\otimes$ be a product operation.
  By \cref{corollary:regular_product_is_coarsest}, we have the following partial ordering on closures on $X \times Y$,
  $\new{c_X} \boxplus \new{c_Y} \leq \new{c_X} \otimes \new{c_Y} \leq \new{c_X} \times \new{c_Y}$.
  By \cref{lemma:ind_product_and_product_of_discrete_spaces}, $\new{c_X} \boxplus \new{c_Y} = \new{c_X} \times \new{c_Y}$.
  Therefore $\new{c_Y} \otimes \new{c_Y} = \new{c_X} \times \new{c_Y}$.
\end{proof}

\subsection{Intervals}
\label{sec:interval}

We formalize the structure of an interval which we will use to develop a homotopy theory and give examples.
Our definitions and terminology are closely related to those of Berger and Moerdijk~\cite{Berger:2006}.
Recall that the terminal object in $\cat{Cl}$ is the one point closure space $*$ with its unique closure.
Also recall that maps from a discrete closure space are always continuous.

\begin{definition}
  \label{def:interval_object}
  Let $\otimes$ be a product operation.
An \emph{interval}
for $\otimes$ 
is a closure space $J$ together with two continuous maps $*\xrightarrow{0}J\xleftarrow{1}*$
and a symmetric, associative continuous map $\vee: J \otimes J \to J$ which has $0$ as its \emph{neutral} element and $1$ as its \emph{absorbing} element.
That is, if we write $s \vee t$ for $\vee(s,t)$ then $0 \vee t = t$ and $1 \vee t = 1$.
A \emph{morphism of intervals} is a continuous map of intervals $f: J \to K$ that preserves the distinguished points and commutes with the symmetric, associative continuous maps. That is, $f 0_J = 0_K$, $f 1_J = 1_K$, and $f(s \vee_J t) = f(s) \vee_K f(t)$.
Let $\cat{Int}(\otimes)$ denote the category whose objects are intervals for $\otimes$ and whose morphisms are morphisms of intervals.
Let $\cat{Int_{0\neq 1}}(\otimes)$ be the full subcategory of $\cat{Int}(\otimes)$ whose objects are intervals with $0\neq 1$.
\end{definition}

\begin{example}
  Consider the discrete closure space $* \amalg *$, which we also write as $J_{\bot} = (\{0,1\},c_{\bot})$.
  We have continuous maps $0,1:* \to \{0,1\}$ with the specified image. By \cref{lem:prod-cl-disc-cl},
  $J_{\bot} \otimes J_{\bot} = (\{0,1)\} \times \{0,1\}, c_{\bot})$.
  Let $\vee: \{0,1\}^2 \to \{0,1\}$ be the continuous map given by $s \vee t = \max(s,t)$.
  With this structure, $* \amalg *$ is an interval for $\otimes$.
  Furthermore, using the universal property of the coproduct, for any interval $J$ for $\otimes$ there is a unique morphism of intervals $* \amalg * \to J$.
  That is, $* \amalg *$ is the initial object in $\cat{Int(\otimes)}$ and $\cat{Int_{0\neq 1}(\otimes)}$. 

  The one point space $*$ has unique continuous maps $0:* \to *$ and $1:* \to *$. It follows from the definition that $* \otimes * = *$ and thus we have a unique continuous map $\vee: * \to *$. With this structure $*$ is an interval for $\otimes$
  and furthermore it is the terminal object in $\cat{Int(\otimes)}$.
\end{example}

\begin{example} \label{ex:intervals}
  In each of the following, we show that the maximum map 
  gives us an interval with $0\neq 1$ 
  for $\times$. 
  By \cref{corollary:regular_product_is_coarsest} and \cref{lemma:interval_object_for_a_product} it will follow that these are also intervals with $0\neq 1$
  for any product operation $\otimes$.
  \begin{enumerate}
  \item $I_{\tau}$ with the inclusions of 
    $0$ and $1$ (\cref{def:interval_objects}) is an interval for $\times$.
    It is elementary to show that the map $I_{\tau} \times I_{\tau} \to I_{\tau}$ given by the maximum function is a continuous map of topological spaces.
  \item For $m\geq 1$ and $0 \leq k \leq 2^m-1$, each of $J_{m,\bot}$, $J_{m,\top}$, $J_m$, $J_{m,\leq}$, and $J_{m,k}$ with the inclusions of the points $0$ and $m$  (\cref{def:interval_objects}) and the maximum map is an interval for $\times$. We will show that the maximum map is continuous in each case.
  \begin{enumerate}
  \item Since $J_{m,\bot}\times J_{m,\bot }$ is a discrete space, the maximum map is continuous.
  \item Similarly, since $J_{m,\top}$ is an indiscrete space,
    the maximum map is continuous. 
  \item $J_m$.
    Suppose $(s',t') \in (c \times c)(s,t)$. Then $\abs{s-s'} \leq 1$ and $\abs{t-t'} \leq 1$. Therefore $|\max(s,t) - \max(s',t')| \leq 1$
      and thus $\max(s',t') \in c(\max(s,t))$. 
      
    \item $J_{m,\leq}$.
      Suppose that $(s',t') \in (c \times c)(s,t)$. Then $s \leq s'$ and $t \leq t'$.
        It follows that $\max(s,t) \leq \max(s',t')$ and thus $\max(s',t') \in c(\max(s,t))$.
    \item $J_{m,k}$.
      Suppose that $(s',t') \in (c \times c)(s,t)$. Then $s' \in c(s)$ and $t' \in c(t)$. If $\max(s,t)=s$ and $\max(s',t')=s'$ or $\max(s,t)=t$ and $\max(s',t') = t'$ then $\max(s',t') \in c(\max(s,t))$. If not, then $s'=t$ and $t'=s$ and we also have that  $\max(s',t') \in c(\max(s,t))$.
  \end{enumerate}  
  As special cases, we have the intervals $J_{\bot}, J_{\top}, J_+$, and $J_-$.
\item Let $\eps \in [0,1]$. Then $I_{\eps^-}$, $I_{\eps}$, and $I_{\eps^+}$
  with the inclusions of the points $0$ and $1$ (\cref{def:metric_interval_objects}) and the maximum map
    is an interval for $\times$. 
  We give the proof the case $I_{\eps^+}$. The other cases are similar.
    Let $A \subset [0,1] \times [0,1]$ and suppose that $(x,y) \in (c_{\eps^+} \times c_{\eps^+})(A)$.
        Then by \cref{lemma:interior_of_a_closed_ball}, for all $\delta_1,\delta_2 > 0$, $(\overline{B}_{\eps+\delta_1}(x) \times \overline{B}_{\eps+\delta_2}(y)) \cap A \neq \varnothing$.
        Thus for all $n \geq 1$, there exists $(x_n,y_n) \in A$ such that $|x-x_n| \leq \eps + \frac{1}{n}$ and $|y-y_n| \leq \eps + \frac{1}{n}$.
        It follows that $|\max(x,y)-\max(x_n,y_n)| \leq \eps + \frac{1}{n}$.
        Therefore $d(\max(x,y),\max(A)) \leq \eps$.
        Hence $\max(x,y) \in c_{\eps^+}(\max(A))$.
  \end{enumerate}
\end{example}

\begin{lemma} \label{lem:interval-constructions}
    Let $J$ and $K$ be intervals for $\otimes$. Then so are the following:
    \begin{enumerate}
    \item \label{it:prod} the \emph{product} $J \otimes K$,
    \item \label{it:coprod} the \emph{coproduct} $J \amalg K$, and
    \item \label{it:wedge} the \emph{wedge product}, given by the pushout (\cref{def:pushout})
    \begin{center}
      \begin{tikzcd}
        * \ar[r,"1"] \ar[d,"0"'] & J \ar[d,dashed,"\iota"]\\
        K \ar[r,dashed,"\kappa"'] & J \vee K
      \end{tikzcd}
    \end{center}
    which we will 
    also call the \emph{concatenation} of $J$ and $K$. We will denote the $m$-fold concatenation of $J$ with itself by $J^{\vee m}$.
  \end{enumerate}
  Furthermore, if $J$ and $K$ are intervals with $0\neq 1$ then so are $J\otimes K,J\amalg K$, and $J\vee K$.
\end{lemma}

\begin{proof}
  By assumption there are symmetric, associative continuous maps $\vee_J: J \otimes J \to J$ and $\vee_{K}:K\otimes K \to K$ which have $0_J$, $0_{K}$ and $1_J$, $1_{K}$ as their neutral elements and absorbing elements, respectively.

  \eqref{it:prod}
  Let $0_{J \otimes K} = (0_J,0_{K})$ and $1_{J \otimes K} = (1_J,1_{K})$.
  Let $\vee_{J \otimes K}$ be the continuous map $(J \otimes K) \otimes (J \otimes K) \isomto J \otimes J \otimes K \otimes K \xto{\vee_J \otimes \vee_K} J \otimes K$.
One may check that it is symmetric and associative and has  $0_{J \otimes K}$ as its neutral element and $1_{J \otimes K}$ as its absorbing element.

  \eqref{it:coprod}
  Let $0_{J \amalg K} = 0_J$ and $1_{J \amalg K} = 1_K$.
  Define $\vee_{J \amalg K}: (J \amalg K) \otimes (J \amalg K) \to J \amalg K$ by
\begin{equation} \label{eq:coprod-interval}
  s\vee_{J \amalg K}t =
  \begin{cases}
s\vee_Jt, & s,t\in J\\
s\vee_{K}t, & s,t\in K\\
t, & s\in J,t\in K\\
s, & s\in K,t\in J.
\end{cases}
\end{equation}
By \cref{lemma:projection_maps_on_products} and
\cref{theorem:pasting_lemma_for_closure_spaces},
$\vee_{J\amalg K}$ is continuous.
One may also check that it is symmetric and associative and has $0_J$ as its neutral element and $1_{K}$ as its absorbing element.

\eqref{it:wedge}
Let $0_{J \vee K} = \iota \circ 0_J$ and $1_{J \vee K} = \kappa \circ 1_{K}$.
There is a universal continuous map $J \amalg K \to J \vee K$. 
  One may check that the map
  \eqref{eq:coprod-interval}
  induces a well-defined continuous map
$\vee_{J\vee K}: (J \vee K) \otimes (J \vee K) \to J \vee K$.
It is symmetric and associative and has $0_{J \vee K}$ as its neutral element and $1_{J \vee K}$ as its absorbing element.

In each of these constructions, if $0_J\neq 1_J$ and $0_K\neq 1_K$, then the induced neutral and absorbing elements will be distinct as well.
\end{proof}

\begin{proposition} \label{prop:canonical-interval-morphisms}
Let $J,K$ be intervals for the product operation $\otimes$.
  There are canonical morphisms of intervals
  $J \amalg K \to J \vee K \to J \otimes K \to J,K$.
\end{proposition}

\begin{proof}
The first map is given by the universal continuous map.
    The second is obtained from \cref{lemma:embedding_maps_into_product} and the universal property of the pushout.
    The third is obtained from \cref{lemma:projection_maps_on_products}.
    One may check that each of these maps respects the inclusion of $0$ and $1$ and the symmetric associative map.
\end{proof}

By the universal property of the pushout, we have the following.

\begin{proposition}
The category $\cat{Int(\otimes)}$ together with tensor product given by concatenation and unit element given by the terminal interval $*$ is a monoidal category.
\end{proposition}

\subsection{Relations between  product operations and intervals}

We study relations between product operations and between intervals which will later give us relations between their corresponding homotopy theories.

\begin{lemma}
\label{lemma:interval_object_for_a_product}
Suppose $\otimes_1\le \otimes_2$. If $J$ is an interval for $\otimes_2$ then $J$ is also an interval for $\otimes_1$.
\end{lemma}

\begin{proof}
Let $\vee:J\otimes_2 J\to J$ be the associative morphism for $J$ an interval for $\otimes_2$. Precomposing $\vee$ with the set-theoretic identity map $J\otimes_1 J\to J\otimes_2 J$, which is continuous by the assumption that $\otimes_1\le \otimes_2$, we get an associative morphism $\vee:J\otimes_1 J\to J$.
\end{proof}

\begin{definition} \label{def:preorder-intervals}
  Define a preorder on intervals for $\otimes$ by setting $J \leq K$ if there exists a morphism of intervals $f: J \to K$.
\end{definition}

From \cref{prop:canonical-interval-morphisms}, we have the following.

\begin{corollary}
\label{corollary:order_among_segment_constructions}
 Let $J,K$ be intervals for $\otimes$. Then
  $J \amalg K \leq J \vee K \leq J \otimes K \leq J,K$.
\end{corollary}

\begin{lemma}
  \label{lemma:indiscrete_intervals_are_maximal}
  For any $K$ in $\cat{Int_{0 \neq 1}}(\otimes)$, $K\le J_{\top}$.
\end{lemma}

\begin{proof}
  Let $K$ be an interval with $0\neq 1$ for $\otimes$.
  Define a map $f:K\to J_{\top}$ by $f(t)=1$ if $t\neq 0_K$ and $f(0_K)=0$.
One may check that this is a morphism of intervals.
\end{proof}

We will also give examples of this preorder relation among many of the intervals in \cref{ex:intervals}.
  First we prove the following.

\begin{lemma} \label{lem:retract}
  Given $m \geq 0$, $J_m$ is a retract of $([0,m],c_{{1}})$.
\end{lemma}

\begin{proof}
  The inclusion $i: J_m \hookrightarrow [0,m]$ is continuous since the closure on $J_m$ is the restriction of $c_{{1}}$ to the subset $J_m$.
    In the other direction, let $r:[0,m] \to J_m$ be given by rounding up.
    That is,
    for $i=0,1,\ldots,m$, let $r(x)=i$ if $x \in [i-\frac{1}{2},i+\frac{1}{2})$.
    Let $A \subset [0,m]$.
    We want to show that $r(\overline{A}) \subset c(r(A))$.
    By definition,
    $i \in r(\overline{A})$ iff there exists $x \in \overline{A}$ such that $x \in [i-\frac{1}{2},i+\frac{1}{2})$.
    Also,
    $i \in c(r(A)$ iff there exists $x \in A$ such that $x \in [i-\frac{3}{2},i+\frac{3}{2})$.
    Therefore $r(\overline{A}) \subset c(r(A))$.
\end{proof}

\begin{example} \label{ex:partial-order}
\begin{enumerate}
\item \label{it:IJmkJm}
  Let $m \geq 1$ and $0 \leq k \leq 2^m-1$.
  Recall (\cref{lemma:comparison_of_intervals}) that we have continuous maps $I_{\tau} \xto{f} J_{m,k} \xto{\Id} J_m$.
  These maps respect the inclusion of the two distinguished points and also commute with taking maximums.
  Thus they are morphisms of intervals and we write $I_{\tau} \leq J_{m,k} \leq J_m$.
In particular, $I_{\tau} \leq J_+ \leq J_{\top}$.
\item \label{it:Ia} 
  Let $(\eps,a) \leq (\eps',a')$.
  By \cref{lemma:order_between_closures_induced by a metric}, the identity map gives a morphisms of intervals
  $I_{(\eps,a)} \to I_{(\eps',a')}$.
  Therefore $I_{(\eps,a)} \leq I_{(\eps',a')}$.
\item \label{it:Jn-to-Jm} Let $1 \leq m \leq n$.
  There is a morphism of intervals $J_n \to J_m$ given by $i \mapsto \min(i,m)$. Therefore $J_n \leq J_m$.
The same map gives morphisms of intervals $J_{n,\bot} \to J_{m,\bot}$, $J_{n,\top} \to J_{m,\top}$, and $J_{n,\leq} \to J_{m,\leq}$.
    Therefore $J_{n,\bot} \leq J_{m,\bot}$, $J_{n,\top} \leq J_{m,\top}$, and $J_{n,\leq} \leq J_{m,\leq}$.
    In particular, $J_m \leq J_{\top}$, $J_{m,\bot} \leq J_{\bot}$, $J_{m,\top} \leq J_{\top}$ and $J_{m,\leq} \leq J_{1,\leq} = J_+$.
  \item \label{it:JnlJmk} Let $1 \leq m \leq n$. Let $0 \leq k \leq 2^m-1$ and $0 \leq \ell \leq 2^n-1$, where the binary representation of $k$ is the first $m$ lowest order bits of the binary representation of $\ell$.
      Then the map $i \mapsto \min(i,m)$ gives a morphism on intervals $J_{n,\ell} \to J_{m,k}$. Therefore $J_{n,\ell} \leq J_{m,k}$.
      In particular, for any $m \geq 1$ and $0 \leq k \leq 2^n-1$, if $k$ is odd then $J_{m,k} \leq J_+$ and if $k$ is even then $J_{m,k} \leq J_-$.  
  \item \label{it:J1-to-Jm} Let $m \geq 1$. There is a morphism of intervals $J_{1,\leq} \to J_{m,\leq}$ given by $0 \mapsto 0$ and $1 \mapsto m$. The same map gives morphisms of intervals $J_{\bot} \to J_{m,\bot}$ and $J_{\top} \to J_{\top}$.
      Therefore $J_+ = J_{1,\leq} \leq J_{m,\leq}$, $J_{\bot} \leq J_{m,\bot}$, and $J_{\top} \leq J_{m,\top}$.
\item \label{it:IJ} Let $m \geq 1$.
  We have morphisms of intervals $i: J_m \to ([0,m],c_{{1}})$ and $r:([0,m],c_{{1}}) \to J_m$ (see \cref{lem:retract}). Furthermore, by rescaling we have $([0,m],c_{{1}}) \isom ([0,1],c_{{\frac{1}{m}}})$ and this homeomorphism is given by inverse morphisms of intervals.
  Thus we have morphisms of intervals $J_m \to I_{{\frac{1}{m}}}$ and $I_{{\frac{1}{m}}} \to J_m$. 
  Therefore $J_m \leq I_{{\frac{1}{m}}}$ and $I_{{\frac{1}{m}}} \leq J_m$.
\item \label{it:I0J1bot} Let $f:I_{\bot}\to J_{\bot}$ be the map of sets given by rounding up. 
  This map is continuous because $I_{\bot}$ is discrete and it respects the structure of the intervals. Therefore it is a morphism of intervals and $I_{\bot} \leq J_{\bot}$.
\end{enumerate} 
\end{example}

\subsection{Homotopy}

An interval and a product operation give rise to a homotopy.
For a product operation $\otimes$ and a closure space $X$, by \cref{lem:prod-cl-disc-cl}, there is a canonical isomorphism $X \otimes * \isom X$.

\begin{definition}
\label{def:homotopy}
Let $f,g:X \to Y \in \cat{Cl}$.
Let $\otimes$ be a product operation and let $J$ be an interval for $\otimes$.
An \emph{elementary $(J,\otimes)$ homotopy} from $f$ to $g$ is a morphism $H$ making the following diagram commute.
\begin{equation} \label{eq:homotopy}
\begin{tikzcd}
  X \otimes * \arrow[dr,"f"] \arrow[d,"\Id_X \otimes 0_{J}"'] \\
  X \otimes J \arrow[r,dashed,"H",pos=0.3] & Y \\
  X \otimes * \arrow[ur,"g"'] \arrow[u,"\Id_X \otimes 1_{J}"]
\end{tikzcd}
\end{equation}
Say that the ordered pair $(f,g)$ is \emph{one-step $(J,\otimes)$-homotopic}.
Note that the existence of an elementary homotopy from $f$ to $g$ does not imply the existence of an elementary homotopy from $g$ to $f$.
However,  for any $f:X \to Y \in \cat{Cl}$, we have an elementary $(J,\otimes)$ homotopy from $f$ to $f$ given by $f\pi_X$, where $\pi_X$ is the canonical continuous projection $\pi_X:X \otimes J \to X$.
Let $\sim_{(J,\otimes)}$ be the equivalence relation on the set $\cat{Cl}(X,Y)$ generated by elementary $(J,\otimes)$ homotopies.
It is given by zigzags of elementary $(J,\otimes)$ homotopies.
If $f\sim_{(J,\otimes)} g$, say $f$ and $g$ are \emph{$(J,\otimes)$-homotopic}.
Call the equivalence relation, $\sim_{(J,\otimes)}$, \emph{$(J,\otimes)$ homotopy}.
\end{definition}

\begin{lemma} \label{lem:interval-homotopy}
Let $\otimes$ be a product operation and let $J$ be an interval for $\otimes$. Then $\vee$ is an elementary $(J,\otimes)$ homotopy from the identity map on $J$ to the composite map $J \to * \xto{1} J$.
\end{lemma}
 
\begin{proof}
By definition, for all $s \in J$, $s \vee 0  = s$ and $s \vee 1 = 1$. The result follows.
\end{proof}

\begin{example} \label{ex:J1discrete}
  Let $\otimes$ be a product operation.
Recall that $J_{\bot}$ and $*$ are the initial and terminal objects, respectively, in $\cat{Int(\otimes)}$.
Since $J_{\bot} \isom * \amalg *$, $X \otimes J_{\bot} \isom X \amalg X$
and we may define $H= f \amalg g$. Therefore, any $f,g:X \to Y \in \cat{Cl}$
  are one-step $(J_{\bot},\otimes)$-homotopic.
Since $X \otimes * \isom X$, we have $H=f$ and $H=g$. It follows that $f \sim_{(*,\otimes)} g$ if and only if $f=g$.
\end{example}

\begin{lemma} 
\label{lemma:product_of_one-step_homotopy}
  Let $\otimes$ be a product operation and $J$ be an interval for $\otimes$.
  By \cref{lem:interval-constructions}, $J \otimes J$ is also an interval for $\otimes$.
  If $(f,g)$ are one-step $(J,\otimes)$-homotopic and $(h,k)$ are one-step $(J,\otimes)$-homotopic then $(f \otimes h, g \otimes k)$ are one-step $(J \otimes J, \otimes)$-homotopic.
\end{lemma}

\begin{proof}
  Let $H$ and $F$ be elementary $(J,\otimes)$ homotopies between $f$ and $g$ and $h$ and $k$, respectively.
  Then the following diagram commutes.
  \begin{center}
  \begin{tikzcd}
  X\otimes Z\otimes *\ar[d,"1_{X\otimes Z}\otimes 0_{J\otimes J}"']\ar[drr,"f\otimes h",bend left=10]\\
  X\otimes Z\otimes J\otimes J\ar[r,"\isom"]&X\otimes J\otimes Z\otimes J\ar[r,"H\otimes F"',pos=0.4] & Y\otimes W \\
  X\otimes Z\otimes *\ar[u,"1_{X\otimes Z}\otimes 1_{J\otimes J}"]\ar[urr,"g\otimes k"',bend right=10]
  \end{tikzcd}
\end{center}\qedhere
\end{proof}

\begin{corollary}
\label{corollary:product_of_homotopies}
 Let $\otimes$ be a product operation and $J$ be an interval for $\otimes$.
  If $f\sim_{(J,\otimes)}g$ and $h\sim_{(J,\otimes)}k$ then $f \otimes h\sim_{(J\otimes J,\otimes)} g \otimes k$.
\end{corollary}

\begin{proof}
By assumption, we have a zigzag of elementary homotopies from $f$ to $g$ and a zigzag of elementary homotopies from $h$ to $k$.
  By adding identity maps as needed, we may assume that these zigzags have the same length and have matching elementary homotopies in the same direction.
  The result then follows from applying \cref{lemma:product_of_one-step_homotopy} to each of the paired elementary homotopies.
\end{proof}

\begin{lemma}
\label{lemma:diagonal_map_into_product}
Let $X$ be a closure space. Then the diagonal map $\Delta:X \to X\times X$ is continuous.
\end{lemma}

\begin{proof}
Let $A \subset X$. Let $x \in c(A)$.
    Then, for each neighborhood $U$ of $x$, $U \cap A \neq \varnothing$. 
    It follows that $(U \times U) \cap \Delta(A) \neq \varnothing$.
    Therefore $(x,x) \in (c \times c)(\Delta(A))$.
\end{proof}

\begin{lemma}
\label{example:diagonal_map_is_an_interval_morphism}
Let $J$ be an interval for $\times$. Then $\Delta:J\to J\times J$ is a morphism of intervals. In particular, $J\le J\times J$.
\end{lemma}

\begin{proof}
By \cref{lemma:diagonal_map_into_product}, $\Delta$ is continuous.
  Furthermore $\Delta(0) = (0,0)$, $\Delta(1) = (1,1)$ and $\Delta(s \vee t) = \Delta(s) \vee \Delta(t)$.
\end{proof}

\begin{example}
\label{example:diagonal_map_into_inductive_product}
Let $X$ be a closure space. The diagonal map $\Delta: X \to X \boxplus X$ need not be continuous. Consider the following.
Let $X$ be $J_{\top}$ or $J_+$. Then $1 \in c(0)$ but $(1,1) \notin c(0,0)$.
 Let $X = I$. Consider $A = [0,1)$. Then $1 \in c(A)$ but $(1,1) \notin c(\Delta(A))$.
\end{example}

  We end this section with a characterization of one-step $(J_{\top},\times)$-homotopy.
  
\begin{proposition} \label{lem:J1times-homotopy}
Let $f,g:(X,\new{c_X}) \to (Y,\new{c_Y}) \in \cat{Cl}$.
  Then $f,g$ are one-step $(J_{\top},\times)$-homotopic iff for all $A \subset X$, $f(\new{c_X}(A)) \cup g(\new{c_X}(A)) \subset \new{c_Y}(f(A)) \cap \new{c_Y}(g(A))$.
\end{proposition}

\begin{proof}
  Let $\new{c_{\top}}$ denote the indiscrete closure on $J_{\top}$.
  
  $(\Rightarrow)$ Let $H:X \times J_{\top} \to Y$ be an elementary $(J_{\top},\times)$ homotopy from $f$ to $g$.
  For all $x \in X$, $H(x,0) = f(x)$ and $H(x,1) = g(x)$.
  Let $A \subset X$.
  Then $f(\new{c_X}(A)) \cup g(\new{c_X}(A)) = H(\new{c_X}(A) \times J_{\top}) = H((\new{c_X} \times \new{c_{\top}})(A \times 0)) \subset \new{c_Y}(H(A \times 0)) = \new{c_Y}(f(A))$.
  Similarly $f(\new{c_X}(A)) \cup g(\new{c_X}(A)) \subset \new{c_Y}(g(A))$.

$(\Leftarrow)$ Define $H: X \times J_{\top} \to Y$ by $H(x,0) = f(x)$ and $H(x,1) = g(x)$.
  Let $A \subset X \times J_{\top}$.
  Let $A_0 = \{(x,0) \in A\}$ and $A_1 = \{(x,1) \in A\}$.
  Then $H((\new{c_X} \times \new{c_{\top}})(A)) = H((\new{c_X} \times \new{c_{\top}})(A_0 \cup A_1)) =
  H((\new{c_X} \times \new{c_{\top}})(A_0) \cup (\new{c_X} \times \new{c_{\top}})(A_1)) = 
  H((\new{c_X} \times \new{c_{\top}})(A_0)) \cup H((\new{c_X} \times \new{c_{\top}})(A_1)) =
  f(\new{c_X}(A_0) \cup g(\new{c_X}(A_0) \cup f(\new{c_X}(A_1)) \cup g(\new{c_X}(A_1))
  \subset \new{c_Y}(f(A_0)) \cap \new{c_Y}(g(A_0)) \cup \new{c_Y}(f(A_1)) \cap \new{c_Y}(g(A_1))
  \subset \new{c_Y}(f(A_0)) \cup \new{c_Y}(g(A_1))
  = \new{c_Y}(H(A_0)) \cup \new{c_Y}(H(A_1)) = \new{c_Y}(H(A_0) \cup H(A_1)) = \new{c_Y}(H(A_0 \cup A_1))
  = \new{c_Y}(H(A))$.
  Therefore $H$ is continuous.
\end{proof}

  As a special case, consider $(X,E), (Y,F) \in \cat{DiGph}$. Recall that $f:X \to Y$ is a digraph homomorphism iff whenever $x\overline{E}x'$ we have that $fx\overline{F}fx'$.
  Note that $J_{\top}$ is the complete digraph on $\{0,1\}$ and $(x,i)\overline{E \times J_{\top}}(x',j)$ iff $x\overline{E}x'$.

\begin{lemma} \label{lem:digraph-one-step-j1-homotopy}
 Let $f,g: (X,E) \to (Y,F) \in \cat{DiGph}$.
  Then $f,g$ are one-step $(J_{\top},\times)$-homotopic iff whenever $x\overline{E}x'$, we have that $fx\overline{F}gx'$ and $gx\overline{F}fx'$.
\end{lemma}

\begin{proof}
A map $H:X \times J_{\top} \to Y$ with $H(x,0) = f(x)$ and $H(x,1) = g(x)$ is a digraph homomorphism iff whenever $x\overline{E}x'$ we have that for all $i,j \in J_{\top}$, $H(x,i)\overline{F}H(x',j)$.
\end{proof}
  
\subsection{Relations between homotopy equivalences}

We study relations between our homotopy theories.

\begin{lemma}
\label{lemma:general_homotopic}
Let $\otimes$ be a product operation.
Let $f,g: X \to Y \in \cat{Cl}$.
\begin{enumerate}
\item $f,g$ are $(I_{\tau},\otimes)$-homotopic iff $f,g$ are one-step $(I_{\tau},\otimes)$-homotopic.
\item \label{it:J_1} $f,g$ are $(J_{\top},\otimes)$-homotopic iff there exists $m \geq 1$ such that $f,g$ are one-step $(J_m,\otimes)$-homotopic.
\item \label{it:Jplus_Jmk} $f,g$ are $(J_+,\otimes)$-homotopic iff there exists $m \geq 1$ and $0 \leq k\leq 2^m-1$ such that $f,g$ are one-step $(J_{m,k},\otimes)$-homotopic.
\end{enumerate}
\end{lemma}

\begin{proof}
Suppose that $f,g$ are $(J_+,\otimes)$-homotopic.
  A $(J_+,\otimes)$ homotopy is obtained by the symmetric transitive closure of the elementary $(J_+,\otimes)$ homotopy.
  Thus for some $m \geq 1$,  
  there is a finite sequence $f=f_0,f_1,\dots, f_m=g$ of maps where consecutive maps $f_i,f_{i+1}$ or $f_{i+1},f_i$
  are one-step $(J_+,\otimes)$-homotopic.
  We may concatenate the homotopies to obtain a homotopy $H:X\otimes J_{m,k} \to Y$ between $f$ and $g$.
  The other cases are easier since the elementary homotopies are symmetric.
  Note that $I_{\tau}^{\vee m}$ is homeomorphic to $I_{\tau}$.
\end{proof}

\begin{proposition} \label{prop:partial-order}
    Let $\otimes_1$, $\otimes_2$ be  product operations with $\otimes_1 \leq \otimes_2$ and let  $J$, $K$ be intervals for $\otimes_2$ such that $J \leq K$.
    If $f,g:X \to Y$ are one-step $(K,\otimes_2)$-homotopic then they are also one-step $(J,\otimes_1)$-homotopic.
(By \cref{lemma:interval_object_for_a_product}, $J,K$ are intervals for $\otimes_1$.)
\end{proposition}

\begin{proof}
  Let $h:J \to K$ be a morphism of intervals.
  Consider the following diagram.
  \begin{equation*}
    \begin{tikzcd}
      X \otimes_1 *   \ar[r,"\Id"',"\isom"]
      \ar[d,"1_{X} \otimes_1 0_J"'] & 
      X \otimes_2 * 
      \ar[d,"1_X \otimes_2 0_J"'] 
      \ar[dr,"1_X \otimes_2 0_K",pos=0.3]\ar[drr,"f",bend left=20] \\
      X \otimes_1 J \ar[r,"\Id"'] &
      X \otimes_2 J \ar[r,"1_X \otimes_2 h"'] &
      X \otimes_2 K \ar[r,dashed,"H"',pos=0.3] & Y\\
      X\otimes_1 *\ar[u,"1_X\otimes_1 1_J"]\ar[r,"\Id"',"\isom"]                        & X\otimes_2 * \ar[urr,"g"',bend right=20]\ar[u,"1_X\otimes_2 1_J"]\ar[ur,"1_X\otimes_2 1_K"',pos=0.3]
    \end{tikzcd}
  \end{equation*}
  The left squares commute by the natural transformation $\otimes_1 \To \otimes_2$.
  The middle triangles commute because $h$ is a morphism on intervals.
 If $f$ and $g$ are one-step $(K,\otimes_2)$-homotopic then there exists a map $H$ such that right triangles commute. It follows that $f$ and $g$ are one-step $(J,\otimes_1)$-homotopic.
\end{proof}

\begin{corollary} \label{cor:partial-order}
 Let $J,K$ be intervals for $\otimes_2$ with $J \leq K$ and let $\otimes_1,\otimes_2$ be product operations with $\otimes_1 \leq \otimes_2$.
  Let $f,g:X \to Y \in \cat{Cl}$.
  If $f \sim_{(K,\otimes_2)} g$ then $f \sim_{(J,\otimes_1)} g$.
\end{corollary}

\begin{corollary}
\label{corollary:product_of_intervals_for_arbitrary_product}
Let $J\in \{I_{\tau},J_{\top},J_+\}$.
  If $f \sim_{(J,\otimes)} g$ and $h \sim_{(J,\otimes)} k$ then $f \otimes h \sim_{( J,\otimes)} g \otimes k$.
\end{corollary}

\begin{proof}
  By \cref{corollary:product_of_homotopies}, $f\otimes h\sim_{(J\otimes J,\otimes)}g\otimes k$.
  By \cref{corollary:order_among_segment_constructions,prop:partial-order}, this implies that $f\otimes h\sim_{(J\vee J,\otimes)}g\otimes k$. 
Thus, there is a zigzag of elementary $(J \vee J,\otimes)$ homotopies from $f \otimes h$ to $g \otimes k$.
  Let $m$ denote the number of elementary homotopies in the zigzag.
  If $J = I_{\tau}$ then these $m$ elementary homotopies combine to give an elementary $([0,2m]_{c_{0^+}},\otimes)$ homotopy.
  Since $I_{\tau}$ is homeomorphic to $[0,2m]_{c_{0^+}}$, this is an elementary $(I_{\tau},\otimes)$ homotopy.
  If $J = J_{\top}$ then these $m$ elementary homotopies combine to give an elementary $(J_{2m},\otimes)$ homotopy.
  If $J = J_{+}$ then these $m$ elementary homotopies combine to give an elementary $(J_{2m,k},\otimes)$ homotopy for some $0 \leq k \leq 2^{2m}-1$.
  In each case, the result then follows by \cref{lemma:general_homotopic}.
\end{proof}

\begin{theorem} \label{thm:J1}
  Let $m \geq 1$, $0^+ < (\eps,a) \leq 1$ and let $\otimes$ be a product operation.
  Let $f,g:X \to Y$ be a continuous map of closure spaces.
  The following are equivalent.
  \begin{enumerate}
  \item \label{it:thm1_1} $f,g$ are $(J_{\top},\otimes)$-homotopic.
  \item \label{it:thm1_m} $f,g$ are $(J_m,\otimes)$-homotopic.
 \item \label{it:thm1_m_top} $f,g$ are $(J_{m,\top},\otimes)$-homotopic.
  \item \label{it:thm1_a} $f,g$ are $(I_{(\eps,a)},\otimes)$-homotopic.
  \end{enumerate}
\end{theorem} 

\begin{proof}
   By \cref{ex:partial-order}\eqref{it:Jn-to-Jm}, $J_m \leq J_{\top}$.
    By \cref{cor:partial-order}, \eqref{it:thm1_1} implies \eqref{it:thm1_m}.
  By \cref{lemma:general_homotopic}\eqref{it:J_1}, if $f,g$ are one-step $(J_m,\otimes)$-homotopic then $f,g$ are $(J_{\top},\otimes)$-homotopic.
  If $f,g$ are $(J_m,\otimes)$-homotopic then they are connected by a finite sequence of elementary $(J_m,\otimes)$ homotopies.
  So they are $(J_{\top},\otimes)$-homotopic.
  Thus \eqref{it:thm1_m} implies \eqref{it:thm1_1}.

 By \cref{ex:partial-order}\eqref{it:Jn-to-Jm}, $J_{m,\top} \leq J_{\top}$.
    By \cref{cor:partial-order}, \eqref{it:thm1_1} implies \eqref{it:thm1_m_top}.
    Similarly, by \cref{ex:partial-order}\eqref{it:J1-to-Jm}, $J_{\top} \leq J_{m,\top}$ and thus \eqref{it:thm1_m_top} implies \eqref{it:thm1_1}.
  
  Choose integers $n$, $N$, with $1 \leq n \leq N$ such that $\frac{1}{N} \leq (\eps,a) \leq \frac{1}{n}$.
  By \cref{ex:partial-order}\eqref{it:Ia} and  \cref{ex:partial-order}\eqref{it:IJ}, $J_N \leq I_{\frac{1}{N}} \leq I_{(\eps,a)} \leq I_{\frac{1}{n}} \leq J_n$. 
    Thus \eqref{it:thm1_a} implies that $f,g$ are $(J_N,\otimes)$-homotopic, which we have shown implies that they are $(J_{\top},\otimes)$-homotopic.
    In addition, we have shown that \eqref{it:thm1_1} implies that $f,g$ are $(J_n,\otimes)$-homotopic and since $I_{(\eps,a)} \leq J_n$ this implies \eqref{it:thm1_a}.
\end{proof}

\begin{theorem} \label{thm:Jplus}
 Let $m \geq 1$ and $0 \leq k \leq 2^m-1$ and let $\otimes$ be a product operation. Let $f,g: X \to Y$ be continuous maps of closure spaces. The following are equivalent.
  \begin{enumerate}
  \item \label{it:thm2_Jplus} $f,g$ are $(J_+,\otimes)$-homotopic.
  \item \label{it:thm2_Jmleq} $f,g$ are $(J_{m,\leq},\otimes)$-homotopic.
  \item \label{it:thm2_Jmk} $f,g$ are $(J_{m,k},\otimes)$-homotopic.
  \end{enumerate}
\end{theorem}

\begin{proof}
By \cref{ex:partial-order}\eqref{it:Jn-to-Jm}, $J_{m,\leq} \leq J_+$.
    By \cref{cor:partial-order}, \eqref{it:thm2_Jplus} implies \eqref{it:thm2_Jmleq}.
    Similarly, by \cref{ex:partial-order}\eqref{it:J1-to-Jm}, $J_+ \leq J_{m,\leq}$ and thus \eqref{it:thm2_Jmleq} implies \eqref{it:thm2_Jplus}.
    By \cref{lemma:general_homotopic}\eqref{it:Jplus_Jmk}, if $f,g$ are one-step $(J_{m,k},\otimes)$-homotopic then they are $(J_+,\otimes)$-homotopic.
    If $f,g$ are $(J_{m,k},\otimes)$-homotopic then they are connected by a sequence of elementary $(J_{m,k},\otimes)$ homotopies. So they are $(J_+,\otimes)$-homotopic. That is, \eqref{it:thm2_Jmk} implies \eqref{it:thm2_Jplus}.
    Finally, note that the equivalence relation generated by elementary $(J_+,\otimes)$ homotopy equals the equivalence relation generated by elementary $(J_-,\otimes)$ homotopy. Since by \cref{ex:partial-order}\eqref{it:JnlJmk} either $J_{m,k} \leq J_+$ or $J_{m,k} \leq J_{-}$, it follows that  \eqref{it:thm2_Jplus} implies \eqref{it:thm2_Jmk}.
\end{proof}

\begin{theorem} \label{cor:homotopy-poset}
Let $X$ and $Y$ be closure spaces. Let $f,g:X\to Y$ be continuous maps. Then we have the following implications 
\begin{equation*}
\begin{tikzcd}
    f \sim_{(J_{\top},\times)} g \ar[r,Rightarrow] \ar[d,Rightarrow] & f \sim_{(J_+,\times)} g \ar[r,Rightarrow] \ar[d,Rightarrow] & f \sim_{(I_{\tau},\times)} g \ar[d,Rightarrow]  \\
    f \sim_{(J_{\top},\boxplus)} g \ar[r,Rightarrow] & f \sim_{(J_+,\boxplus)} g \ar[r,Rightarrow] & f \sim_{(I_{\tau},\boxplus)} g
  \end{tikzcd}
\end{equation*}
Furthermore,
among
homotopy relations obtained from an interval with $0\neq 1$ and a product operation,
the relation $\sim_{(J_{\top},\times)}$
implies any other such homotopy relation.
\end{theorem}

\begin{proof}
Combining \cref{corollary:regular_product_is_coarsest,cor:partial-order} gives the vertical implications.
The horizontal implications follow from \cref{ex:partial-order}\eqref{it:IJmkJm} and \cref{cor:partial-order}.
The
second statement
follows from \cref{cor:partial-order,lemma:indiscrete_intervals_are_maximal,corollary:regular_product_is_coarsest}.
\end{proof}

For closure spaces $X$ and $Y$ define a partial order on pairs $(J,\otimes)$, where $J$ is an interval and $\otimes$ is a product operation, given by $(J,\otimes) \leq (J',\otimes')$ if for all $f,g:X \to Y$, $f \sim_{(J,\otimes)} g$
implies that
$f \sim_{(J',\otimes')} g$.
That is  $(J,\otimes) \leq (J',\otimes')$ if $(J',\otimes')$ gives a coarser partition of the set of continuous maps from $X$ to $Y$.
By \cref{cor:homotopy-poset,ex:J1discrete},
we have the following poset, which is independent of $X$ and $Y$.

\begin{theorem} \label{thm:homotopy-poset}
With the above partial order on
  intervals and product operations, we have
  the following Hasse diagram,
  
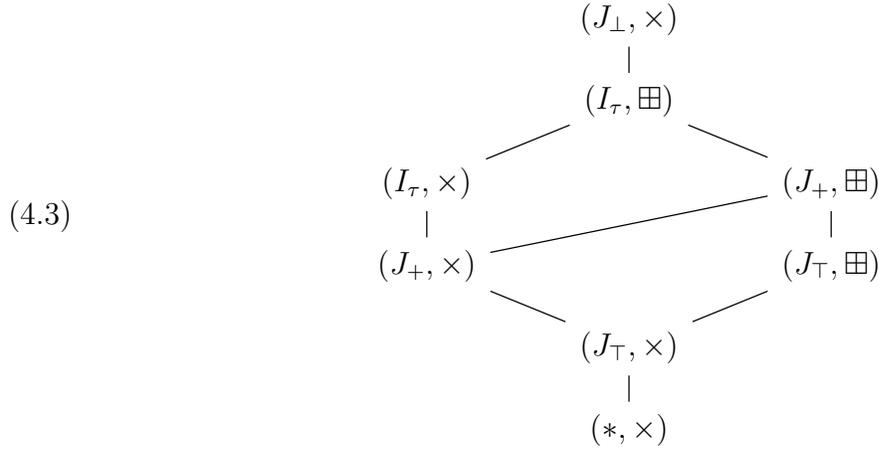
\begin{equation*} 
    \begin{tikzcd}[every arrow/.append style={dash},row sep=2ex]
      & (J_{\bot},\times) \ar[d] \\
      & (I_{\tau},\boxplus) \ar[dr] \ar[dl] \\
      (I_{\tau},\times) \ar[d] & & (J_+,\boxplus) \ar[dll] \ar[d]\\
      (J_+,\times) \ar[dr] & & (J_{\top},\boxplus) \ar[dl]\\
      & (J_\top,\times) \ar[d] \\
      & (*,\times)
    \end{tikzcd}
  \end{equation*}
  where $f \sim_{(*,\times)} g$ if and only if $f=g$ and thus $(*,\times)$ is the minimum
    and $f \sim_{(J_{\bot},\times)} g$ for all $f,g:X \to Y$ and thus $(J_{\bot},\times)$ is the maximum.
\end{theorem}

\begin{proposition} \label{ex:nontrivial}
For each of the equivalence relations in \cref{cor:homotopy-poset} there exist $f,g:X \to Y$ such that $f$ is not homotopy equivalent to $g$.
\end{proposition}

\begin{proof}
By \cref{cor:homotopy-poset}, it suffices to verify the statement for the case $(I_{\tau},\boxplus)$.
  Let $X$ be the two point discrete space. Then the identity map $\mathbf{1}_X$ is not $(I_{\tau},\boxplus)$ homotopic to a constant map. Indeed, if it was then we would have a homotopy $H:X\boxplus I_{\tau}\to X$ between the two maps. However, since $X$ is discrete, by \cref{lemma:ind_product_and_product_of_discrete_spaces}, this would be equivalent to asking for a homotopy $H:X\times I_{\tau}\to X$ between the two maps, which we know does not exist since the two point discrete space $X$ is not $(I_{\tau},\times)$-contractible.
\end{proof}

\begin{definition}
  \label{def:homotopy_equivalent_spaces}
  Let $J$ be an interval and let $\otimes$ be a product operation.
  We say that ${X}$ and ${Y}$ are \emph{$(J,\otimes)$-homotopy equivalent}
  if there exist morphisms $f:{X\to Y}$ and $g:{Y\to X}$ such that
  $gf \ \sim_{(J,\otimes)} \Id_X$ and
  $fg \ \sim_{(J,\otimes)} \ \Id_Y$.
  We say that ${X}$ is \emph{$(J,\otimes)$-contractible} 
  if it is $(J,\otimes)$-homotopy equivalent to the one-point space.
\end{definition}

\begin{example} \label{ex:contractible}
  Let $\otimes$ be a product operation and $J$ be an interval for $\otimes$.
  By Lemma~\ref{lem:interval-homotopy}, $J$ is $(J,\otimes)$-contractible.
Furthermore,
 for $n \geq 0$ and $J\in \{I_{\tau},J_{\top},J_+\}$, $J^{\otimes n}$ is $(J,\otimes)$-contractible, which follows by induction using \cref{corollary:product_of_intervals_for_arbitrary_product}.
\end{example}

\begin{example} \label{ex:Z} 
 Consider $(\Z,d)$ where $d(x,y) = |x-y|$. Then 
 $(\Z,c_1)$ is $(I_{\tau},\times)$-contractible~\cite[Lemma 4.49]{rieser2021vcech}.
 In contrast, we will show that the space $(\Z,c_1)$ is not $(J_+,\boxplus)$ contractible.
 Indeed, suppose that $f_0,\dots, f_m$ is a zigzag of one-step $(J_+,\boxplus)$-homotopic maps where $f_0=\mathbf{1}_{\Z}$ is the identity map and $f_m=0$ is the constant map to $0$.
A map $f:\Z \to \Z$ gives a continuous map $f:(\Z,c_1) \to (\Z,c_1)$ iff for all $n$, $\abs{f(n)-f(n-1)} \leq 1$.
  Let $f$ and $g$ be two such maps.
 By \cref{def:homotopy,def:inductive_product_closure}, the ordered pair $(f,g)$ is one-step $(J_+,\boxplus)$-homotopic iff for all $n$, $\abs{f(n)-g(n)} \leq 1$.
Since $f_0(m+1) = m+1$, $f_m(m+1) \geq 1$, we have a contradiction.
\end{example}

\subsection{Restrictions to full subcategories}

We end this section by remarking that a number of our homotopy theories have been previously studied in various full subcategories of closure spaces.

The interval $I$ lies in the full subcategory $\mathbf{Top}$. The intervals $J_{m,k}$, for $m\ge 1, 0\le k\le 2^{m-1}$ lie in the full subcategory $\qdcl \isom \cat{DiGph}$. The intervals $J_m$ for $m\ge 1$ and $I_{(\eps,a)}$, for $a=-1,0$ and $0\le \eps\le 1$ lie in the full subcategory $\suqdcl\isom \cat{Gph}$.

The product closure restricts to $\cat{Top}$, $\cat{Gph}$ and $\cat{DiGph}$ by \cref{prop:limits_and_colimits_of_closure_spaces,prop:regular_digraph_product_is_product}. The inductive product closure restricts to $\cat{Gph}$ and $\cat{DiGph}$ by \cref{prop:digraph_product_is_inductive_product,corollary:ind_prod_of_quasi_dis_is_quasi_dis}, where it is known as the cartesian product of graphs and digraphs. 

\begin{lemma}
\label{lemma:restriction_of_homotopies}
\begin{enumerate}
\item $(I_{\tau},\times)$-homotopy restricts to $\mathbf{Top}$ where it is called homotopy.
\item $(I_{\varepsilon^-},\times)$-homotopy restricts to $\mathbf{Gph}$ where it is called homotopy (\cite{sossinsky1986tolerance}). 
\item $(J_{\top},\boxplus)$-homotopy restricts to $\mathbf{Gph}$ where it is called $A$-homotopy or discrete homotopy (\cite{kramer1998combinatorial,barcelo2001foundations,barcelo2005perspectives,
,babson2006homotopy}).
\item $(J_{\top},\times)$-homotopy restricts to $\mathbf{Gph}$ where it is called $\times$-homotopy (\cite{dochtermann2009hom}).
\item $(J_+,\boxplus)$-homotopy restricts to $\mathbf{DiGph}$ where it is called homotopy (\cite{grigor2014homotopy}).
\item $(J_{\top},\times)$-homotopy restricts to $\mathbf{DiGph}$ where it is called bihomotopy (\cite{dochtermann2023homomorphism}).
\end{enumerate}
\end{lemma}

\section{Homology in closure spaces} \label{section:homology}

In this section we define several homology theories for closure spaces. We start by using some of our previously defined intervals and product operations to define various simplices and cubes. We then use the standard constructions to produce corresponding simplicial and cubical \new{singular} homology theories. \new{Let $\otimes$ denote either $\times$ or $\boxplus$.} For a closure space $X$,
\new{let $X^{\otimes n}$ denote the corresponding $n$-fold product of $X$ with itself.}

\subsection{Cubical homology} \label{sec:cubical}

We use intervals and either the product or the inductive product to define cubical singular homology theories. 
Let $J$ be one of the intervals $I_{\tau}, J_{\top}, J_+$.

\begin{definition}
\label{def:cube}
For $n\ge 1$, define the \emph{$(J,\otimes)$ $n$-cube} to be $\gcube{n}=J^{\otimes n}$. Define $\gcube{0}$ to be the one point space. Denote $\ttimesgcube{n}$ by $\timesgcube{n}$.
\end{definition}

\new{By \cref{def:product_operator}, if $J = (J,c)$ then $\gcube{n} = (J^n,c^{\otimes n})$. Furthermore, note that $\indcube{n}$ is the set $\{0,1\}^n$ with the indiscrete topology.}

\begin{definition}
  \label{def:singular_cubes}
Let $X$ be a closure space. Given a \emph{$(J,\otimes)$ singular $n$-cube}, $\sigma:\gcube{n}\to X$, for $1\le i\le n$ define
\begin{gather}
\new{d_{i,0}^n}(\sigma)(a_1,\dots, a_{n-1})=\sigma (a_1,\dots ,a_{i-1},0,a_i,\dots, a_{n-1})\\
\new{d_{i,1}^n}(\sigma)(a_1,\dots, a_{n-1})=\sigma (a_1,\dots ,a_{i-1},1,a_i,\dots, a_{n-1}).
\end{gather}
Say $\sigma$ is \emph{degenerate} if $\new{d_{i,0}^n}\sigma=\new{d_{i,1}^n}\sigma$ for some $i$. 
Let $\gcubechain{n}{X}$ be the quotient of the free abelian group on the $(J,\otimes)$ singular $n$-cubes in $X$, which we will denote by $\gcubechainwhole{n}{X}$, by the free abelian group on the degenerate singular $n$-cubes.
Elements of $\gcubechain{n}{X}$ are called \emph{$(J,\otimes)$ singular cubical $n$-chains} in $X$.
The \emph{boundary} map $\partial_n: \gcubechain{n}{X} \to \gcubechain{n-1}{X}$ is the linear map defined by 
\[\partial_n \sigma=\sum_{i=1}^n(-1)^i(\new{d_{i,0}^n}\sigma-\new{d_{i,1}^n}\sigma).
\]
One can check that $\partial_{n-1}\partial_n=0$ and thus $(\gcubechain{\bullet}{X},\partial_{\bullet})$ is a chain complex of abelian groups.
The \emph{cubical singular homology groups} are the homology groups of this chain complex, which we denote by $\gcubehom{\bullet}{X}$.
\end{definition}

\begin{definition} \label{def:reduced-homology}
  We may augment the singular chain complex with the \emph{augmentation map} $\eps: \gcubechain{0}{X} \to \Z$ given by $\sum_i n_i \sigma_i = \sum_i n_i$.
  The homology of the augmented singular chain complex is called \emph{reduced homology} and denoted $\rgcubehom{\bullet}{X}$.
\end{definition}

\begin{example}
\label{example:homology_of_one_point}
Let $*$ denote the one point space.
There is a single nondegenerate $(J,\otimes)$ $0$-cube given by the identity map and for $k \geq 1$ the $(J,\otimes)$ singular $k$-cubes are all degenerate.
Therefore $\rgcubehom{k}{*} = 0$ for all $k\ge 0$.
\end{example}

\new{Let $f:X \to Y$ be a continuous map of closure spaces.
Let $\sigma:\gcube{n}\to X$ be a $(J,\otimes)$ singular $n$-cube.
Then $f\circ \sigma:\gcube{n}\to Y$ is a $(J,\otimes)$ singular $n$-cube in $Y$.  Furthermore $f$ induces a group homomorphism $f_{\#}:\gcubechainwhole{n}{X}\to \gcubechainwhole{n}{Y}$, which sends degenerate cubes to degenerate cubes. Thus it also induces a group homomorphism $f_{\#}:\gcubechain{n}{X}\to \gcubechain{n}{Y}$. It can be checked that for all $n \geq 0$ these maps respect the boundary operators and thus they induce maps on homology $f_{*}:\gcubehom{n}{X}\to \gcubehom{n}{Y}$. In particular, for each $n \geq 0$, we have a functor $\gcubehom{n}{-}:\cat{Cl}\to \cat{Ab}$.}

\begin{theorem}
\label{theorem:homotopy_invariance_of_homology}
Let $f,g:(X,\new{c_X})\to (Y,\new{c_Y})$.
If $f\sim_{(J,\otimes)}g$, then $f_*=g_*:\gcubehom{n}{X,\new{c_X}}\to\gcubehom{n}{Y,\new{c_Y}}$.
\end{theorem}

\begin{proof}
  It is sufficient to assume that $f$ and $g$ are one-step $(J,\otimes)$-homotopic.
  That is, there exists
  $H:\new{J\otimes X}\to Y$ 
  such that \new{$H(0,-)=f(-)$} and \new{$H(1,-)=g(-)$}.
By definition,
  $\gcube{n}\otimes J$  is
  $\gcube{n+1}$.
  
  Let $\sigma:\gcube{n}\to (X,\new{c_X})$ be a singular $n$-cube. Define a map $\new{P_n}:\gcubechain{n}{X,\new{c_X}}\to \gcubechain{n+1}{Y,\new{c_Y}}$
  as follows.
  For a $(J,\otimes)$ singular $n$-cube $\sigma:\gcube{n}\to (X,\new{c_X})$ let
\new{$P_n(\sigma):\gcube{n+1}\to (Y,\new{c_Y})$ be the $(J,\otimes)$ singular $(n+1)$-cube defined by
\begin{equation*}
P_n(\sigma)(a_1,\dots,a_{n+1})=H(a_1,\sigma(a_2,\dots, a_{n+1})).
\end{equation*}
Note that immediately from the definition we have the following:
\begin{gather*}
\new{d_{1,0}^{n+1}}(P_n(\sigma))=H(0,\sigma)=f_{\#}(\sigma),\\
\new{d_{1,1}^{n+1}}(P_n(\sigma))=H(1,\sigma)=g_{\#}(\sigma),\\
\new{d_{i,0}^{n+1}}(P_n(\sigma))=P_{n-1}\new{d_{i-1,0}^n}(\sigma), 2\le i\le n+1,\\
\new{d_{i,1}^{n+1}}(P_n(\sigma))=P_{n-1}\new{d_{i-1,1}^n}(\sigma), 2\le i\le n+1.
\end{gather*}
}
\new{The last two equalities show that if $\sigma$ is degenerate, $P(\sigma)$ is also degenerate.}
We now show that $\new{\partial_{n+1} P_n=g_{\#}-f_{\#}-P_{n-1}\partial_n}$. We have \new{
\begin{align*}
\partial_{n+1} P_n(\sigma)=&\sum_{i=1}^{n+1}(-1)^i(\new{d_{i,0}^{n+1}}P_n(\sigma)-\new{d_{i,1}^{n+1}} P_n(\sigma))
=\\
&=-(\new{d_{1,0}^{n+1}}P_n(\sigma)-\new{d_{1,1}^{n+1}}P_n(\sigma))+\sum_{i=2}^{n+1}(-1)^iP_{n-1}(\new{d_{i-1,0}^n}(\sigma)-\new{d_{i-1,1}^n}(\sigma))=\\
&=-f_{\#}(\sigma)+g_{\#}(\sigma)+\sum_{i=1}^{n}(-1)^{i+1}P_{n-1}(\new{d_{i,0}^n}(\sigma)-\new{d_{i,1}^n}(\sigma))=\\
&=-f_{\#}(\sigma)+g_{\#}(\sigma)-P_{n-1}\partial_n(\sigma).
\end{align*}}
\new{Extending linearly we get that $\partial_{n+1} P_n=g_{\#}-f_{\#}-P_{n-1}\partial_n$.}  
Thus, $P$ is induces a homomorphism $P:\gcubechain{n}{X,\new{c_X}}\to \gcubechain{n+1}{Y,\new{c_Y}}$.
Therefore, $P$ is a chain homotopy between $f_{\#}$ and $g_{\#}$ and hence $f_*=g_*$.
\end{proof}

By \cref{theorem:homotopy_invariance_of_homology,example:homology_of_one_point}
we get the following corollary.

\begin{corollary}
\label{corollary:reduced_homology_of_contractible_spaces}
Let $X$ be a $(J,\otimes)$-contractible closure space. Then $\rgcubehom{n}{X}=0$ for all $n\ge 0$.
\end{corollary}

\new{The following examples will help demonstrate that four of our six cubical singular homology groups are pairwise distinct. Furthermore, these four homology groups are also distinct from the other two.}

\begin{example} \label{ex:J+}
  The only continuous maps from $J_{\top}$ to $J_+$ are the constant maps, which are degenerate $(J_{\top},\times)$ and $(J_{\top},\boxplus)$ singular $1$-cubes, and thus $\iindcubehom{0}{J_+}=\indcubehom{0}{J_+}=\Z^2$.
\end{example}

\begin{example} \label{ex:R}
  Consider the space $(\R,c_{0^+})$.
  The only continuous maps from 
$\indcube{1}$, $\iindcube{1}$, $\dicube{1}$ and $\idicube{1}$ into $(\R,c_{0^+})$ are the constant maps. Therefore
 \[\indcubehom{0}{\R,c_{0^+}}=\iindcubehom{0}{\R,c_{0^+}}=\dicubehom{0}{\R,c_{0^+}}=\idicubehom{0}{\R,c_{0^+}}=\bigoplus_{x\in \R}\Z.\] 
\end{example} 

\begin{example}
  \label{example:homology_of_directed_inductive_cubes}
  Consider $\idicube{2} = J_+^{\boxplus 2}$.
We will show that $\dicubehom{1}{\idicube{2}}\cong\Z$.
Let $a$,$b$,$c$,$d$, denote the vertices $(0,0)$, $(0,1)$, $(1,0)$ and $(1,1)$ of $\idicube{2}$.
Denote the 
singular $k$ cubes by the images of the vertices of $\dicube{k}$, when those vertices are listed in lexicographic order. 
Then, $\dicubechain{0}{\idicube{2}} = \Z\langle a,b,c,d \rangle$ and
$\dicubechain{1}{\idicube{2}} = \Z \langle ab, ac, bd, cd \rangle$. 
\new{Note that  for example, $abab$ is degenerate, and the map corresponding to $abcd$ is not continuous. Similar arguments exclude other potential elements 
resulting in}  $\dicubechain{2}{\idicube{2}} = \langle aaab, abbb, aaac, accc, bbbd, bddd, cccd, cddd \rangle$.  
One can check that $\partial_2 = 0$. 
Since, $\ker \partial_1 = \Z \langle \tau \rangle$, where $\tau = ab + bd - ac - cd$, the result follows.
\end{example}

\begin{example}
  \label{example:homology_of_undirected_inductive_cubes}
 Consider $\iindcube{2} = J_{\top}^{\boxplus 2}$.
We will show that $\indcubehom{1}{\iindcube{2}}\cong\Z$.
  Let $a$,$b$,$c$,$d$, denote the vertices $(0,0)$, $(0,1)$, $(1,0)$ and $(1,1)$ of $\iindcube{2}$.
  Denote the 
singular $k$ cubes by the images of the vertices of $\indcube{k}$, when those vertices are listed in lexicographic order. 
Then, $\indcubechain{0}{\iindcube{2}} = \Z\langle a,b,c,d \rangle$ and
$\indcubechain{1}{\iindcube{2}} = \Z \langle ab, ba, ac, ca, bd, db, cd, dc \rangle$.

Furthermore, $\indcubechain{2}{\iindcube{2}}$ is the free abelian group with generators of the form 
$uuuv$, $uuvu$, $uvuu$, $vuuu$, and $uvvu$ where
$(u,v) \in \{(a,b),(b,a),(a,c),(c,a),(b,d),(d,b),(c,d),(d,c)\}$.
Note that, for example, $\partial_2(abaa)=ab+ba-2aa=ab+ba$ since $aa$  is degenerate. No map whose image contains $3$ distinct vertices of $J_{\top}\boxplus J_{\top}$ is continuous, and the map corresponding to $abcd$ is not continuous.
The result follows from checking that $\ker \partial_1 = \Z \langle ab+ba, ac+ca, bd + db, cd + dc, \tau \rangle$, where $\tau = ab + bd - ac -cd$ and
$\im \partial_2 =  \Z\langle ab+ba, ac+ca, bd + db, cd + dc \rangle$.
\end{example}


\begin{proposition}
\label{prop:cub_theories_are_distinct}
With the possible exception of
$H_{\bullet}^{(I_{\tau},\times)}$ and
$H_{\bullet}^{(I_{\tau},\boxplus)}$ our six cubical singular homology theories are distinct. More specifically, we have the following table:
\end{proposition}

\new{
\begin{center}
\begin{tabular}{|*{7}{c|}}
                               \cline{1-1}
  $\iindcubehom{\bullet}{-}$                    \\ \cline{1-2}
  $\indcubehom{\bullet}{-}$ & \textnormal{Distinct}                 \\ \cline{1-3}
  $\idicubehom{\bullet}{-}$ & \textnormal{Distinct} &  \textnormal{Distinct}            \\ \cline{1-4}
  $\dicubehom{\bullet}{-}$ & \textnormal{Distinct} & \textnormal{Distinct} & \textnormal{Distinct}         \\ \cline{1-5}
  $\itopcubehom{\bullet}{-}$ & \textnormal{Distinct} & \textnormal{Distinct} & \textnormal{Distinct} & \textnormal{Distinct}      \\ \cline{1-6}
  $\topcubehom{\bullet}{-}$ & \textnormal{Distinct} & \textnormal{Distinct} & \textnormal{Distinct} & \textnormal{Distinct} & \textnormal{Open}  \\ \hline
    & $\iindcubehom{\bullet}{-}$  & $\indcubehom{\bullet}{-}$  & $\idicubehom{\bullet}{-}$ & $\dicubehom{\bullet}{-}$ & $\itopcubehom{\bullet}{-}$ & $\topcubehom{\bullet}{-}$ \\ \hline
\end{tabular}    
\end{center}
}

\begin{proof}
Since $(\R,c_{0^+})$ is $(I_{\tau},\times)$-contractible it is also $(I_{\tau},\boxplus)$-contractible by \cref{cor:homotopy-poset}. Thus $\topcubehom{0}{\R,c_{0^+}} \isom\itopcubehom{0}{\R,c_{0^+}} \isom \mathbb{Z}$ by \cref{theorem:homotopy_invariance_of_homology}.
Therefore by \cref{ex:R} each of $\indcubehom{\bullet}{-},\iindcubehom{\bullet}{-},\dicubehom{\bullet}{-},\idicubehom{\bullet}{-}$ is distinct from each of $\topcubehom{\bullet}{-},\itopcubehom{\bullet}{-}$.
By \cref{ex:contractible,theorem:homotopy_invariance_of_homology} we have that $\dicubehom{0}{J_+} \isom \idicubehom{0}{J_+} \isom \mathbb{Z}$ .  Combined with \cref{ex:J+} we have that  $\iindcubehom{\bullet}{-}$, $\indcubehom{\bullet}{-}$ are each distinct from each of $\dicubehom{\bullet} {-}, \idicubehom{\bullet}{-}$.

By \cref{ex:contractible,theorem:homotopy_invariance_of_homology} we have that $J_+^{\boxplus 2}$ is $(J_+,\boxplus)$-contractible and thus $\idicubehom{1}{J_+^{\boxplus 2}}=0$. By
\cref{example:homology_of_directed_inductive_cubes} we then have that $\idicubehom{\bullet}{-}$ is distinct from $\dicubehom{\bullet}{-}$.

Similarly,
 $\iindcubehom{1}{J_{\top}^{\boxplus 2}}=0$, 
which together with
\cref{example:homology_of_undirected_inductive_cubes} 
shows
that $\iindcubehom{\bullet}{-}$ is distinct from $\indcubehom{\bullet}{-}$.
\end{proof}

\subsection{Simplicial homology} \label{sec:simplicial}

In the case of the product, we define corresponding simplicial singular homology theories.
Let $J$ be one of $I_{\tau}$, $J_{\top}$, or $J_+$. Denote $(J,\times)$ simply by $J$.

\new{We give a uniform definition of simplices as subspaces of cubes.}

\begin{definition}
\label{def:simplices}
For $n\ge 0$ define $\iota:\{0,\dots ,n\}\to \{0,1\}^n$ by $\iota(k)=(\underbrace{1,\dots, 1}_{k},\underbrace{0,\dots, 0}_{n-k})$.
\new{Define the \emph{$J$ $n$-simplex}, denoted $\gsimp{n}$, to be the convex hull of $\im(\iota)$ in $\ttimesgcube{n}$ with the subspace closure.}
For $0\le i\le n$, 
the \emph{$i$-face} of $\gsimp{n}$ is the convex hull of the image of $\iota|_{\{0,\dots, \hat{i},\dots, n\}}$.
\end{definition}

Note that $\topsimp{n}$ is homeomorphic to the standard $n$-simplex,
$\indsimp{n}$ is homeomorphic to $J_{n,\top}$, the set $\{0,1,\ldots,n\}$ with the indiscrete topology, and
$\disimp{n}$ is homeomorphic to $J_{n,\leq}$, the set $\{0,1,\dots ,n\}$ with the closure operator $c(i)=\{j\,\mid\, i\le j\}$.

\begin{definition}
  \label{def:singular_simplices}
Let  $X$ be a closure space. Let $\gsimpchain{n}{X}$ be the free abelian group on the \emph{$J$ singular $n$-simplices}, $\sigma:\gsimp{n}\to X$. For $n\ge 1$, let $\partial_n:\gsimpchain{n}{X}\to \gsimpchain{n-1}{X}$ be the map defined by
\[\partial_n\sigma = \sum_{i=0}^n(-1)^id_i\sigma,
\]
where $d_i\sigma$ is the restriction of $\sigma$ to the $i$-th face of $\gsimp{n}$.
Since $\partial_{n-1}\partial_n=0$  we have a chain complex of free abelian groups, $(\gsimpchain{\bullet}{X},\partial_{\bullet})$ whose homology groups we denote by $\gsimphom{n}{X}$
and are called the simplicial singular homology groups.
\end{definition}

\new{
Let $f:X \to Y$ be a continuous map of closure spaces.
Let $\sigma:\gsimp{n}\to X$ be a $J$ singular $n$-simplex.
Then $f\circ \sigma:\gsimp{n}\to Y$ is a singular $J$ singular $n$-simplex in $Y$.  Furthermore $f$ induces a group homomorphism $f_{\#}:\gsimpchain{n}{X}\to \gsimpchain{n}{Y}$. It can be checked that for all $n \geq 0$ these maps respect the boundary operators and thus they induce maps on homology $f_{*}:\gsimphom{n}{X}\to \gsimphom{n}{Y}$. In particular,
for each $n \geq 0$ we have a functor
$\gsimphom{n}{-}:\cat{Cl}\to \cat{Ab}$.}

In a companion paper~\cite{bubenik2021eilenberg} we show that the corresponding simplicial singular homology groups and cubical singular homology groups agree.

\subsection{Restrictions to full subcategories}

We end this section by remarking that a number of our homology theories have been previously studied in various full subcategories of closure spaces.

Recall that $\suqdcl\cong \mathbf{Gph}$.
Under this isomorphism, $\indsimp{n}$ corresponds to $K_{n+1}$, the complete graph on $n+1$ vertices, and $\iindcube{n}$ corresponds to the hypercube graph $Q_n$. Also recall that $\qdcl\cong \mathbf{DiGph}$.
Under this isomorphism, $\disimp{n}$ corresponds to $K_{n+1}^{\nearrow}$, the digraph of the poset $(\{0,\dots, n\},\le)$ and  $\idicube{n+}$ corresponds to the hypercube digraph $Q_n^{\nearrow}$, where the vertices are the elements of $\{0,1\}^n$ and the directed edges given by $(a,a+e_i)$ where $e_i$ is a standard basis vector.

\begin{lemma}
\label{lemma:restricting_homology}
\begin{enumerate}
\item $\topcubehom{\bullet}{X}$ restricts to $\mathbf{Top}$ where it is called singular homology.
\item $\indcubehom{\bullet}{X}$ restricts to $\mathbf{Gph}$ where it is the homology of the clique complex of a graph. 
\item $\iindcubehom{\bullet}{X}$ restricts to $\mathbf{Gph}$ where it is called discrete (cubical) homology 
(\cite{barcelo2014discrete,Barcelo:2019,barcelo2018homology}).
\item $\dicubehom{\bullet}{X}$ restricts to $\mathbf{DiGph}$ where it is the homology of the directed clique complex.  
\end{enumerate}
\end{lemma}

\subsection{Homology with coefficients} \label{sec:coefficients}

   Let $C$ be one of the chain complexes of \cref{sec:cubical} or \cref{sec:simplicial}. 
Let $A$ be an abelian group, which we consider to be a chain complex concentrated in degree zero. 
Then the tensor product $C \otimes A$ is a chain complex whose homology $H_{\bullet}(C \otimes A)$ is called the homology of $C$ with coefficients in $A$.
As a special case, for a field $k$, the homology groups $H_j(C \otimes k)$ are $k$-vector spaces.

\section{Simplicial complexes from closure spaces} \label{sec:sc}

\new{In this section we give a sequence of adjunctions from closure spaces to graphs to simplicial complexes to hypergraphs.
Using these functors, we generalize the Vietoris-Rips complex of a metric space~\cite{Adamaszek:2017} and the (intrinsic) \v{C}ech complex of a metric space~\cite{cdso:geometric} to closure spaces.}

\subsection{Hypergraphs and simplicial complexes}
\label{sec:hypergraphs}

We define hypergraphs and simplicial complexes and related categories and functors. We obtain a sequence of adjunctions connecting closure spaces and hypergraphs via graphs and simplicial complexes.
\new{The definitions and adjunctions are straightforward and we encourage the enterprising reader to skip ahead to the statement of \cref{thm:adjunctions} and to work out the details for themselves.}

\begin{definition}
\label{def:hypergraph}
A \emph{simple hypergraph} $H$ is a pair $H=(X,E)$ where $X$ is a set and $E$ is a collection of non-empty subsets of $X$. We will call a simple hypergraph a \emph{hypergraph}. Elements of $X$ are called \emph{vertices} of the hypergraph $H$ and elements of $E$ are called \emph{hyperedges} of the hypergraph $H$. A \emph{hypergraph homomorphism} $f:(X,E)\to (Y,F)$ between two hypergraphs is a map $f:X\to Y$ such that for each $e\in E$, $f(e)\in F$.
Let $\cat{HypGph}$ denote the category of hypergraphs and hypergraph homomorphisms.
Say that a hypergraph has \emph{finite type} if its hyperedges are finite sets.
Say that a hypergraph is \emph{downward closed} if $\tau \in E$ and $\varnothing \neq \sigma \subset \tau$ implies that $\sigma \in E$ and $x \in X$ implies that $\{x\} \in E$.
A \emph{simplicial complex} is a downward-closed finite-type hypergraph.
Let $\cat{HypGph_{ft}}$, $\cat{HypGph_{dc}}$, and $\cat{Simp}$ denote the full subcategories of $\cat{HypGph}$ consisting of finite type hypergraphs, downward closed hypergraphs, and simplicial complexes.
In $\cat{Simp}$, hyperedges and hypergraph homomorphisms are called \emph{simplices} and \emph{simplicial maps}, respectively.
\end{definition}

Let $(X,E) \in \cat{HypGph}$.
Define the \emph{downward closure} of $E$, $\dc(E)$, to be the collection of nonempty subsets $\sigma$ of $X$ such there exists $\tau \in E$ with $\sigma \subset \tau$ or $\sigma = \{x\}$ for some $x \in X$.
Assume $f:(X,E) \to (Y,F) \in \cat{HypGph}$.
Given $\varnothing \neq \sigma \subset \tau \in E$,
$\varnothing \neq f(\sigma) \subset f(\tau) \in F$.
So $f(\sigma) \in \dc(F)$.
Also $f(\{x\}) = \{f(x)\} \in Y$.
So $f(\{x\}) \in \dc(F)$.
Therefore $f:(X,\dc(E)) \to (Y,\dc(F)) \in \cat{HypGph}$.
Thus the mappings $(X,E)$ to $(X,\dc(E)$ and $f:(X,E) \to (Y,E)$ to $f:(X,\dc(E)) \to (Y,\dc(F))$ define a functor $\dc:\cat{HypGph} \to \cat{HypGph_{dc}}$.

\begin{proposition} \label{prop:adjunction-hypgphdc-hypgph}
Let $(X,E) \in \cat{HypGph}$ and $(Y,\new{F}) \in \cat{HypGph_{dc}}$.
  Given a set map $f:X \to Y$, $f:(X,E) \to (Y,F)$ is a hypergraph homomorphism iff $f:(X,\dc(E)) \to (Y,F)$ is a hypergraph homomorphism.
  Thus, we have a natural isomorphism
  \[ \cat{HypGph_{dc}}((X,\dc(E)),(Y,F)) \isom \cat{HypGph}((X,E),(Y,F)).
  \]
That is, $\dc$ is left adjoint to the inclusion functor
$\cat{HypGph_{dc}} \incl \mathbf{HypGph}$.
\end{proposition}

\begin{proof}
  $(\Rightarrow)$ If $\tau \in E$ and $\varnothing \neq \sigma \subset \tau$ then $\varnothing \neq f(\sigma) \subset f(\tau) \in F$ and thus $f(\sigma) \in F$. If $x \in X$ then $f(\{x\}) = \{f(x)\} \in F$.
  $(\Leftarrow)$ $E \subset \dc(E)$.
\end{proof}

Let $(X,E) \in \cat{HypGph_{dc}}$.
Define $\tr_{\infty}(E) = \{\sigma \in E \ | \ \abs{\sigma} < \infty\}$.
Let $\tr_{\infty}: \cat{HypGph_{dc}} \to \cat{Simp}$ denote the functor defined by mapping $(X,E)$ to $(X,\tr_{\infty}(E))$ and mapping $f:(X,E) \to (Y,F)$ to $f:(X,\tr_{\infty}(E)) \to (Y,\tr_{\infty}(F))$.

Let $(X,E) \in \cat{Simp}$.
Define $\cosk_{\infty}(E)$ be the collection of nonempty subsets $\tau$ of $X$ such that for all finite nonempty subsets $\sigma \subset \tau$, $\sigma \in E$.
Note that $\sigma \in \cosk_{\infty}(E)$ and $\abs{\sigma} < \infty$ implies that $\sigma \in E$.
Let $\cosk_{\infty}:\cat{Simp} \to \cat{HypGph_{dc}}$ denote the functor defined by mapping $(X,E)$ to $(X,\cosk_{\infty}(E))$ and mapping $f:(X,E) \to (Y,F)$ to $f:(X,\cosk_{\infty}(E)) \to (Y,\cosk_{\infty}(F))$.

\begin{proposition} \label{prop:adjunction-simp-hypgphdc}
Let $(X,E) \in \cat{HypGph_{dc}}$ and $(Y,E) \in \cat{Simp}$.
  Given a set map $f:X \to Y$, $f:(X,\tr_{\infty}(E)) \to (Y,F)$ is a simplicial map iff $f:(X,E) \to (Y,\cosk_{\infty}(F))$ is a hypergraph homomorphism.
  Thus, we have a natural isomorphism
  \[ \cat{Simp}((X,\tr_{\infty}(E)),(Y,F)) \isom \cat{HypGph_{dc}}((X,E),(Y,\cosk_{\infty}(F))).
  \]
That is, $\tr_{\infty}$ is left adjoint to $\cosk_{\infty}$.
\end{proposition}

\begin{proof}
  $(\Rightarrow)$ Let $\tau \in E$. Note that for all nonempty finite subsets of $f(\tau)$ equal $f(\sigma)$ for some nonempty finite $\sigma\subset \tau$. Since $\tau\in E$, for all nonempty, finite subsets $\sigma \subset \tau$, $\sigma \in \tr_{\infty}(E)$ and hence $f(\sigma)\in F$. Thus for all nonempty finite subsets $\sigma' \subset f(\tau)$, $\sigma' \in F$. Therefore $f(\tau) \in \cosk_{\infty}(F)$.

  $(\Leftarrow)$ If $\sigma \in \tr_{\infty}(E)$ then $f(\sigma) \in \cosk_{\infty}(F)$ and $\abs{f(\sigma)} < \infty$. Therefore $f(\sigma) \in F$.
\end{proof}

Let $(X,E) \in \cat{Simp}$. Define $\tr_1(E) = \{\sigma \in E \ | \ \abs{\sigma} = 2\}$. 
Let $\tr_1: \cat{Simp} \to \cat{Gph}$ denote the functor defined by mapping $(X,E)$ to $(X,\tr_1(E))$ and mapping $f:(X,E) \to (Y,F)$ to $f:(X,\tr_1(E)) \to (Y,\tr_1(F))$.

Let $(X,E) \in \cat{Simp}$.
Define $\cosk_1(E)$ be the collection of nonempty finite subsets $\tau$ of $X$ such that for all distinct $x,y \in \tau$, $\{x,y\} \in E$.
Note that this includes all subsets of $X$ of cardinality one.
Let $\cosk_1:\cat{Gph} \to \cat{Simp}$ denote the functor defined by mapping $(X,E)$ to $(X,\cosk_1(E))$ and mapping $f:(X,E) \to (Y,F)$ to $f:(X,\cosk_1(E)) \to (Y,\cosk_1(F))$.
Given a graph $(X,E)$, the simplicial complex $(X,\cosk_1(E))$ is called the \emph{clique complex} of the graph. A simplicial complex in the image of $\cosk_1: \cat{Gph} \to \cat{Simp}$ is called a \emph{flag complex} or, equivalently, is said to satisfy Gromov's \emph{no-$\Delta$ condition}.
Observe that for a graph $(X,E)$, $(X,\tr_1(\cosk_1(E))) = (X,E)$.

\begin{proposition} \label{prop:adjunction-gph-simp}
Let $(X,E) \in \cat{Simp}$ and $(Y,E) \in \cat{Gph}$.
  Given a set map $f:X \to Y$, $f:(X,\tr_1(E)) \to (Y,F)$ is a graph homomorphism iff $f:(X,E) \to (Y,\cosk_1(F))$ is a simplicial map.
  Thus, we have a natural isomorphism
  \[ \cat{Gph}((X,\tr_1(E)),(Y,F)) \isom \cat{Simp}((X,E),(Y,\cosk_1(F))).
  \]
That is, $\tr_1$ is left adjoint to $\cosk_1$.
\end{proposition}

\begin{proof}
  $(\Rightarrow)$ Let $\sigma \in E$. For all $x \neq y \in \sigma$, $\{x,y\} \in \tr_1(E)$. Thus $f(x)=f(y)$ or $\{f(x),f(y)\} \in F$. Therefore $f(\sigma) \in \cosk_1(F)$. 

  $(\Leftarrow)$ Let $x,y \in X$ such that $\{x,y\} \in E$ and $x \neq y$. Then $f(\{x,y\}) \in \cosk_1(F)$, which implies that either $f(x) = f(y)$ or $\{f(x),f(y)\} \in F$.
Therefore $f$ is a graph homomorphism.
\end{proof}

Combining  \cref{prop:adjunction-cl-clqd,prop:adjunction-clqd-clsqd,lemma:digraphs_are_quasi_discrete_closure_spaces} with  \cref{prop:adjunction-gph-simp,prop:adjunction-simp-hypgphdc,prop:adjunction-hypgphdc-hypgph}, we have the following sequence of adjunctions. Recall that adjunctions compose to give adjoint functors.

\begin{theorem} \label{thm:adjunctions}
  We have the following composite adjunction between $\cat{Cl}$ and $\cat{HypGph}$.
  \begin{equation*}
    \begin{tikzcd}
      \cat{Cl} \ar[r,shift right=1ex,"\qd"',"\bot"] &
      \qdcl \ar[l,shift right=1ex,hook'] \ar[r,shift right=1ex,"s"',"\bot"] &
      \suqdcl \ar[l,shift right=1ex,hook'] \ar[r,shift right=1ex,"\Phi"',"\isom"] &
      \cat{Gph} \ar[l,shift right=1ex,"\Psi"'] \ar[r,shift right=1ex,"\cosk_1"',"\bot"] &
      \cat{Simp} \ar[l,shift right=1ex,"\tr_1"'] \ar[r,shift right=1ex,"\cosk_{\infty}"',"\bot"] &
      \cat{HypGph_{dc}} \ar[l,shift right=1ex,"\tr_{\infty}"'] \ar[r,shift right=1ex,hook,"\bot"] &
      \cat{HypGph} \ar[l,shift right=1ex,"\dc"'] 
    \end{tikzcd}
  \end{equation*}
\end{theorem}

We end this section by defining one more functor that will be used \new{in} \cref{sec:vr}.

\begin{definition} \label{def:Gamma}
  Let $(X,\new{c_X}) \in \qdcl$. Define $\Gamma(\new{c_X})$ to be the downward closure of the collection of subsets of $X$, $\{\new{c_X}(x)\}_{x \in X}$.
  Let $\Gamma:\qdcl \to \cat{HypGph_{dc}}$ be the functor defined by mapping $(X,\new{c_X}) \to (X,\Gamma(\new{c_X}))$ and mapping $f:(X,\new{c_X}) \to (Y,\new{c_Y})$ to $f:(X,\Gamma(\new{c_X})) \to (Y,\Gamma(\new{c_Y}))$.
\end{definition}

\subsection{Vietoris-Rips and \v Cech complexes}
\label{sec:vr}

We give functorial constructions of Vietoris-Rips complexes and \v Cech complexes for closure spaces which send one-step $(J_{\top},\times)$-homotopic maps to contiguous simplicial maps. We give an adjoint functor to the Vietoris-Rips construction which sends contiguous simplicial maps to one-step $(J_{\top},\times)$-homotopic maps.

Let $(X,\new{c_X})$ be a closure space.
Let $\VR(\new{c_X})$ be the collection of nonempty finite subsets $\sigma \subset X$, such that for all $x \in \sigma$, $\sigma \subset \new{c_X}(x)$.
Note that for all $x \in X$, $x \in \new{c_X}(x)$, so $\{x\} \in \VR(\new{c_X})$.
Also, if $\tau \in \VR(\new{c_X})$ and $\varnothing \neq \sigma \subset \tau$ then $\sigma \in \VR(\new{c_X})$. Thus, $(X,\VR(\new{c_X}))$ is a simplicial complex.
Assume $f:(X,\new{c_X}) \to (Y,\new{c_Y}) \in \cat{Cl}$.
Let $\sigma \in \VR(X)$.
Then $f(\sigma)$ is a finite nonempty subset of $Y$.
Furthermore, all elements of $f(\sigma)$ are of the form $f(x)$ for some $x \in \sigma$.
As $\sigma \subset \new{c_X}(x)$, it follows that $f(\sigma) \subset f(\new{c_X}(x)) \subset \new{c_Y}(f(x))$.
Therefore $f(\sigma) \in \VR(\new{c_Y})$, which implies that $f:(X,\VR(\new{c_X})) \to (Y,\VR(\new{c_Y}))$ is a simplicial map.

\begin{definition}
\label{def:vr}
Define the functor $\VR:\cat{Cl} \to \cat{Simp}$ by mapping $(X,\new{c_X})$ to $(X,\VR(\new{c_X}))$ and $f:(X,\new{c_X}) \to (Y,\new{c_Y})$ to $f:(X,\VR(\new{c_X})) \to (Y,\VR(\new{c_Y}))$.
\end{definition}

Let $(X,d)$ be a metric space and $\eps > 0$.
Then
$(X,\VR(c_{\eps,d}))$ consists of simplices $\{x_0,\ldots,x_n\}$ where $d(x_i,x_j) \leq \eps$ for all $i,j$.
That is, it is the usual Vietoris-Rips complex on $(X,d))$.
We also have the variant,
$(X,\VR(c_{\eps^-,d}))$, consisting of simplices $\{x_0,\ldots,x_n\}$ where $d(x_i,x_j) < \eps$ for all $i,j$~\new{\cite{Adamaszek:2017}}.

Let $(X,\new{c_X})$ be a closure space.
Define $\Cech(\new{c_X})$ to be the collection of nonempty finite subsets $\sigma \subset X$ such that there exists $x \in X$ with $\sigma \subset \new{c_X}(x)$.
Note that for all $x \in X$, $x \in \new{c_X}(x)$, so $\{x\} \in  \Cech(\new{c_X})$.
Furthermore, if $\tau \in \Cech(\new{c_X})$ and $\varnothing \neq \sigma \subset \tau$ then $\sigma \in  \Cech(\new{c_X})$.
That is, $(X, \Cech(\new{c_X}))$ is a simplicial set.
Assume $f:(X,\new{c_X}) \to (Y,\new{c_Y}) \in \cat{Cl}$.
Let $\sigma \in  \Cech(\new{c_X})$.
Then $f(\sigma)$ is a nonempty finite subset of $Y$.
There exists $x \in X$ such that $\sigma \subset \new{c_X}(x)$,
which implies that $f(\sigma) \subset f(\new{c_X}(x)) \subset \new{c_Y}(f(x))$.
Therefore $f(\sigma) \in  \Cech(\new{c_Y})$
and thus $f:(X, \Cech(\new{c_X})) \to (Y, \Cech(\new{c_Y}))$ is a simplicial map.

\begin{definition} \label{def:cech}
Define the functor $ \Cech:\cat{Cl} \to \cat{Simp}$ by mapping $(X,\new{c_X})$ to $(X, \Cech(\new{c_X}))$ and $f:(X,\new{c_X}) \to (Y,\new{c_Y})$ to $f:(X, \Cech(\new{c_X})) \to (Y, \Cech(\new{c_Y}))$.
\end{definition}

Let $(X,d)$ be a metric space and $\eps > 0$.
Then
$(X, \Cech(c_{\eps,d}))$ consists of simplices $\{x_0,\ldots,x_n\}$ such that there is a $x \in X$ with $d(x_i,x) \leq \eps$ for all $i$.
That is, it is the (intrinsic) \v{C}ech complex on $(X,d))$~\new{\cite{cdso:geometric}}.
We also have the variant,
$(X, \Cech(c_{\eps^-,d}))$ consisting of simplices $\{x_0,\ldots,x_n\}$ such that there is a $x \in X$ with $d(x_i,x) < \eps$ for all $i$.

It is straightforward from the definitions to check that $\VR=\VR \iota \qd$ and $\check{\text{C}}=\check{\text{C}} \iota \qd$, where $\iota:\qdcl\to \cat{Cl}$ is the inclusion functor. That is, for an closure space $(X,\new{c_X})$, we have $(X,\VR(\new{c_X}))=(X,\VR(\qd(\new{c_X}))), (X,\check{\text{C}}(\new{c_X}))=(X,\check{\text{C}}(\qd(\new{c_X})))$.

\begin{example}
\label{example:Vietoris_Rips_and_Cech_Generalization}
Let $X=\{x,y,z\}$ be a $3$-point set with the closure operator $\new{c_X}$, defined by 
\[\new{c_X}(x)=\{x,y\},\new{c_X}(y)=\{x,y\},\new{c_X}(z)=\{x,y,z\}\]
Note that $\new{c_X}$ does not arise from a metric since it is not symmetric.
From the definitions, 
$\VR(\new{c_X})=\{\{x\},\{y\},\{z\},\{x,y\}\}$ and $\text{\v{C}}(\new{c_X})=\{\{x\},\{y\},\{z\},\{x,y\},\{x,z\},\{y,z\},\{x,y,z\}\}$. See  \cref{fig:Vietoris_Rips_Generalization}.

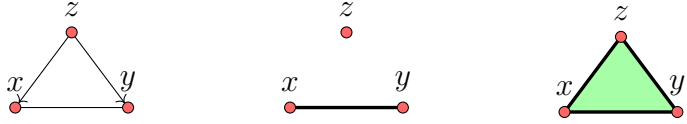
\begin{figure}[h]
\centering
\begin{tikzpicture}[scale=0.5]
   \node[label=$x$,shape=circle,fill=red!60,draw=black,scale=0.375] (A) at (-6,-4) {};
      \node[label=$y$,shape=circle,fill=red!60,draw=black,scale=0.375] (B) at (-3,-4) {};
      \node[label=$z$,shape=circle,fill=red!60,draw=black,scale=0.375] (C) at (-4.5,-2) {};
              \path [](B) edge node[left] {} (A);
              \path [->](C) edge node[left] {} (B);
              \path [->](C) edge node[left] {} (A);
 \end{tikzpicture}\quad \quad \quad \quad
 \begin{tikzpicture}[scale=0.5]  
    \node[label=$x$,shape=circle,fill=red!60,draw=black,scale=0.375] (A') at (-6,-4) {};
      \node[label=$y$,shape=circle,fill=red!60,draw=black,scale=0.375] (B') at (-3,-4) {};
      \node[label=$z$,shape=circle,fill=red!60,draw=black,scale=0.375] (C') at (-4.5,-2) {};
              \path [line width=0.45mm](B') edge node[left] {} (A');
    \end{tikzpicture}\quad \quad \quad \quad
\begin{tikzpicture}[scale=0.5]
  \draw[fill=green!35, line width=0.45mm] (-3,-4)--(-6,-4)--(-4.5,-2)--cycle;
      \node[label=$x$,shape=circle,fill=red!60,draw=black,scale=0.375] (A'') at (-6,-4) {};
      \node[label=$y$,shape=circle,fill=red!60,draw=black,scale=0.375] (B'') at (-3,-4) {};
      \node[label=$z$,shape=circle,fill=red!60,draw=black,scale=0.375] (C'') at (-4.5,-2) {};
\end{tikzpicture}
\caption{For the Alexandroff closure space in \cref{example:Vietoris_Rips_and_Cech_Generalization},
  we have its corresponding directed graph (left), its Vietoris-Rips complex (middle), and its \v Cech complex (right).}
\label{fig:Vietoris_Rips_Generalization}
\end{figure}
\end{example}

\begin{definition} \label{def:star}
Let $(X,E) \in \cat{Simp}$. For $x \in X$, let 
$\St(x)$ be 
\new{the set of vertices in the closed star of $x$, that is,}
the union of the simplices in $E$ containing $x$.
Since $\{x\} \in E$, $x \in \St(x)$.
By \cref{def:induced-qd}, let the \emph{star} closure, $\St(E)$, be the induced Alexandroff closure, given by $\St(E)(A) = \bigcup_{a \in A} \St(a)$,
which equals the union of the simplices in $E$ that intersect $A$.
\end{definition}

Assume $f:(X,E) \to (Y,F) \in \cat{Simp}$.
For $A \subset X$, $f(\St(E)(A))$ is the union of $f(\sigma)$, where $\sigma$ is a simplex in $E$ that intersects $A$, which implies that $f(\sigma)$ intersects $f(A)$.
Therefore $f(\St(E)(A)) \subset \St(F)(f(A))$.
Hence $f:(X,\St(E)) \to (Y,\St(F))$ is a continuous map.

\begin{definition} \label{def:G}
  Define $\St:\cat{Simp} \to \cat{Cl}$ to be the functor given by mapping $(X,E)$ to $(X,\St(E))$ and mapping $f:(X,E) \to (Y,F)$ to $f:(X,\St(E)) \to (Y,\St(F))$.
\end{definition}

\begin{theorem} \label{prop:adjunction-vr-g}
Let $(X,E) \in \cat{Simp}$ and $(Y,\new{c_Y}) \in \cat{Cl}$.
  Given a set map $f:X \to Y$,
  $f:(X,\St(E)) \to (Y,\new{c_Y})$ is a continuous map iff $f:(X,E) \to (Y,\VR(\new{c_Y}))$ is a simplicial map.
  Thus, we have a natural isomorphism
  \[ \cat{Cl}((X,\St(E)),(Y,\new{c_Y})) \isom \cat{Simp}((X,E),(Y,\VR(\new{c_Y}))).
  \]
That is, $\VR$ is right adjoint to $\St$.
\end{theorem}

\begin{proof}
  $(\Rightarrow)$  Let $\sigma \in E$.
  Let $x \in \sigma$. Then $\sigma \subset \St(E)(x)$, which implies that $f(\sigma) \subset f(\St(E)(x)) \subset \new{c_Y}(f(x))$.
  Therefore $f(\sigma) \in \VR(\new{c_Y})$.

  $(\Leftarrow)$ Let $A \subset X$.
  Then $f(\St(E)(A)) = f(\bigcup_{\sigma \in E, \sigma \cap A \neq \varnothing}\sigma) = \bigcup_{\sigma \in E, \sigma \cap A \neq \varnothing}f(\sigma)$.
  If $\sigma \in E$ then $f(\sigma) \in \VR(\new{c_Y})$, which implies that for all $x \in \sigma$, $f(\sigma) \subset \new{c_Y}(f(x))$. So, 
  $\bigcup_{\sigma \in E, \sigma \cap A \neq \varnothing}f(\sigma) \subset \new{c_Y}(f(A))$.
\end{proof}

The following two examples show that the functor $\Cech$ does not preserve limits or colimits.

\begin{example}
\label{example:cech_doesnt_preserve_colimits}
The functor $\Cech$ does not preserve pushouts.
Let $P$ be the closure space on the left of \cref{fig:Vietoris_Rips_Generalization}.
Let $A$ be the set $\{x,y\}$ with the discrete closure.
Let $X$ be the set $\{x,y\}$ with the indiscrete closure.
Let $Y$ be the closure space corresponding the directed graph on the left of
\cref{fig:Vietoris_Rips_Generalization} with the bottom (undirected) edge removed.
Then $P$ is the pushout of the continuous maps $A \to X$ and $A \to Y$
given by $x \mapsto x$ and $y \mapsto y$.
The \v{C}ech complex of $P$ is the simplicial complex on the right of
\cref{fig:Vietoris_Rips_Generalization}.
However, if we apply the functor $\Cech$ first and then take the pushout then we obtain the boundary of this simplicial complex.
\end{example}

\begin{example}
\label{example:cech_doesnt_preserve_limits}
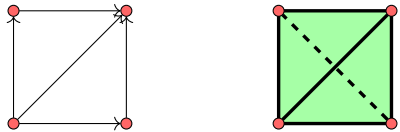
\begin{figure}[H]
\centering
\begin{tikzpicture}[scale=0.5]       
    \node[shape=circle,fill=red!60,draw=black,scale=0.375] (A) at (1,-4) {};          
    \node[shape=circle,fill=red!60,draw=black,scale=0.375] (B) at (4,-4) {};
     \node[shape=circle,fill=red!60,draw=black,scale=0.375] (C) at (1,-1) {};
      \node[shape=circle,fill=red!60,draw=black,scale=0.375] (D) at (4,-1) {};         
           \path [->](A) edge node[left] {} (B);
     \path [->](C) edge node[left] {} (D);
     \path [->](A) edge node[left] {} (C);
     \path [->](B) edge node[left] {} (D);
          \path [->](A) edge node[left] {} (D);   
    \end{tikzpicture}\quad\quad\quad\quad
    \begin{tikzpicture}[scale=0.5]                       
     \draw[fill=green!35, line width=0.45mm] (6,-4) rectangle (9,-1);
    \node[shape=circle,fill=red!60,draw=black,scale=0.375] (A') at (6,-4) {};          
    \node[shape=circle,fill=red!60,draw=black,scale=0.375] (B') at (9,-4) {};
     \node[shape=circle,fill=red!60,draw=black,scale=0.375] (C') at (6,-1) {};
      \node[shape=circle,fill=red!60,draw=black,scale=0.375] (D') at (9,-1) {};         
           \path [](A') edge node[left] {} (B');
     \path [](C') edge node[left] {} (D');
     \path [](A') edge node[left] {} (C');
     \path [](B') edge node[left] {} (D');
          \path [line width=0.45mm](A') edge node[left] {} (D'); 
     \path [dashed,line width=0.45mm](B') edge node[left] {} (C'); 
\end{tikzpicture}
\caption{The closure space $J_+\times J_+$ (left) and its \v{C}ech complex (right).}
\label{fig:cech_limits}
\end{figure}

The functor $\Cech$ does not preserve equalizers.
  Consider the closure space $J_+ \times J_+$
  on the left of \cref{fig:cech_limits},
  its subset $A = \{(1,0),(0,1)\}$, and the 
  the maps from $J_+ \times J_+$ to $J_{\top}$ given by the constant function $1$ and the indicator function on $A$.
  The equalizer of these maps is the subspace $A$ and its \v{C}ech complex is the simplicial complex with two vertices and no edges.
  In contrast, if we apply $\Cech$ to the equalizer diagram we obtain two maps from the $3$-simplex to the $1$-simplex whose equalizer is the $1$-simplex.

\end{example}

Since right adjoints preserve limits and left adjoints preserve colimits, from \cref{example:cech_doesnt_preserve_colimits,example:cech_doesnt_preserve_limits}
we have the following.

\begin{theorem}
\label{theorem:cech_doesnt_have_adjoint}
The functor $\Cech: \cat{Cl} \to \cat{Simp}$ does not have a have a right or a left adjoint.
\end{theorem}

We observe that the functors $\VR$, $\Cech$, and $\St$ may be written as compositions of the more elementary functors in \cref{thm:adjunctions,def:Gamma}.

\begin{proposition} \label{prop:functor-expansion}
  $\VR = \cosk_1 \circ \Phi \circ s \circ \qd$, $\St = \Psi \circ \tr_1$, and $ \Cech = \tr_{\infty} \circ \Gamma \circ \qd$.
It follows that restricted to $\suqdcl$, $\St \circ \VR = \Id_{\suqdcl}$.
\end{proposition}


It follows from \cref{prop:reverse} that the Vietoris-Rips functor is symmetric in the sense that $\VR = \cosk_1 \circ \Phi \circ s \circ \qd = \cosk_1 \circ \Phi \circ s \circ (-)^{\transpose} \circ \qd$.
On the other hand, there is a \emph{reverse \v{C}ech functor} given by $\Cech^{\transpose} = \tr_{\infty} \circ \Gamma \circ (-)^{\transpose} \circ A$.
Following~\cite{chowdhury2018functorial}, we may also call $\Cech$ the \emph{source \v{C}ech functor} and $\Cech^{\transpose}$ the \emph{sink \v{C}ech functor}.
Note that $\Cech^{\transpose}
= \tr_{\infty} \circ \Gamma \circ (-)^{\transpose} \circ A 
= \tr_{\infty} \circ \Gamma \circ A \circ (-)^{\transpose} \circ A
= \Cech \circ (-)^{\transpose} \circ A$.
\cref{fig:cech_and_cech_transpose} shows that $\Cech\neq \Cech^{\transpose}$.
As was the case for $\Cech$, $\Cech^{\transpose}$ does not preserve pushouts or equalizers and hence does not have a left or right adjoint.

\begin{figure}[H]
\centering
\begin{tikzpicture}[scale=0.5]       
    \node[shape=circle,fill=red!60,draw=black,scale=0.375] (A) at (1,-4) {};          
    \node[shape=circle,fill=red!60,draw=black,scale=0.375] (B) at (4,-4) {};
     \node[shape=circle,fill=red!60,draw=black,scale=0.375] (C) at (1,-1) {};
      \node[shape=circle,fill=red!60,draw=black,scale=0.375] (D) at (4,-1) {};         
           \path [->](A) edge node[left] {} (B);
     \path [->](A) edge node[left] {} (C);
          \path [->](A) edge node[left] {} (D);   
    \end{tikzpicture}\quad\quad\quad\quad
    \begin{tikzpicture}[scale=0.5]                       
     \draw[fill=green!35, line width=0.45mm] (6,-4) rectangle (9,-1);
    \node[shape=circle,fill=red!60,draw=black,scale=0.375] (A') at (6,-4) {};          
    \node[shape=circle,fill=red!60,draw=black,scale=0.375] (B') at (9,-4) {};
     \node[shape=circle,fill=red!60,draw=black,scale=0.375] (C') at (6,-1) {};
      \node[shape=circle,fill=red!60,draw=black,scale=0.375] (D') at (9,-1) {};         
           \path [](A') edge node[left] {} (B');
     \path [](C') edge node[left] {} (D');
     \path [](A') edge node[left] {} (C');
     \path [](B') edge node[left] {} (D');
          \path [line width=0.45mm](A') edge node[left] {} (D'); 
     \path [dashed,line width=0.45mm](B') edge node[left] {} (C'); 
\end{tikzpicture}\quad\quad\quad\quad
      \begin{tikzpicture}[scale=0.5]                       
    \node[shape=circle,fill=red!60,draw=black,scale=0.375] (A') at (6,-4) {};          
    \node[shape=circle,fill=red!60,draw=black,scale=0.375] (B') at (9,-4) {};
     \node[shape=circle,fill=red!60,draw=black,scale=0.375] (C') at (6,-1) {};
      \node[shape=circle,fill=red!60,draw=black,scale=0.375] (D') at (9,-1) {};         
     \path [line width=0.45mm](A') edge node[left] {} (B');
     \path [line width=0.45mm](A') edge node[left] {} (C');
     \path [line width=0.45mm](A') edge node[left] {} (D'); 
\end{tikzpicture}
\caption{A closure space given by a digraph (left), its source \v{C}ech complex (middle), and its sink \v{C}ech complex (right).}
\label{fig:cech_and_cech_transpose}
\end{figure}
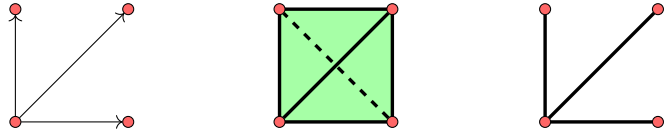

\subsection{Contiguous maps} \label{sec:contiguous}

We relate elementary $(J_{\top},\times)$ homotopies and contiguous simplicial maps via the functors $\VR$, $\Cech$, and $\St$.
Simplicial maps $f,g:(X,E)\to (Y,F)$ are said to be \emph{contiguous} if for all $\sigma\in E$, $f(\sigma)\cup g(\sigma)\in F$.
The contiguous relation is a reflexive and symmetric relation on the set of simplicial maps $\cat{Simp}((X,E),(Y,F))$. Taking the transitive closure of this relation yields an equivalence relation whose equivalence classes are known as \emph{contiguity classes}. Two simplicial complexes $(X,E)$ and $(Y,F)$ are said to be \emph{strongly equivalent} if there are simplicial maps $f:(X,E)\to (Y,F)$, $g:(Y,F)\to (X,E)$ such that $fg$ is in the same contiguity class as $1_Y$ and $gf$ is in the same contiguity class as $1_X$~\cite{barmak2012strong}.
Two simplicial complexes are strongly equivalent if and only if they have the same \emph{strong homotopy type}~\cite{barmak2012strong}, a notion that has been used to study discrete Morse functions~\cite{donovan2023homotopy}.

\begin{theorem}
\label{proposition:homotopy_and_contiguity}
Let $f,g:(X,\new{c_X})\to (Y,\new{c_Y}) \in \cat{Cl}$ be one-step $(J_{\top},\times)$-homotopic maps.
Then $f,g:(X,\VR(\new{c_X})) \to (Y,\VR(\new{c_Y}))$ are contiguous simplicial maps and so are
$f,g:(X, \Cech(\new{c_X})) \to (Y, \Cech(\new{c_Y}))$
and
$f,g:(X, \Cech^{\transpose}(\new{c_X})) \to (Y, \Cech^{\transpose}(\new{c_Y}))$.
Conversely, let $f,g:(X,E)\to (Y,F)$ be contiguous simplicial maps.
Then $f,g: (X,\St(E)) \to (Y,\St(F))$ are one-step $(J_{\top},\times)$-homotopic.
\end{theorem}

\begin{proof}
  Let $\sigma \in \VR(\new{c_X})$. For $x \in \sigma$, $\sigma \subset \new{c_X}(x)$.
  Then $f(\sigma) \subset f(\new{c_X}(x))$ and $g(\sigma) \subset g(\new{c_X}(x))$, and by \cref{lem:J1times-homotopy}, $f(\new{c_X}(x)) \cup g(\new{c_X}(x)) \subset \new{c_Y}(f(x)) \cap \new{c_Y}(g(x))$.
  Therefore $f(\sigma) \cup g(\sigma) \subset \new{c_Y}(f(x)) \cap \new{c_Y}(g(x))$.
  Thus, for all $y \in f(\sigma) \cup g(\sigma)$, $f(\sigma) \cup g(\sigma) \subset \new{c_Y}(y)$.
  Hence $f(\sigma) \cup g(\sigma) \in \VR(\new{c_Y})$.
  Therefore $f$ and $g$ are contiguous.

  Let $\sigma \in  \Cech(\new{c_X})$. Then there exists $x \in X$ such that $\sigma \subset \new{c_X}(x)$.
  Then by \cref{lem:J1times-homotopy}, $f(\sigma) \subset f(\new{c_X}(x)) \subset \new{c_Y}(f(x)) \cap \new{c_Y}(g(x))$.
  Similarly, $g(\sigma) \subset g(c(x)) \subset \new{c_Y}(f(x)) \cap \new{c_Y}(g(x))$.
  Therefore $f(\sigma) \cup g(\sigma) \subset \new{c_Y}(f(x)) \cap \new{c_Y}(g(x)) \subset \new{c_Y}(f(x))$.
  Hence $f(\sigma) \cup g(\sigma) \in  \Cech(\new{c_Y})$,
  and thus $f$ and $g$ are contiguous.

    Recall that $\Cech^{\transpose} = \Cech \circ (-)^{\transpose} \circ A$.
    We will show if
    $f,g:(X,\new{c_X}) \to (Y,\new{c_Y})$ are one-step $(J_{\top},\times)$-homotopic
    then so are
    $f,g:(X,A(\new{c_X})) \to (Y,A(\new{c_Y}))$ and
    $f,g:(X,(A(\new{c_X}))^{\transpose}) \to (Y,(A(\new{c_Y}))^{\transpose})$,
    from which it follows that
    $f,g:(X,\Cech^{\transpose}(\new{c_Y})) \to (Y,\Cech^{\transpose}(\new{c_Y}))$ are contiguous.
    First, assume that $H$ is a one-step $(J_{\top},\times)$-homotopy from $f:(X,\new{c_X}) \to (Y,\new{c_Y})$ to $g:(X,\new{c_X}) \to (Y,\new{c_Y})$.
  Apply the functor $A$ to the commutative diagram \eqref{eq:homotopy}.
  One may check that $(X \times \{0,1\},A(\new{c_X} \times c_{\top})) = (X \times \{0,1\},A(\new{c_X}) \times c_{\top})$.
  Indeed, the latter closure is Alexandroff and both closures agree for points in $X \times \{0,1\}$.
  Thus, we have the desired one-step $(J_{\top},\times)$-homotopy.
  Second, apply \cref{lem:digraph-one-step-j1-homotopy}.

  Assume that $f,g:(X,E) \to (Y,F)$ are contiguous simplicial maps.
  Let $A \subset X$. Then $\St(E)(A)$ is the union of all simplices in $E$ that intersect $A$.
  Hence 
  \[f(\St(E)(A)) \cup g(\St(E)(A)) = \bigcup_{\sigma \in E, \sigma \cap A \neq \varnothing} f(\sigma) \cup g(\sigma) \subset \bigcup_{\tau \in F, \tau \cap f(A) \neq \varnothing} \tau = \St(F)(f(A)).\] 
  Similarly $f(\St(E)(A)) \cup g(\St(E)(A)) \subset \St(F)(g(A))$.
  Therefore 
  \[f(\St(E)(A)) \cup g(\St(E)(A)) \subset \St(F)(f(A)) \cap \St(F)(g(A)).\]   So by \cref{lem:J1times-homotopy}, $f,g: (X,\St(E)) \to (Y,\St(F))$ are one-step $(J_{\top},\times)$-homotopic.
\end{proof}

Applying  \cref{proposition:homotopy_and_contiguity} inductively to a sequence of one-step $(J_{\top},\times)$-homotopy maps or a sequence of contiguous maps yields the following corollary.

\begin{corollary}
\label{corollary:contiguous}
Let $f,g:(X,\new{c_X})\to (Y,\new{c_Y}) \in \cat{Cl}$ and suppose that $f\sim_{(J_{\top},\times)}g$ .
Then $f,g:(X,\VR(\new{c_X})) \to (Y,\VR(\new{c_Y}))$ are in the same contiguity class and so are
$f,g:(X, \Cech(\new{c_X})) \to (Y, \Cech(\new{c_Y}))$
and $f,g:(X, \Cech^{\transpose}(\new{c_X})) \to (Y, \Cech^{\transpose}(\new{c_Y}))$.
Conversely, let $f,g:(X,E)\to (Y,F)$ be simplicial maps in the same contiguity class.
Then $f,g: (X,\St(E)) \to (Y,\St(F))$ are $(J_{\top},\times)$-homotopic.
\end{corollary}

We also have the following.

\begin{theorem}
Let $f,g: (X,\new{c_X}) \to (Y,\new{c_Y}) \in \suqdcl$.
  Then $f,g: (X,\new{c_X}) \to (Y,\new{c_Y}) \in \suqdcl$ are one-step $(J_{\top},\times)$-homotopic iff
  $f,g: (X,\VR(\new{c_X})) \to (Y,\VR(\new{c_Y})) \in \cat{Simp}$ are contiguous.
  Furthermore, $f,g: (X,\new{c_X}) \to (Y,\new{c_Y}) \in \suqdcl$ are $(J_{\top},\times)$-homotopic iff $f,g: (X,\VR(\new{c_X})) \to (Y,\VR(\new{c_Y})) \in \cat{Simp}$ are in the same contiguity class.
\end{theorem}

\begin{proof}
  The forward direction of the two statements is contained in
  \cref{proposition:homotopy_and_contiguity} and \cref{corollary:contiguous}, respectively.
  For the reverse direction, we also use 
  \cref{proposition:homotopy_and_contiguity} and \cref{corollary:contiguous}, respectively, together with
  \cref{prop:functor-expansion}, which gives us that
  $\St(\VR(\new{c_X})) = \new{c_X}$ and $\St(\VR(\new{c_Y})) = \new{c_Y}$.
\end{proof}

Finally, as a consequence of these results we have the following.

\begin{theorem}
  Suppose that $(X,\new{c_X})$ and $(Y,\new{c_Y})$ are closure spaces that are $(J_{\top},\times)$ homotopy equivalent. Then $(X,\VR(\new{c_X}))$ and $(Y,\VR(\new{c_Y}))$ are strongly equivalent,
and so are
$(X,\Cech(\new{c_X}))$ and $(Y,\Cech(\new{c_Y}))$,
as well as
$(X,\Cech^{\transpose}(\new{c_X}))$ and $(Y,\Cech^{\transpose}(\new{c_Y}))$.
Conversely, if $(X,E)$ and $(Y,F)$ are simplicial complexes that
are strongly equivalent, then $(X,\St(E))$ and $(Y,\St(F))$ are $(J_{\top},\times)$ homotopy equivalent.
Furthermore, if $(X,\new{c_X}),(Y,\new{c_Y}) \in \cat{Cl_{sA}}$ then $(X,\new{c_X})$ and $(Y,\new{c_Y})$ are $(J_{\top},\times)$-homotopy equivalent iff $(X,\VR(\new{c_X}))$ and $(Y,\VR(\new{c_Y}))$ are strongly equivalent.
\end{theorem}

\section{\new{Filtrations of} closure spaces and persistent homology} \label{sec:fcl}



\new{In this section, we define filtrations of closure spaces and give a sequence of generalizations from metric spaces to filtrations of closure spaces. 
If we apply our cubical and simplicial singular homology functors to any of these, then we obtain persistence modules.}


\subsection{\new{Filtrations}} \label{sec:filtrations}

\new{Let $P$ be a partially ordered set and let $\cat{C}$ be a category.}

\begin{definition} \label{def:filtered-closure-space}
\new{A \emph{$P$-filtration in $\cat{C}$ is} a functor $X:P\to \cat{C}$ such that each inequality $p\le q$ in $P$ is mapped to a monomorphism. Morphisms of $P$-filtrations in $\cat{C}$ are natural transformations of such functors. Let $\cat{F
_P C}$ denote the category of $P$-filtrations in $\cat{C}$ and their morphisms.}
\end{definition}

\new{A $P$-filtration of closure spaces,} $(X_{\bullet},c_{\bullet})$, is $P$-indexed set of closure spaces $\{(X_p,c_p)\}_{p\in P}$ such that for all $p\le q$ there is an injective continuous map $X_{p\le q}:(X_p,c_p)\to (X_q,c_q)$. A morphism of $P$-\new{filtrations of} closure spaces $f:(X_{\bullet},c_{X_{\bullet}})\to (Y_{\bullet},c_{Y_{\bullet}})$ consists of continuous maps $f_p:(X_p,c_{X_p})\to (Y_p,c_{Y_p})$ for each $p\in P$ such that for all $p\le q$, we have $f_q \circ X_{p\le q}=Y_{p\le q} \circ f_p$.

\begin{example} \label{ex:fcl}
Consider the following \new{two} examples.
First, let $X$ be a set together with closure operators $\{c_p\}_{p \in P}$ such that for all $p \leq q$, $c_p \leq c_q$.
Then for $p \leq q$, the identity gives a continuous map $(X,c_p) \to (X,c_q)$ and 
thus $(X,c_{\bullet}) \in \cat{F_P Cl}$. 
\new{An important class of such filtrations of closure spaces is given by metric closures for scales $r\ge 0$ as discussed in \cref{sec:clos-from-thick}. 
}
Second, consider a closure space $(X,c)$ together with a set map $f:X \to P$.
For $p \in P$, let $X_p = f^{-1}(D_p)$, where $D_p=\{q\in P\,|\, q\le p\}$, and let $c_p$ be the subspace closure. 
Then $(X_{\bullet},c_{\bullet}) \in \cat{F_P Cl}$.
\new{An important case is given by $P=\R$, called the \emph{sublevel-set filtration}.}
\end{example}



\new{
\begin{definition}
    Let $X \in \cat{F_P Cl}$ be such that for all $p \leq q$, $X_{p \leq q}: (X_p,c_p) \to (X_q,c_q)$ is given by $X_p \subset X_q$.
Abusing notation, let $X = \bigcup_{p \in P} X_p$ and for $A \subset X$, let $c(A) = \bigcup_{p \in P} c_p(A \cap X_p)$.
Then $(X,c) = \colim X$.
Furthermore, for all $p \in P$, the canonical continuous map $(X_p,c_p) \to (X,c)$ is given by $X_p \subset X$.
In this case, call $X$ a \emph{$P$-filtered closure space}.
We denote the full subcategory of $\cat{F_P Cl}$ consisting of $P$-filtered closure spaces by $\cat{F_P^{\subset} Cl}$.
Similarly, we define \emph{$P$-filtered (di)graphs}, \emph{$P$-filtered simplicial complexes}, and \emph{$P$-filtered topological spaces}.
\end{definition}
}


\new{Let $X,Y \in \cat{F_P^{\subset} Cl}$.
Observe that a morphism $f:X \to Y$ induces a canonical set function $f:X \to Y$ such that $f|_{X_p} = f_p$.
Conversely, if we are given a set function $f:X \to Y$ such that for all $p \in P$, $f|_{X_p}: X_p \to Y_p$ is continuous, then we have an induced morphism of $P$-filtered closure spaces, $f:X \to Y$, such that $f_p = f|_{X_p}$.
}

\new{
Let $X \in \cat{F_P Cl}$. 
Then for each $p \in P$, the canonical map $X_p \to \colim X$ is a monomorphism. 
Choose a closure space $X'$ in the isomorphism class  $\colim X$.
Then we have a $P$-filtered closure space $X'$, where $X'_p$ is given by the image of the canonical map from $X_p$ to $X'$.
Furthermore, $X' \isom X$.}

\new{
\begin{definition}
    Say that a $P$-filtration in $\cat{C}$, $X$, is \emph{right continuous} if for all $p \in P$, $X_p = \lim_{p \leq q} X_q$.
    Let $\cat{F_P^{rc} C}$ denote the full subcategory of right continuous $P$-filtrations in $\cat{C}$.
\end{definition}
}

\subsection{Persistence modules and interleaving} \label{sec:pm}
\new{Let $P=(P,\le)$ be a partial order.}
We recall the definitions of persistence modules, persistence diagrams, interleavings and matchings, and their corresponding distances. A functor from $\new{P}$ to a category $\cat{C}$ is called a \emph{persistence module} indexed by $P$ and with values in $\cat{C}$~\cite{bubenik2015metrics}.

\begin{definition}
\label{def:interleaving_distance}
Let $\cat{C}$ be a category, let $P$ be either of the totally ordered sets $[0,\infty)$ or $\new{\R} = (\R,\leq)$ and let $\eps \geq 0$.
Let $M,N: \new{P} \to \cat{C}$.
We say $M$ and $N$ are $\eps$-\emph{interleaved} if there are collections of maps $\{\varphi_p:M_p\to N_{p+\eps}\}_{p\in P}$ and $\{\psi_p:N_p\to M_{p+\eps}\}_{p\in P}$ such that:
\begin{enumerate}
\item \label{it:interleaving1} For all $p,q\in P$ with $p\le q$, $\varphi_{q}\circ M_{p\le q}=N_{p+\eps\le q+\eps}\circ\varphi_{p}$.
\item \label{it:interleaving2} For all $p,q\in P$ with $p\le q$,  $\psi_{q}\circ N_{p\le q}=M_{p+\eps\le q+\eps}\circ\psi_{p}$.
\item \label{it:interleaving3} For all $p\in P$, $\psi_{p+\eps}\circ\varphi_p=M_{p+2\eps}$.
\item \label{it:interleaving4} For all $p\in P$, $\varphi_{p+\eps}\circ \psi_p=N_{p+2\eps}$.
\end{enumerate}
The \emph{interleaving distance} between $M$ and $N$ is then defined to be 
\[d_I(M,N)=\inf\{\eps\,|\,\text{M and N are }\eps\text{-interleaved}\}.\]
\end{definition}

Assume that we have persistence modules
  indexed by 
  $R = (\R,\leq)$ and with values in $\cat{Vect}$, where $\cat{Vect}$ is the category of $k$-vector spaces and $k$-linear maps for some field $k$.
  Such a persistence module $M$ is called \emph{q-tame} if for each $s < t$, the linear map $M_{s < t}:M_s\to M_t$ has finite rank.
  Persistence modules that are q-tame have well-defined \emph{persistence diagrams}~\cite{Bauer:2015,cdsgo:book}. For a persistence module $M$ that is $q$-tame, we denote its persistence diagram by $D(M)$. Let $\eps \geq 0$. An \emph{$\eps$-matching} between persistence diagrams is a partial bijection such that matched pairs are within $\eps$ of each other and unmatched elements are within $\eps$ of the diagonal.
  The \emph{bottleneck distance} between two persistence diagrams is the infimum of all $\eps \geq 0$ such that there exists an $\eps$-matching between them~\cite{cseh:stability}.

\subsection{Filtered closure spaces  from metric spaces}
\label{sec:fcl-metric}

Here we consider filtered closure spaces \new{using the} closure operators \new{on} a metric space in \cref{def:closure_induced_by_a_metric}.

Recall the poset $[0,\infty) \times \{-1,0,1\}$ with the lexicographic order from \cref{section:metric_spaces_and_closures} and the corresponding closure operators (\cref{def:closure_induced_by_a_metric}).
Given a metric space $(X,d)$, there is a $[0,\infty) \times \{-1,0,1\}$-filtered symmetric closure space
$(X,\{c_{\eps}\}_{\eps \in [0,\infty) \times \{-1,0,1\}})$.
This filtered closure space restricts to
the $[0,\infty)$-filtered symmetric Alexandroff closure space
$(X,\{c_{\eps^-}\}_{\eps \in [0,\infty)})$,
the $[0,\infty)$-filtered symmetric Alexandroff closure space
$(X,\{c_{\eps}\}_{\eps \in [0,\infty)})$,
the $[0,\infty)$-filtered symmetric closure space
$(X,\{c_{\eps^+}\}_{\eps \in [0,\infty)})$.

The latter three are persistence modules indexed by $[0,\infty)$ with values in $\cat{Cl_s}$. 
For each of the three distinct pairs in the set $\{-1,0,1\}$, 
$(X,\{c_{\eps}\}_{\eps \in [0,\infty) \times \{-1,0,1\}})$
restricts to a persistence module indexed by $I_{0^+}$, which shows that each of the three $[0,\infty)$-indexed persistence modules have pairwise interleaving distance zero. 
Furthermore,
$(X,\{c_{\eps}\}_{\eps \in [0,\infty) \times \{-1,0,1\}})$
gives a coherent interleaving of these three persistence modules~\cite{bdsn}.
Combining this construction with \cref{lem:metric-map-to-closure-map},
and using the next definition we obtain the subsequent result.

\begin{definition}
Say that $F,G: \cat{C} \to \cat{D}^{\new{P}}$ are \emph{objectwise interleaved} if for all $C \in \cat{C}$, $F(C)$ and $G(C)$ are interleaved.
\end{definition}

\begin{theorem} \label{thm:metric-fcl}
   Let $\cat{Met} \to \cat{F_{[0,\infty) \times \{-1,0,1\}}^{\new{\subset}} Cl_s}$ be the functor defined by mapping $(X,d)$ to $(X,c_{\bullet,d})$ and mapping $f:(X,d) \to (Y,e)$ to $f:(X,c_{\bullet,d}) \to (Y,c_{\bullet,e})$.
Restricting to $-1$ and $0$, this functor specializes to two functors
$\cat{Met} \to \cat{F_{[0,\infty)}^{\new{\subset}} Cl_{sA}}$.
Restricting to $1$ this functor specializes to a functor
$\cat{Met} \to \cat{F_{[0,\infty)}^{\new{\subset}} Cl_{s}}$.
The objectwise interleaving distance between the resulting three functors $\cat{Met} \to \cat{F_{[0,\infty)}^{\new{\subset}} Cl_s}$ is zero and these interleavings are coherent.
\end{theorem}

Let $H_{\bullet}$ be one of the cubical singular homology functors,
$H_{\bullet}^{(I_{\tau},\times)}$,
$H_{\bullet}^{(J_{\top},\times)}$,
$H_{\bullet}^{(J_+,\times)}$,
$H_{\bullet}^{(I_{\tau},\boxplus)}$,
$H_{\bullet}^{(J_{\top},\boxplus)}$, or
$H_{\bullet}^{(J_+,\boxplus)}$,
from \cref{sec:cubical}
or one of the singular simplicial homology functors,
$H_{\bullet}^{I_{\tau}}$,
$H_{\bullet}^{J_{\top}}$,
$H_{\bullet}^{J_+}$,
from \cref{sec:simplicial},
each with coefficients in the field $k$ (\cref{sec:coefficients}).
By composing
the functors in \cref{thm:metric-fcl} with one of these homology functors we obtain the following~\cite{bubenik2015metrics,bubenik2017interleaving}.

\begin{corollary} \label{cor:metric-pm}
     For each of our cubical and simplicial singular homology functors and for each $j \geq 0$, there is a functor $\cat{Met} \to \cat{Vect^{[0,\infty) \times \{-1,0,1\}}}$, which maps $(X,d)$ to $H_j(X,c_{\bullet,d})$ and $f:(X,d) \to (Y,e)$ to $f_*:H_j(X,c_{\bullet,d}) \to H_j(Y,c_{\bullet,e})$.
Choosing one of $\{-1,0,1\}$, this functor specializes to three functors
$\cat{Met} \to \cat{Vect^{[0,\infty)}}$,
the objectwise interleaving distance between these functors is zero,
and these interleavings are coherent.
\end{corollary}

For the remainder of this section, we restrict to the case $0 \in \{-1,0,1\}$. That is, we consider the closures $(X,c_{r,d})$ for $r \geq 0$,
where for $A \subset X$, $c_{r,d}(A) = \{x \in X \ | \ d(a,x) \leq r \text{ for some } a \in A \}$.

\subsection{Full subcategories of \new{filtrations of} closure spaces}

We consider a sequence of generalizations from metric spaces to \new{filtrations of} closure spaces and their corresponding persistence modules.

\begin{definition}
\label{def:lawvere_metric}
Let $\mathbf{Lawv}$ denote the category of small Lawvere metric spaces, i.e., extended quasi-pseudo-metric spaces, and $1$-Lipschitz maps.
\end{definition}

\begin{definition}
\label{def:weighted_digraphs}
Let $P$ be a poset.
Recall that we may identify a simple digraph $(X,E)$ with its corresponding spatial digraph $(X,\overline{E})$, where $\overline{E} = E \cup \Delta$.
A $P$-weighted digraph consists of a (simple) digraph $(X,E)$ and a function $w: \overline{E} \to P$, called the \emph{weight}, such that for $xEx'$, $w(x,x),w(x',x') \leq w(x,x')$.
For $xEx'$ we call $w(x,x')$ the weight of the directed edge $(x,x')$ and for $x \in X$ we call $w(x,x)$ the weight of the vertex $x$, which we also denote by $w(x)$.
A morphism of weighted digraphs $f:(X,E,w) \to (Y,F,v)$ is a digraph homomorphisms $f:(X,E) \to (Y,F)$ such that
for all $x\overline{E}x'$,
$v(f(x),f(x'))\le w(x,x')$.
Let $\mathbf{w_P DiGph}$ be the category of $P$-weighted digraphs $(X,E,w)$ and their morphisms.
\end{definition}

\begin{proposition} \label{prop:embeddings-R}
We have the following full embeddings of categories
\[
\cat{Met} \hookrightarrow \cat{Lawv} \hookrightarrow \cat{w_{\new{\R}} DiGph} \hookrightarrow \cat{F_{\new{\R}}^{rc,\new{\subset}}DiGph} \isomto \cat{F_{\new{\R}}^{rc,\new{\subset}} Cl_{A}} \hookrightarrow \cat{F_{\new{\R}}Cl}, 
\]
whose composition sends the metric space $(X,d)$ to the $\new{\R}$-filtered closure space $(X_{\bullet},c_{\bullet,d})$,
where $X_a = \varnothing$ if $a < 0$ and $X_a=X$ if $a\geq 0$.
\end{proposition}

\new{
\begin{proof}
    For each small Lawvere metric space $(X,d)$, we have the digraph with vertex set $X$ and edges consisting of all ordered pairs $(x,x')$ such that $d(x,x') < \infty$, and weight given by $d$.
    For each $\R$-weighted digraph $(X,E,w)$, we have the corresponding right continuous $\R$-filtered digraph given by $X_a = \{x \in X \ | \ w(x) \leq a\}$ and $E_a = \{(x,x') \in E \ | \ w(x,x') \leq a\}$.
\end{proof}
}

Note that the image of the composition $\cat{Met} \hookrightarrow \cat{Lawv} \hookrightarrow \cat{w_{\new{\R}} DiGph} \hookrightarrow \cat{F_{\new{\R}}^{rc,\new{\subset}} DiGph}$ lies in $\cat{F_{\new{\R}}^{rc,\new{\subset}} Gph}$.
In particular, a metric space $(X,d)$ is mapped to the right continuous $\new{\R}$-filtered graph $(X_{\bullet},E_{\bullet})$, where
for $x \in X$, $x \in X_t$ iff $t \geq 0$ and
for $x,x' \in X$, $x \overline{E}_t x'$ iff $d(x,x') \leq t$.

Composing the functors in \cref{prop:embeddings-R} with the functors induced by homology, we have the following.

\begin{corollary} \label{cor:functor-to-pm-2}
  For each of our cubical and simplicial singular homology functors and for each of the categories $\cat{Met}$,
  $\cat{Lawv}$, $\cat{w_{\new{\R}} DiGph}$, 
  $\cat{F_{\new{\R}}Cl}$, we have a functor to $\cat{Vect^{\new{\R}}}$.
\end{corollary}

\subsection{Lipschitz maps}

We now extend \new{the results of the previous section} to the case of Lipschitz maps.
\new{Since we are allowing additional morphisms of metric spaces, we will also need additional morphisms of weighted digraphs, filtered digraphs and filtered closure spaces.}

\begin{definition} \label{def:Lipschitz}
Let $\cat{Lip}$ denote the category of metric spaces and Lipschitz maps.
Let $\cat{LipLawv}$ denote the category of small Lawvere metric spaces and Lipschitz maps.
\end{definition} 


  


\begin{definition} \label{def:plus}
  Let $\cat{w_{[\infty,\infty)}^+ DiGph}$ denote the category of $[-\infty,\infty)$-weighted digraphs and maps $f:(X,E,w) \to (Y,F,v)$ given by digraph homomorphisms $f:(X,E) \to (Y,F)$ such that
there exists an $L \geq 0$ such that
for all $x\overline{E}x'$, $v(f(x),f(x')) \leq w(x,x') + L$.

Let $\cat{F_{\new{\R}}^{\new{\subset},+} DiGph}$ denote the category of $\R$-filtered digraphs together with morphisms $f:(X_{\bullet},E_{\bullet}) \to (Y_{\bullet},F_{\bullet})$
consisting of functions $f:X \to Y$, where $X = \bigcup_{t \in \R} X_t$ and $Y = \bigcup_{t \in \R} Y_t$, 
such that there exists an $L \geq 0$ for which for all $t$,
$f(X_t) \subset Y_{t+L}$ and
whenever $x\overline{E}x'$ we have that $f(x)\overline{F}_{t+L}f(x')$.

Similarly, let $\cat{F^{\new{\subset},+}_{\new{\R}} Cl}$ denote the category of
$\new{\R}$-filtered
closure spaces together with morphisms $f:(X_{\bullet},\new{c_{X_{\bullet}}}) \to (Y_{\bullet},\new{c_{Y_{\bullet}}})$
consisting of functions $f:X \to Y$, 
such that there exists an $L \geq 0$ for which for all $t$, $f(X_t) \subset Y_{t+L}$ and for all $A \subset X_t$, $f(\new{c_{X_t}}(A)) \subset \new{c_{Y_{t+L}}}(f(A))$. Say that $f$ has \emph{shift} $L$.
\end{definition}






\begin{proposition} \label{prop:embeddings-Lip}
  We have the following full embeddings of categories
\[\cat{Lip} \hookrightarrow \cat{LipLawv} \hookrightarrow \cat{w_{[-\infty,\infty)}^{+}DiGph} \hookrightarrow \cat{F^{rc,\new{\subset},+}_{\new{\R}}DiGph} \isomto \cat{F_{\new{\R}}^{rc,\new{\subset},+}Cl_{A}} \hookrightarrow \cat{F^{\new{\subset},+}_{\new{\R}}Cl}\]
whose composition sends the metric space $(X,d)$ to $(X,c_{\exp(\bullet),d})$
and sends a map $f$ with Lipschitz constant $K\geq 1$ to a map with shift $\log K$.
\end{proposition}

\begin{proof}
\new{
  The second functor 
  is defined on objects by mapping $(X,d)$ to $(X,E,w)$, where $E = \{(x,x') \in X \times X \ | \ x \neq x', d(x,x') < \infty$ and $w:E \cup \Delta: [-\infty,\infty)$ is given by $w = \log d$, where $\log 0 = -\infty$.
  It is defined on morphisms by sending the function $f$ to the function $f$.
  Note that $d_Y(fx,fx') \leq Kd_X(x,x')$ implies that $\log d_Y(fx,fx') \leq \log d_X(x,x') + \log K$.
}

\new{
  The third functor is defined in the same way as the corresponding functor in \cref{prop:embeddings-R}.
  Note that for $f:(X,E,w) \to (Y,F,v)$, $v(fx) \leq w(x) + L$ implies that $f(X_a) \leq Y_{a+L}$ for all $a$, 
  and $v(fx,fx') \leq w(x,x') + L$ implies that if $x E_a x'$ then $(fx)F_{a+L}(fx')$.
}

\new{
  Composing these functors, $(X,d)$ is mapped to $(X_{\bullet},E_{\bullet})$, where $X_a=X$ for all $a$, and
  $E_a = \{(x,x') \ | \ \log d(x,x') \leq a\} = \{(x,x') \ | \ d(x,x') \leq \exp a\}$.
  This filtered digraph corresponds to the filtered closure space $(X,c_{\exp(\bullet),d})$.
  In addition, a $K$-Lipschitz map is sent to a morphism with shift $\log K$.
  }
%
%
\end{proof}

It follows directly from the definitions that
$(X_{\bullet},\new{c_{X_{\bullet}}}), (Y_{\bullet},\new{c_{Y_{\bullet}}}) \in \cat{F_{\new{\R}} Cl}$
are isomorphic if and only if
they are $0$-interleaved (\cref{def:interleaving_distance}).
More generally, we have the following.

\begin{theorem} \label{thm:isom-interleaving}
  Let $L \geq 0$.
  Then $(X_{\bullet},\new{c_{X_{\bullet}}}), (Y_{\bullet},\new{c_{Y_{\bullet}}}) \in \cat{F_{\new{\R}}^{\new{\subset},+} Cl}$
  are isomorphic via maps with shift $L$ if and only if
  they are $L$-interleaved.
  In particular, $(X,d), (Y,e) \in \cat{Lip}$ are isomorphic via maps with Lipschitz constant $K \geq 1$ if and only if $(X,c_{\exp(\bullet),d})$ and $(Y,c_{\exp(\bullet),e})$ are $\log(K)$-interleaved.
\end{theorem}

\begin{proof}
 Let $(X_{\bullet},\new{c_{X_{\bullet}}}), (Y_{\bullet},\new{c_{Y_{\bullet}}}) \in \cat{F_{\new{\R}}^+ Cl}$.
  Let $X = \bigcup_t X_t$ and $Y = \bigcup_t Y_t$.

  $(\Rightarrow)$
  There exist functions $f:X \to Y$ and $g:Y \to X$ such that
  for all $t$,
  $f|_{X_t}: (X_t,c_{\new{X_t}}) \to (Y_{t+L},c_{\new{Y_{t+L}}})$ and 
  $g|_{Y_t}: (Y_t,\new{c_{Y_t}}) \to (X_{t+L},\new{c_{X_{t+L}}})$,
  $gf = 1_X$, and $fg = 1_Y$.
  These maps provide the desired $L$-interleaving.

  $(\Leftarrow)$
  For all $t$ we have maps
  $f_t: (X_t,\new{c_{X_t}}) \to (Y_{t+L},\new{c_{Y_{t+L}}})$ and 
  $g_t: (Y_t,\new{c_{Y_t}}) \to (X_{t+L},\new{c_{X_{t+L}}})$ satisfying the four conditions in \cref{def:interleaving_distance}.
  By \cref{def:interleaving_distance}\eqref{it:interleaving1},
  the maps $\{f_t\}$ define a function $f:X \to Y$ with $f|_{X_t} = f_t$.
  By \cref{def:interleaving_distance}\eqref{it:interleaving2},
  the maps $\{g_t\}$ define a function $g:Y \to X$ with $g|_{Y_t} = g_t$.
  By \cref{def:interleaving_distance}\eqref{it:interleaving3},
  $gf = 1_X$.
  By \cref{def:interleaving_distance}\eqref{it:interleaving4},
  $fg = 1_Y$.
  Therefore we have the desired isometry.
\end{proof}

\section{Stability} \label{sec:stability}

\new{Given $\eps \geq 0$, we define the notion of $\eps$-correspondence of filtered closure spaces. Using this definition we obtain a distance for filtrations of closure spaces that generalizes the Gromov-Hausdorff distance for metric spaces. 
We prove that applying any homotopy-invariant functor to a pair of filtered closure spaces with an $\eps$-correspondence, produces $\eps$-interleaved persistence modules.
As a consequence, we obtain stability theorems that generalize many of the existing stability theorems for $\R$-indexed persistence modules.}


\subsection{Sublevel sets}

We start by defining sublevel set filtrations and showing that they are stable.
The following is a special case of \cref{ex:fcl}.

\begin{definition}
  Let $(X,c) \in \cat{Cl}$ and let $f:X \to \R$ be a set map.
  Define $\Sub(f) \in \cat{F_{\new{\R}}^{\new{\subset}} Cl}$ to be given by $\Sub(f)_t = f^{-1}(-\infty,t]$ together with the subspace closure.
\end{definition}

For a closure space $(X,c)$ and $f,g:X \to \R$, let $d_{\infty}(f,g) = \sup_{x \in X}\abs{f(x)-g(x)}$.
Then for $\eps = d_{\infty}(f,g)$, $\Sub(f),\Sub(g) \in \cat{F_{\new{\R}}^{\new{\subset}} Cl}$ are $\eps$-interleaved. It follows, \new{by the functoriality of homology}, that for any of our cubical and simplicial singular homology theories, $H_j(\Sub(f))$ and $H_j(\Sub(g))$ are $\eps$-interleaved~\cite{bubenik2015metrics,bubenik2017interleaving}.
Furthermore, if coefficients are in a field and $H_j(\Sub(f))$ and $H_j(\Sub(g))$ are q-tame, then there is an $\eps$-matching between $D(H_j(\Sub(f)))$ and $D(H_j(\Sub(g)))$~\cite{Bauer:2015,cdsgo:book}.

\begin{theorem}[Sublevel set Stability Theorem] \label{thm:sublevelset-stability}
  Let $(X,c) \in \cat{Cl}$ and $f,g:X \to \R$. Let $H$ denote one of our cubical or simplicial singular homology theories and let $j \geq 0$. Then
  \[
    d_I(H_j(\Sub(f)), H_j(\Sub(g))) \leq d_{\infty}(f,g).
  \]
  If coefficients are in a field and $H_j(\Sub(f))$ and $H_j(\Sub(g))$ are q-tame then
  \[
    d_B(D(H_j(\Sub(f))), D(H_j(\Sub(g)))) \leq d_{\infty}(f,g).
  \]  
\end{theorem}

\subsection{Correspondences and Gromov-Hausdorff distance}

Next, we use multivalued maps and correspondences to define a Gromov-Hausdorff distance for filtered closure spaces.
\new{For a product $X \times Y$,
  denote the canonical projections by
  $\pi_1:X \times Y \to X$ and $\pi_2:X \times Y \to Y$.
}

\begin{definition}
\label{def:multivalued_map}
\new{Given sets $X$ and $Y$, a \emph{multivalued map} from $X$ to $Y$,
denoted $C:X\rightrightarrows Y$, is a subset of $X\times Y$, also denoted $C$,
such that $\pi_1(C) = X$.
If $(x,y) \in C$, we write $xCy$.
}
\new{Given $A \subset X$, let $C(A) = \pi_2(C \cap (A \times Y))$.}
That is, 
  $C(A) = \{y \in Y \ | \ \exists a \in A, aCy\}$.
\end{definition}

\begin{definition}
\label{def:subordinate_to_a_multivalued_map}
A \emph{(single-valued)} map \new{$f:X\to Y$} is \emph{subordinate} to \new{a multivalued map $C:X\rightrightarrows Y$} if we have $(x,f(x))\in C$ for every $x\in X$. In that case we write $f:X\xrightarrow{C}Y$. The composition of two multivalued maps $C:X\rightrightarrows Y$ and $D:Y\rightrightarrows Z$ is the multivalued map $D\circ C:X\rightrightarrows Z$ defined by:
\[x(D\circ C)z\Longleftrightarrow \exists y\in Y, \, xCy, yDz\]
\end{definition}

\begin{lemma} \label{lem:subordinate}
  Let $X \rightrightarrows Y$ be a multivalued map. Then there exists a single-valued map $f:X \to Y$ subordinate to $C$.
\end{lemma}

\begin{proof}
  For each $x \in X$, choose
  \new{$f(x) \in Y$ such that $xCf(x)$.}
\end{proof}

Note that if $f$ is subordinate to $C$ and $g$ is subordinate to $D$ then $g \circ f$ is subordinate to $D \circ C$.
In addition, if $A \subset X$, then $(D \circ C)(A) = D(C(A))$.

\begin{definition}
\label{def:transpose_of_a_multivalued_map}
If $C:X\rightrightarrows Y$ is a multivalued map, the \emph{transpose} of $C$, denoted $C^T$, is the image of $C$ through the symmetry map $(x,y)\mapsto (y,z)$.
\end{definition}

\begin{remark} 
\new{Although $C^T$ is well-defined as a subset of $Y\times X$, it is not always a multivalued map because it may not project surjectively onto $Y$.}
\end{remark}

\begin{definition}
\label{def:correspondences}
A multivalued map $C:X\rightrightarrows Y$ is a \emph{correspondence} if
\new{$\pi_2(C) = Y$},
or equivalently, if $C^T$ is also a multivalued map.
\end{definition}

Note that $C:X\rightrightarrows Y$ is a correspondence \new{if and only if the} identity maps $\mathbf{1}_X$ and $\mathbf{1}_Y$ are subordinate to the compositions $C^T\circ C$ and $C\circ C^T$, respectively.

\new{Next we use correspondences to define the Gromov-Hausdorff distance for metric spaces. This definition agrees with the definition that embeds metric spaces into a common metric space and uses the Hausdorff distance in that space~\cite[Theorem 7.3.25]{bbi:book}.}

\begin{definition}
\label{def:GH-metric}
  Let $(X,d), (Y,e) \in \cat{Met}$ and let $\eps \geq 0$.
    Say that $C: X \rightrightarrows Y$ is a \emph{metric $\eps$-multivalued map} if whenever $xCy$ and $x'Cy'$ then $e(y,y') \leq d(x,x') + \eps$.
    Say that a correspondence $C: X \rightrightarrows Y$ is a \emph{metric $\eps$-correspondence} if $C$ and $C^T$ are metric $\eps$-multivalued maps.
    The \emph{metric distortion} of a correspondence $C: X \rightrightarrows Y$ is given by $\mdist(C) = \inf\{ \eps \geq 0 \ | \ C \text{ is a metric $\eps$-correspondence}\}$, with $\mdist(C) = \infty$ if there is no such $\eps$.
    The \emph{Gromov-Hausdorff distance} between $(X,d)$ and $(Y,e)$ is given by $d_{GH}((X,d),(Y,e)) = \frac{1}{2} \inf \mdist(C)$, where the infimum is taken over all correspondences $C: X \rightrightarrows Y$~\cite[\new{Definition 7.3.10}]{bbi:book}.
\end{definition}

  
Recall that $\new{\R}$ denotes the totally ordered set $(\R,\leq)$
and for $(X_{\bullet},c_{\bullet}) \in \cat{F_{\new{\R}}^{\new{\subset}} Cl}$, $X = \bigcup_t X_t$.

\begin{definition} \label{def:eps-multivalued-map}
  Let $(X_{\bullet},\new{c_{X_{\bullet}}}), (Y_{\bullet},\new{c_{Y_{\bullet}}}) \in \cat{F_{\new{\R}}^{\new{\subset}} Cl}$.
  A multivalued map from $(X_{\bullet},\new{c_{X_{\bullet}}})$ to $(Y_{\bullet},\new{c_{Y_{\bullet}}})$
  is a multivalued map $C: X \rightrightarrows Y$.
  Write $C: (X_{\bullet},\new{c_{X_{\bullet}}}) \rightrightarrows (Y_{\bullet},\new{c_{Y_{\bullet}}})$.
  Similarly, a correspondence from $(X_{\bullet},\new{c_{X_{\bullet}}})$ to $(Y_{\bullet},\new{c_{Y_{\bullet}}})$ is a correspondence $C: X \rightrightarrows Y$.
 Let $\eps \geq 0$.
 Say that a multivalued map $C: (X_{\bullet},\new{c_{X_{\bullet}}}) \rightrightarrows (Y_{\bullet},\new{c_{Y_{\bullet}}})$ is an \emph{$\eps$-multivalued map}
 if
 for all $t$,
 \begin{enumerate}
 \item whenever $x \in X_t$ and $xCy$ we have that $y \in Y_{t+\eps}$, and
 \item whenever $A \subset X_t$ and $f:X \xto{C} Y$ we have that $C(\new{c_{X_t}}(A)) \subset \new{c_{Y_{t+\eps}}}(f(A))$.
 \end{enumerate}
 Say that a correspondence $C: (X_{\bullet},\new{c_{X_{\bullet}}}) \rightrightarrows (Y_{\bullet},\new{c_{Y_{\bullet}}})$ is an \emph{$\eps$-correspondence}
 if $C : (X_{\bullet},\new{c_{X_{\bullet}}}) \rightrightarrows (Y_{\bullet},\new{c_{Y_{\bullet}}})$ is an $\eps$-multivalued map and
 if $C^T : (Y_{\bullet},\new{c_{Y_{\bullet}}}) \rightrightarrows (X_{\bullet},\new{c_{X_{\bullet}}})$ is an $\eps$-multivalued map.
\end{definition}

\begin{lemma} \label{lem:eps-multivalued-map}
A multivalued map $C: (X_{\bullet},\new{c_{X_{\bullet}}}) \rightrightarrows (Y_{\bullet},\new{c_{Y_{\bullet}}})$ is an \emph{$\eps$-multivalued map}
 iff
for all $t$,
 \begin{enumerate}
 \item whenever $x \in X_t$ and $xCy$ we have that $y \in Y_{t+\eps}$, and
 \item for all
   \new{$S \subset C \cap (X_t \times Y)$},
   $C (c_{X_t} (\pi_{\new{1}} S)) \subset c_{Y_{t+\eps}} (\pi_{\new{2}} S)$.
 \end{enumerate}
\end{lemma}

\begin{proof}
  $(\Leftarrow)$ Let $A \subset X_t$ and $f:X \xto{C} Y$.
  Let $S = \{(a,f(a) \ | \ a \in A\}$.
  Then $\pi_{\new{1}}(S) = A$, $\pi_{\new{2}}(S) = f(A)$,
  \new{and $S \subset C \cap (X_t \times Y)$}.
  Hence $C (c_{\new{X_t}}(A))
  \new{= C(c_{X_t}(\pi_1 S))}
  \subset
  \new{c_{Y_{t+\eps}}(\pi_2 S) = 
  c_{Y_{t+\eps}}}(f(A))$.

  $(\Rightarrow)$ Let
  \new{$S \subset C \cap (X_t \times Y)$.}
  Let $A = \pi_{\new{1}}(S)$. Then $A \subset X_t$. 
  Choose $f:X \xto{C} Y$ \new{such that for all $a \in A$, $(a,f(a)) \in S$}.
  Then $f(A) \subset \pi_{\new{2}}(S)$.
  Therefore $C (c_{\new{X_t}} (\pi_{\new{1}} S)) \subset \new{c_{Y_{t+\eps}}} (f(A)) \subset \new{c_{Y_{t+\eps}}} (\pi_{\new{2}} S)$.
\end{proof}

\begin{lemma} \label{lem:compose-eps-mvm}
  If $C:(X_{\bullet},\new{c_{X_{\bullet}}}) \rightrightarrows (Y_{\bullet},\new{c_{Y_{\bullet}}})$ is an $\eps$-multivalued map and $D:(Y_{\bullet},\new{c_{Y_{\bullet}}}) \rightrightarrows (Z_{\bullet},\new{c_{Z_{\bullet}}})$ is a $\delta$-multivalued map, then $D \circ C:(X_{\bullet},\new{c_{X_{\bullet}}}) \rightrightarrows (Z_{\bullet},\new{c_{Z_{\bullet}}})$ is an $(\eps+\delta)$-multivalued map.
\end{lemma}

\begin{proof}
  We will use \new{\cref{lem:eps-multivalued-map}}.
  \begin{enumerate}
  \item For all $t$ and $x \in X_t$, if $xDCz$ then there is a $y \in Y$ such that $xCy$, $yDz$. Since $x \in X_t$, $y \in Y_{t+\eps}$ and thus $z \in Z_{t+\eps+\delta}$.
  \item 
\new{Let $S \subset DC \cap (X_t \times Z)$.}
    We want to show that 
    $DC(\new{c_{X_t}}(\pi_{\new{1}}(S))) \subset \new{c_{Z_{t+\eps+\delta}}}(\pi_{\new{2}}(S))$.
   \new{Let $T = C \cap (\pi_1 S \times Y)$.}
    Then $\pi_{\new{1}} (T) = \pi_{\new{1}} (S)$ 
   and 
   \new{$T \subset C \cap (X_t \times Y)$.}
    Thus $\pi_{\new{2}}(T) \subset Y_{t+\eps}$ and
    $(DC)(c_{\new{X_t}} (\pi_{\new{1}}(S))) = D(C(\new{c_{X_t}} (\pi_{\new{1}} T))) \subset D(\new{c_{Y_{t+\eps}}} (\pi_{\new{2}} T))$.
    \new{Let $U = D \cap (\pi_2 T \times Z)$.}
    Then $\pi_{\new{1}} (U) = \pi_{\new{2}} (T)$
    and 
    \new{$U \subset D \cap (Y_{t + \eps} \times Z)$.}
      Thus 
    $D(\new{c_{Y_{t+\eps}}} (\pi_{\new{2}} (T))) = D(\new{c_{Y_{t+\eps}}} (\pi_{\new{1}} (U))) \subset \new{c_{Z_{t+\eps+\delta}}} (\pi_{\new{2}} (U))$.
    It remains to show that $\pi_{\new{2}} (U) = \pi_{\new{2}} (S)$.
    \begin{align*}
      \pi_{\new{2}}(S) &= \{z \in Z \ | \ \exists x \in X, \exists y \in Y, xCy, yDz, x \in \pi_{\new{1}}(S)\}\\
    &= \{z \in Z \ | \ \exists (x,y) \in T, yDz\}\\
    &= \{z \in Z \ | \ \exists y \in \pi_{\new{2}} T, yDz\}\\
    &= 
    \new{\pi_2 (D \cap (\pi_2 T \times Z))}\\
    &= \pi_{\new{2}} (U).\qedhere
    \end{align*}
  \end{enumerate}
\end{proof}

\begin{definition}
Let $(X_{\bullet},\new{c_{X_{\bullet}}}), (Y_{\bullet},\new{c_{Y_{\bullet}}}) \in \cat{F_{\new{\R}}^{\new{\subset}} Cl}$.
 For a correspondence $C: (X_{\bullet},\new{c_{X_{\bullet}}}) \rightrightarrows (Y_{\bullet},\new{c_{Y_{\bullet}}})$ 
  define the \emph{distortion} of $C$ by $\dist(C) = \inf \{ \eps \geq 0 \ | \ C \text{ is an } \eps\text{-correspondence}\}$, where $\dist(C) = \infty$ if there is no such $\eps$.
\end{definition}

\begin{definition} \label{def:GH-FRCl}
  Let $(X_{\bullet},\new{c_{X_{\bullet}}}), (Y_{\bullet},\new{c_{Y_{\bullet}}}) \in \cat{F_{\new{\R}}^{\new{\subset}} Cl}$.
  Define the \emph{Gromov-Hausdorff distance} between $(X_{\bullet},\new{c_{X_{\bullet}}})$ and $(Y_{\bullet},\new{c_{Y_{\bullet}}})$ to be given by
  $d_{GH}((X_{\bullet},\new{c_{X_{\bullet}}}), (Y_{\bullet},\new{c_{Y_{\bullet}}})) = \frac{1}{2} \inf \dist(C)$, where the infimum is taken over all correspondences $C: (X_{\bullet},\new{c_{X_{\bullet}}}) \rightrightarrows (Y_{\bullet},\new{c_{Y_{\bullet}}})$.
\end{definition}

It follows from \cref{lem:compose-eps-mvm} that $d_{GH}$ satisfies the triangle inequality.
  Also note that since the definition of $\eps$-correspondence is symmetric, $d_{GH}$ is symmetric.
Furthermore if $(X_{\bullet},\new{c_{X_{\bullet}}}) \isom (Y_{\bullet},\new{c_{Y_{\bullet}}})$ then
$d_{GH}((X_{\bullet},\new{c_{X_{\bullet}}}), (Y_{\bullet},\new{c_{Y_{\bullet}}})) = 0$.
\new{However, $d_{GH}((X_{\bullet},\new{c_{X_{\bullet}}}), (Y_{\bullet},\new{c_{Y_{\bullet}}})) = 0$ does not imply that $(X_{\bullet},\new{c_{X_{\bullet}}}) \isom (Y_{\bullet},\new{c_{Y_{\bullet}}})$.
For example, consider 
$X_a = \varnothing$ if $a <0$ and $X_a = *$ if $a \geq 0$, and 
$Y_a = \varnothing$ if $a \leq 0$ and $Y_a = *$ if $a > 0$.
}

\begin{theorem} \label{thm:gh-metric}
The Gromov-Hausdorff distance is an extended pseudometric on 
\new{any set of
$\R$-filtrations of} closure spaces.
\end{theorem}

\new{
\begin{proof}
    Let $X,Y \in \cat{F_{\R} Cl}$. 
    As observed in \cref{sec:filtrations}, a choice of colimits of $X$ and $Y$ determines $\R$-filtered closure spaces $X'$ and $Y'$, with $X' \isom X$ and $Y' \isom Y$.
    Define $d_{GH}(X,Y) = d_{GH}(X',Y')$, which is well defined
    since $d_{GH}$ is an isomorphism invariant.
\end{proof}
}

\begin{proposition} \label{prop:gh-dinfty}
Let $(X,c) \in \cat{Cl}$ and let $f,g:X \to \R$. Then $d_{GH}(\Sub(f),\Sub(g)) \leq \frac{1}{2}d_{\infty}(f,g)$.
\end{proposition}

\begin{proof}
 Let $\eps = d_{\infty}(f,g)$.
  Consider the correspondence of $\Sub(f)$ and $\Sub(g)$ given by the diagonal $\Delta = \{(x,x) \ | \ x \in X\} \subset X \times X$.
  By assumption, for all $t$, $\Sub(f)_t \subset \Sub(g)_{t+\eps}$ and $\Sub(g)_t \subset \Sub(f)_{t+\eps}$.
  Furthermore, the only function $f:X \to X$ subordinate to $\Delta$ is the identity function.
  It follows that $\Delta$ is an $\eps$-correspondence.
\end{proof}

The following example shows that the inequality in the previous proposition may be strict.

\begin{example}  Let $(X,c) = (\{0,1\},c_{\bot})$.
  Let $f,g:X \to \R$ be given by $f(i) = i$ and $g(i) = 1-i$, respectively.
  Then $d_{\infty}(f,g) = 1$.
  Let $C = \{(0,1),(1,0)\} \subset X \times X$.
  We will use \cref{def:eps-multivalued-map} to
  show that $C$ is a $0$-correspondence.
  Note that $\Sub(f)_t = \varnothing$ if $t< 0$, $\Sub(f)_t = \{0\}$ if $0 \leq t < 1$ and $\Sub(f)_t = X$ if $t \geq 1$.
  Similarly, $\Sub(g)_t = \varnothing$ if $t< 0$, $\Sub(g)_t = \{1\}$ if $0 \leq t < 1$ and $\Sub(g)_t = X$ if $t \geq 1$.
  Thus, if $xCy$ then $x \in \Sub(g)_t$ iff $y \in \Sub(f)_t$.
  The only function $h:X \to X$ subordinate to $C$ is given by $h(0)=1$ and $h(1)=0$.
  Therefore for all $A \subset X$, $C(A) = h(A)$ and $C^T(A) = h(A)$.
  Thus $C$ is a $0$-correspondence. 
  Hence $2 d_{GH}(\Sub(f),\Sub(g)) = 0 < d_{\infty}(f,g)=1$.
\end{example}

Recall (\cref{prop:embeddings-R}) that $\cat{Met}$ is a full subcategory of $\cat{F_{\new{\R}}^{\new{\subset}} Cl}$ given by mapping a metric space $(X,d)$ to $(X_{\bullet},c_{\bullet,d})$, where
for $t < 0$, $X_t = \varnothing$ and for $t \geq 0$, $X_t = X$ and
for $A \subset X$, $c_{t,d}(A) = \{x \in X \ | \ d(a,x) \leq t \text{ for some } a \in A \}$.

\begin{lemma}
  Let $(X,d),(Y,e) \in \cat{Met}$.
  Then $(X,c_{\bullet,d}), (Y,c_{\bullet,e}) \in \cat{F_{\new{\R}}^{\new{\subset}} Cl}$.
  Let $C: X \rightrightarrows Y$ and
  let $\eps \geq 0$.
  Then $C$ is a metric $\eps$-multivalued map iff $C$ is an $\eps$-multivalued map.
\end{lemma}

\begin{proof}
    $(\Rightarrow)$ 
  Since 
  $X_t \neq \varnothing$ implies that $t \geq 0$
and hence $Y_{t+\eps}=Y$, 
  the first condition of \cref{def:eps-multivalued-map} is trivial.
  For the second condition, let $A \subset X$ and $f:X \xto{C} Y$.
  Let $y\in C (c_{t,d}(A))$. Then,
    $ \exists a \in A, \exists x \in X, d(a,x) \leq t, xCy$.
  Since $aCf(a)$, it follows that $e(f(a),y) \leq d(a,x) + \eps \le t + \eps$ and thus $y\in c_{t+\eps,e} (f(A))$.
  Therefore $C (c_{t,d} (A)) \subset c_{t+\eps,e} (f(A))$.

 $(\Leftarrow)$
  Assume $xCy$, $x'Cy'$. Let $t = d(x,x')$ and let $S = \{(x,y)\} \subset C$.
  Since $x' \in c_{t,d} (\pi_{\new{1}}(S))$, we have that $y' \in C (c_{t,d} (\pi_{X} S))$, which by \cref{lem:eps-multivalued-map} is contained in $c_{t+\eps,e} (\pi_{\new{2}} S) = c_{t+\eps,e}(y)$.
  Therefore $e(y,y') \leq t+\eps$.
\end{proof}

It follows that restricted to metric spaces our definition of Gromov-Hausdorff distance for $\new{\R}$-filtered closure spaces, \cref{def:GH-FRCl}, agrees with the definition of the Gromov-Hausdorff distance for metric spaces, \cref{def:GH-metric}. 
\new{We define a \emph{Lawvere metric space} to be a category enriched over the monoidal poset $(([0,\infty],\geq),+,0)$~\cite{lawvere:metric}.}

\begin{theorem} \label{thm:gh}
The full embedding, $(\cat{Met},d_{GH})\to (\cat{F_{\new{\R}} Cl},d_{GH})$, 
\new{of Lawvere metric spaces},
is an isometric embedding.
\end{theorem}

\begin{proposition}
  Let $\eps \geq 0$.
  If $(X_{\bullet},\new{c_{X_{\bullet}}}), (Y_{\bullet},\new{c_{Y_{\bullet}}}) \in \cat{F_{\new{\R}}^{\new{\subset}} Cl}$ and
 $C:(X_{\bullet},\new{c_{X_{\bullet}}}) \rightrightarrows (Y_{\bullet},\new{c_{Y_{\bullet}}})$ is an $\eps$-multivalued map then
$C:(X_{\bullet},\qd(\new{c_{X_{\bullet}}})) \rightrightarrows (Y_{\bullet},\qd(\new{c_{Y_{\bullet}}}))$ is an $\eps$-multivalued map.
Furthermore,
  if $(X_{\bullet},\new{c_{X_{\bullet}}}), (Y_{\bullet},\new{c_{Y_{\bullet}}}) \in \cat{F_{\new{\R}}^{\new{\subset}} Cl_{A}}$ and
 $C:(X_{\bullet},\new{c_{X_{\bullet}}}) \rightrightarrows (Y_{\bullet},\new{c_{Y_{\bullet}}})$ is an $\eps$-multivalued map 
then
$C:(X_{\bullet},s(\new{c_{X_{\bullet}}})) \rightrightarrows (Y_{\bullet},s(\new{c_{Y_{\bullet}}}))$ is an $\eps$-multivalued map.
\end{proposition}

\begin{proof}
  The first statement follows from specializing the second condition in \cref{def:eps-multivalued-map} to $x \in X_t$.
  By \cref{lem:eps-multivalued-map}, we may write this specialized second condition as $x \in X_t$, $xCy$, $x' \in \new{c_{X_t}}(x)$, $x'Cy'$ implies that $y' \in \new{c_{Y_{t+\eps}}}(y)$. \new{By specializing this observation to those $x'\in c_{X_t}(x)$ such that $x\in c_{X_t}(x')$, we have that}
  $x \in X_t$, $xCy$, $x' \in s(\new{c_{X_t}}(x))$, $x'Cy'$ implies that $y' \in s(\new{c_{Y_{t+\eps}}}(y))$.
  By \cref{lem:eps-multivalued-map}, we obtain the desired result.
\end{proof}

\begin{corollary} \label{cor:dGH-inequalities}
  Let $(X_{\bullet},\new{c_{X_{\bullet}}}), (Y_{\bullet},\new{c_{Y_{\bullet}}}) \in 
  \cat{F_{\new{\R}} Cl}$.
  Then
  \begin{gather*}
  d_{GH}((X_{\bullet},\qd(\new{c_{X_{\bullet}}})),(Y_{\bullet},\qd(\new{c_{Y_{\bullet}}}))) \leq
  d_{GH}((X_{\bullet},\new{c_{X_{\bullet}}}), (Y_{\bullet},\new{c_{Y_{\bullet}}})),\textnormal{ and}\\
  d_{GH}((X_{\bullet},s(\qd(\new{c_{X_{\bullet}}}))), (Y_{\bullet},s(\qd(\new{c_{Y_{\bullet}}})))) \leq
  d_{GH}((X_{\bullet},\qd(\new{c_{X_{\bullet}}})), (Y_{\bullet},\qd(\new{c_{Y_{\bullet}}}))).
  \end{gather*}
\end{corollary}

\subsection{Correspondences, 
interleaving, \new{and stability}}
\label{subsection:correspondences_homotopy}
We connect elementary $(J_{\top},\times)$ homotopies, $\eps$-correspondences, and contiguous simplicial maps, 
\new{to obtain stability theorems}.

\begin{lemma} \label{lem:graph-homomorphism}
    Let $(X_{\bullet},\new{c_{X_{\bullet}}}), (Y_{\bullet},\new{c_{Y_\bullet}}) \in \cat{F_{\new{\R}}^{\new{\subset}} Cl}$.
  Let $C: (X_{\bullet},\new{c_{X_{\bullet}}}) \rightrightarrows (Y_{\bullet},\new{c_{Y_{\bullet}}})$ be an $\eps$-multivalued map.
  Let $f:X \xto{C} Y$.
  Then for all $t$, $f|_{X_t}: (X_t,\new{c_{X_t}}) \to (Y_{t+\eps},\new{c_{Y_{t+\eps}}})$ is  continuous.
\end{lemma}

\begin{proof}
  Let $x \in X_t$. Since $xCf(x)$, by \cref{def:eps-multivalued-map} we have that $f(x) \in Y_{t+\eps}$.
  Let $A \subset X_t$. Then, by \cref{def:eps-multivalued-map}, $f (c_t (A)) \subset C (\new{c_{X_t}} (A)) \subset \new{c_{Y_{t+\eps}}} (f(A))$.
\end{proof}

\begin{proposition} \label{prop:one-step-J1-homotopic}
  Let $(X_{\bullet},\new{c_{X_{\bullet}}}), (Y_{\bullet},\new{c_{Y_{\bullet}}}) \in \cat{F_{\new{\R}}^{\new{\subset}} Cl}$.
  Let $C: (X_{\bullet},\new{c_{X_{\bullet}}}) \rightrightarrows (Y_{\bullet},\new{c_{Y_{\bullet}}})$ be an $\eps$-multivalued map.
  For all $f$, $g$ subordinate to $C$, and for all $t \in \R$, $f|_{X_t},g|_{X_t}: (X_t,\new{c_{X_t}}) \to (Y_{t+ \eps}, \new{c_{Y_{t + \eps}}})$ are one-step $(J_{\top},\times)$-homotopic.
\end{proposition}

\begin{proof}
  By \cref{lem:graph-homomorphism}, $f|_{X_t},g|_{X_t}: (X_t,\new{c_{X_t}}) \to (Y_{t+ \eps}, \new{c_{Y_{t + \eps}}}) \in \cat{Cl}$.
Let $A \subset X_t$. By the continuity of $f|_{X_t}$ and $g|_{X_t}$, we have $\new{f|_{X_t}}(\new{c_{X_t}}(A)) \subset \new{c_{Y_{t+\eps}}}(\new{f|_{X_t}}(A))$ and $\new{g|_{X_t}}(\new{c_{X_t}}(A)) \subset \new{c_{Y_{t+\eps}}}(\new{g|_{X_t}}(A))$.
  Since $C$ is an $\eps$-multivalued map, we also have $\new{g|_{X_t}} (\new{c_{X_t}} (A)) \subset C (\new{c_{X_t}} (A)) \subset \new{c_{Y_{t+\eps}}} (\new{f|_{X_t}} (A))$ and
  similarly $\new{f|_{X_t}} (\new{c_{X_t}} (A)) \subset \new{c_{Y_{t+\eps}}} (\new{g|_{X_t}} (A))$. \new{Combining all these observations we have that $f|_{X_t}(c_{X_t}(A))\cup g|_{X_t}(c_{X_t}(A))\subset c_{Y_{t+\eps}}(f|_{X_t}(A))\cap c_{Y_{t+\eps}}(g|_{X_t}(A))$. Therefore, by}
   \cref{lem:J1times-homotopy}, $f|_{X_t}$ and $g|_{X_t}$ are one-step $(J_{\top},\times)$-homotopic.  
\end{proof}



\begin{definition}
    Let $(X_{\bullet},\new{c_{X_{\bullet}}}), (Y_{\bullet},\new{c_{Y_{\bullet}}}) \in \cat{F_{\new{\R}}^{\new{\subset}} Cl}$.
    Say that $(X_{\bullet},\new{c_{X_{\bullet}}})$ and $(Y_{\bullet},\new{c_{Y_{\bullet}}})$ are
    \emph{$\eps$ one-step $(J_{\top},\times)$-homotopy interleaved} if there exist
  $f: (X_{\bullet},\new{c_{X_{\bullet}}}) \to (Y_{\bullet},\new{c_{Y_{\bullet + \eps}}})$ and
$g: (Y_{\bullet},\new{c_{Y_{\bullet}}}) \to (X_{\bullet},\new{c_{X_{\bullet + \eps}}})$ such that
for all $t$,
$g|_{Y_{t+\eps}} \circ f|_{X_t}$ is one-step $(J_{\top},\times)$ homotopic \new{to the map} $X_t \incl X_{t+2\eps}$ and
$f|_{X_{t+\eps}} \circ g|_{Y_t}$ is one-step $(J_{\top},\times)$ homotopic \new{to the map} $Y_t \incl Y_{t+2\eps}$.
\end{definition}

\new{
\begin{definition}
    Let $J$ be an interval for the product operation $\otimes$.
    Let $\cat{D}$ be a category.
    Say that a functor $H: \cat{Cl} \to \cat{D}$ is \emph{$(J,\otimes)$-homotopy invariant} if 
    $f \sim_{(J,\otimes)} g$ implies that $H(f) = H(g)$.
\end{definition}
}

\new{By functoriality, if $X,Y \in \cat{Cl}$ are $(J,\otimes)$-homotopy equivalent and $H$ is a $(J,\otimes)$-homotopy invariant functor, then $H(X) \isom H(Y)$.}

\new{The following is the main result of this section.}

\new{
\begin{theorem}[Sharp Stability Theorem]
    \label{thm:stability-new}
     Let $(X_{\bullet},c_{X_{\bullet}}), (Y_{\bullet},c_{Y_{\bullet}}) \in \cat{F_{\R}^{\subset} Cl}$ and let $\eps \geq 0$. 
     Suppose there is an $\eps$-correspondence $C: (X_{\bullet},c_{X_{\bullet}}) \rightrightarrows (Y_{\bullet},c_{Y_{\bullet}})$. 
    Then $(X_{\bullet},c_{X_{\bullet}})$ and $(Y_{\bullet},c_{Y_{\bullet}})$ are $\eps$ one-step $(J_{\top},\times)$-homotopy interleaved.
    Furthermore, if $H:\cat{Cl} \to \cat{D}$ is a $(J_{\top},\times)$-homotopy invariant functor, 
    then the persistence modules  
    $H(X_{\bullet},c_{X_{\bullet}})$ and
    $H(Y_{\bullet},c_{Y_{\bullet}})$
    are $\eps$-interleaved.    
\end{theorem}
}

\begin{proof}
  Let  $C: (X_{\bullet},\new{c_{X_{\bullet}}}) \rightrightarrows (Y_{\bullet},\new{c_{Y_{\bullet}}})$ be \new{an} $\eps$-correspondence.
  By \cref{lem:graph-homomorphism} there exist $f:X \xto{C} Y$ and $g:Y \xto{C^T} X$ such that
  $f:(X_{\bullet},\new{c_{X_{\bullet}}}) \to (Y_{\bullet + \eps}, \new{c_{Y_{\bullet + \eps}}})$ and
  $g:(Y_{\bullet},\new{c_{Y_{\bullet}}}) \to (X_{\bullet + \eps}, \new{c_{X_{\bullet + \eps}}})$.
  Since the composition $g \circ f: (X_{\bullet},\new{c_{X_{\bullet}}}) \to (X_{\bullet+2\eps},\new{c_{X_{\bullet+2\eps}}})$ given by $(g \circ f)_t = g_{t+\eps} \circ f_t$ and the inclusion map $(X_{\bullet},\new{c_{X_{\bullet}}}) \incl (X_{\bullet+2\eps},\new{c_{X_{\bullet+2\eps}}})$ are both subordinate to $C^T \circ C$ 
  and similarly the composition $f \circ g: (Y_{\bullet},\new{c_{Y_{\bullet}}}) \to (Y_{\bullet+2\eps},\new{c_{Y_{\bullet+2\eps}}})$ given by $(f \circ g)_t = f_{t+\eps} \circ g_t$ and the inclusion map $(Y_{\bullet},\new{c_{Y_{\bullet}}}) \incl (Y_{\bullet+2\eps},\new{c_{Y_{\bullet+2\eps}}})$ are both subordinate to $C \circ C^T$\new{. By} \cref{prop:one-step-J1-homotopic} we have that
  $(X_{\bullet},\new{c_{X_{\bullet}}})$ and $(Y_{\bullet},\new{c_{Y_{\bullet}}})$ are $\eps$ one-step $(J_{\top},\times)$-homotopy interleaved.  

  \new{Applying a $(J_{\top},\times)$-homotopy invariant functor $H$, $H(f)$ and $H(g)$ give an $\eps$-interleaving of the persistence modules $H(X_{\bullet},c_{X_{\bullet}})$ and $H(Y_{\bullet},c_{Y_{\bullet}})$.}
\end{proof}

\new{Applying the algebraic stability theorem for q-tame persistence modules~\cite{Bauer:2015,cdsgo:book}, we have the following.}

\new{
\begin{corollary}[Matching Theorem] 
\label{prop:stability2}
  Let $(X_{\bullet},c_{X_{\bullet}}), (Y_{\bullet},c_{Y_{\bullet}}) \in \cat{F_{\R}^{\subset} Cl}$ and let $\eps \geq 0$.
  Suppose that for every $\delta > 0$, there is an $(\eps + \delta)$-correspondence from $(X_{\bullet},c_{X_{\bullet}})$ to $(Y_{\bullet},c_{Y_{\bullet}})$.
  Let $H:\cat{Cl} \to \cat{Vect}$ be a $(J_{\top},\times)$-homotopy invariant functor
  and suppose that $H(X_{\bullet},c_{X_{\bullet}})$ and $H(Y_{\bullet},c_{Y_{\bullet}})$ are
  q-tame.
  Then there exists an $\eps$-matching between
  $D(H(X_{\bullet},c_{X_{\bullet}}))$ and $D(H(Y_{\bullet},c_{Y_{\bullet}}))$,
  where $D(M)$ denotes the persistence diagram of the q-tame persistence module $M$~\cite{Bauer:2015,cdsgo:book}.
\end{corollary}
}

\new{By \cref{cor:homotopy-poset}, if a functor $H: \cat{Cl} \to \cat{D}$ is homotopy invariant for one of $(J_+,\times)$, $(J_{\top},\boxplus)$, $(I_{\tau},\times)$, $(J_+,\boxplus)$, or $(I_{\tau},\boxplus)$, then $H$ is also $(J_{\top},\times)$-homotopy invariant.
Therefore we can apply \cref{thm:stability-new} and
\cref{prop:stability2}.}

\new{
\begin{example} \label{rem:homotopy-invariant-functor}
 Let $H$ denote one of our cubical or simplicial singular homology theories and let $j \geq 0$.
 If there is an $\eps$-correspondence between the $\R$-filtered closure spaces $(X_{\bullet},c_{X_{\bullet}})$ and $(Y_{\bullet},c_{Y_{\bullet}})$ then 
 by \cref{thm:stability-new}, 
 the persistence modules $H_j(X_{\bullet},c_{X_{\bullet}})$ and $H_j(Y_{\bullet},c_{Y_{\bullet}})$ are $\eps$-interleaved.
If in addition the coefficients are in a field and the persistence modules are q-tame, then by \cref{prop:stability2}, there is an $\eps$-matching between their persistence diagrams.
\end{example}
}



As a direct consequence of \cref{thm:stability-new} we have the following.

\begin{theorem}[Stability Theorem] \label{thm:stability}
  Let $(X_{\bullet},\new{c_{X_{\bullet}}}), (Y_{\bullet},\new{c_{Y_{\bullet}}}) \in \cat{F_{\new{\R}} Cl}$.
  \new{Let $H:\cat{Cl} \to \cat{D}$ be a $(J_{\top},\times)$-homotopy invariant functor.}
Then
  \begin{equation*}
    d_{I}(H(X_{\bullet},\new{c_{X_{\bullet}}}), H(Y_{\bullet},\new{c_{Y_{\bullet}}})) \leq 2 d_{GH}((X_{\bullet},\new{c_{X_{\bullet}}}), (Y_{\bullet},\new{c_{Y_{\bullet}}}) ),
  \end{equation*}
where $d_I$ denotes the interleaving distance.
\end{theorem}

As a special case, \cref{thm:stability} strengthens \cref{thm:sublevelset-stability}, by \cref{prop:gh-dinfty}.
As a direct consequence of \cref{prop:stability2} we have the following.

\begin{corollary}[Bottleneck Stability Theorem] \label{thm:stability-bottleneck}
  Let $(X_{\bullet},\new{c_{X_{\bullet}}}), (Y_{\bullet},\new{c_{Y_{\bullet}}}) \in \cat{F_{\new{\R}} Cl}$.
  \new{Let $H:\cat{Cl} \to \cat{Vect}$ be a $(J_{\top},\times)$-homotopy invariant functor.}
If $H(X_{\bullet},\new{c_{X_{\bullet}}})$ and $H(Y_{\bullet},\new{c_{Y_{\bullet}}})$ are
q-tame 
  then 
  \begin{equation*}
    d_{B}(D(H(X_{\bullet},\new{c_{X_{\bullet}}})), D(H(Y_{\bullet},\new{c_{Y_{\bullet}}}))) \leq 2 d_{GH}((X_{\bullet},\new{c_{X_{\bullet}}}), (Y_{\bullet},\new{c_{Y_{\bullet}}}) ),
  \end{equation*}
  where $D(M)$ denotes the persistence diagram of the q-tame persistence module $M$ and $d_B$ denotes the bottleneck distance.
\end{corollary}

\new{
\begin{example}
    By \cref{rem:homotopy-invariant-functor}, \cref{thm:stability}   applies to our  cubical and simplicial singular homology functors, and if we use coefficients in a field then we may apply \cref{thm:stability-bottleneck}.
\end{example}
}

\begin{definition}
  Let $(X_{\bullet},E_{\bullet}), (Y_{\bullet},F_{\bullet}) \in \cat{F_{\new{\R}}^{\new{\subset}} Simp}$.
  Say that $(X_{\bullet},E_{\bullet})$ and $(Y_{\bullet},F_{\bullet})$ are
  \emph{$\eps$-contiguity interleaved} if there exist
  $f: (X_{\bullet},E_{\bullet}) \to (Y_{\bullet},F_{\bullet + \eps})$ and
$g: (Y_{\bullet},F_{\bullet}) \to (X_{\bullet},E_{\bullet + \eps})$ such that
for all $t$,
$g|_{Y_{t+\eps}} \circ f|_{X_t}$ is contiguous \new{to} $X_t \incl X_{t+2\eps}$ and
$f|_{X_{t+\eps}} \circ g|_{Y_t}$ is contiguous \new{to} $Y_t \incl Y_{t+2\eps}$.
\end{definition}

\new{By \cref{proposition:homotopy_and_contiguity},
we have the following.}

\new{
\begin{proposition} \label{prop:stability-new}
    Let $S \in \{\VR,\Cech,\Cech^{\transpose}\}$.
    If $(X_{\bullet},c_{X_{\bullet}}), (Y_{\bullet},c_{Y_{\bullet}}) \in \cat{F_{\R}^{\subset} Cl}$
    are $\eps$-one-step $(J_{\top},\times)$-homotopy interleaved then
    $(X_{\bullet},S(c_{X_{\bullet}})), (Y_{\bullet},S(c_{Y_{\bullet}})) \in \cat{F_{\R}^{\subset} Simp}$
    are $\eps$-contiguity interleaved.
    Conversely, if
    $(X_{\bullet},E_{\bullet}), (Y_{\bullet},F_{\bullet}) \in \cat{F_{\R}^{\subset} Simp}$
    are $\eps$-contiguity interleaved then $(X_{\bullet},\St(E_{\bullet})), (Y_{\bullet},\St(F_{\bullet})) \in \cat{F_{\R}^{\subset} Cl}$
    are $\eps$-one-step $(J_{\top},\times)$-homotopy interleaved.
\end{proposition}
}

\new{
\begin{example} \label{ex:simplicial-homology-functor}
    Let $S \in \{\VR,\Cech,\Cech^{\transpose}\}$.
    Let $H$ denote simplicial homology and let $j \geq 0$.
    If there is an $\eps$-correspondence between $(X_{\bullet},c_{X_{\bullet}})$ and $(Y_{\bullet},c_{Y_{\bullet}})$ then 
    by \cref{prop:stability-new} and \cref{thm:stability-new},
    the persistence modules $H_j(X_{\bullet},S(c_{X_{\bullet}}))$ and $H_j(Y_{\bullet},S(c_{Y_{\bullet}}))$
    are $\eps$-interleaved.
    If in addition the coefficients are in a field and the persistence modules are q-tame, then by \cref{prop:stability2}, there is an $\eps$-matching between their persistence diagrams.
\end{example}
}

\new{
By \cref{ex:simplicial-homology-functor}, we may apply \cref{thm:stability} and \cref{thm:stability-bottleneck} to obtain the following.
}
 

\begin{theorem}[Rips and \v{C}ech Stability Theorem]  \label{thm:stability-rips}
  Let $(X_{\bullet},\new{c_{X_{\bullet}}}), (Y_{\bullet},\new{c_{Y_{\bullet}}}) \in \cat{F_{\new{\R}} Cl}$.
  Let \new{$S \in \{\VR, \Cech, \Cech^{\transpose}\}$}.
  Let $H$ denote simplicial homology and let $j \geq 0$. Then
  \begin{multline*}
    d_{I}(H_j(X_{\bullet},S(\new{c_{X_{\bullet}}})), H_j(Y_{\bullet},S(\new{c_{Y_{\bullet}}})))
    = d_{I}(H_j(X_{\bullet},S(A(\new{c_{X_{\bullet}}}))), H_j(Y_{\bullet},S(A(\new{c_{Y_{\bullet}}}))))\\
    \leq 2 d_{GH} ((X_{\bullet},S(A(\new{c_{X_{\bullet}}}))), (Y_{\bullet},S(A(\new{c_{Y_{\bullet}}}))))
     \leq 2 d_{GH} ((X_{\bullet},S(\new{c_{X_{\bullet}}})), (Y_{\bullet},S(\new{c_{Y_{\bullet}}}))).
  \end{multline*}
  If coefficients are in a field and the persistence modules are 
  q-tame then 
  \begin{gather*}
    d_{B}(D(H_j(X_{\bullet},S(\new{c_{X_{\bullet}}}))), D(H_j(Y_{\bullet},S(\new{c_{Y_{\bullet}}}))))
    \leq 2 d_{GH} ((X_{\bullet},S(A(\new{c_{X_{\bullet}}}))), (Y_{\bullet},S(A(\new{c_{Y_{\bullet}}})))).
    \end{gather*}
  \end{theorem}

\subsection*{Acknowledgments}

This research was partially supported by the Southeast Center for Mathematics and Biology, an NSF-Simons Research Center for Mathematics of Complex Biological Systems, under National Science Foundation Grant No. DMS- 1764406 and Simons Foundation Grant No. 594594.
This material is based upon work supported by, or in part by, the Army Research Laboratory and the Army Research Office under contract/grant number W911NF-18-1-0307.
\new{Most of this work was part of the second author's PhD dissertation under the direction of the first author in the Department of Mathematics at the University of Florida~\cite{Milicevic:thesis}.}
\new{The first author thanks the Mathematical Institute at the University of Oxford for hosting him while he was editing this manuscript.}
We thank Jamie Scott for suggesting the use of concatenation, 
\new{and we thank the referee for their many helpful comments which improved our paper. Finally, we thank Antonio Rieser for helping us strengthen the presentation of our stability results.}

\vspace{1em}

On behalf of all authors, the corresponding author states that there is no conflict of interest.

\printbibliography
\end{document}